\documentclass[11pt]{article}
\usepackage[utf8]{inputenc}
\usepackage[sc]{mathpazo}
\usepackage{fullpage}
\usepackage{amsfonts}
\usepackage{amsthm}
\usepackage{amsmath}
\usepackage{xcolor}
\usepackage{amssymb}
\usepackage{hyperref}
\usepackage{fullpage}
\usepackage{algorithm}
\usepackage{algorithmic}
\usepackage[normalem]{ulem}
\usepackage{verbatim}
\usepackage{thmtools}
\usepackage{thm-restate}
\usepackage{caption}

\usepackage[style=alphabetic,backend=bibtex,maxbibnames=20,maxcitenames=6,giveninits=true,doi=false,url=false]{biblatex}
\newcommand*{\citet}[1]{\AtNextCite{\AtEachCitekey{\defcounter{maxnames}{2}}} \textcite{#1}}

\newcommand*{\citep}[1]{\cite{#1}}

\bibliography{multi-params}

\theoremstyle{plain}
\newtheorem{theorem}{Theorem}
\newtheorem*{theorem*}{Theorem}
\newtheorem{lemma}{Lemma}
\newtheorem{definition}{Definition}

\newtheorem{corollary}{Corollary}
\newtheorem{remark}{Remark}

\newtheorem{claim}{Claim}
\newtheorem{informallemma}{Informal Lemma}

\newcommand{\J}{\mathcal{J}}
\newcommand{\R}{\mathbb{R}}
\newcommand{\E}{\mathbb{E}}
\newcommand{\calX}{\mathcal{X}}

\renewcommand{\phi}{\varphi}
\renewcommand{\tilde}{\widetilde}

\newcommand{\poly}{\mathrm{poly}}

\newcommand{\yuval}[1]{\textcolor{red}{ Yuval: {#1}}}

\newcommand{\vardis}[1]{\textcolor{blue}{ Vardis: {#1}}}

\newlength{\commentindent}
\makeatletter
\renewcommand{\algorithmiccomment}[1]{\unskip\hfill\makebox[\commentindent][l]{//~#1}\par}
\LetLtxMacro{\oldalgorithmic}{\algorithmic}
\renewcommand{\algorithmic}[1][0]{%
  \oldalgorithmic[#1]%
  \renewcommand{\ALC@com}[1]{%
    \ifnum\pdfstrcmp{##1}{default}=0\else\algorithmiccomment{##1}\fi}%
}
\makeatother

\newcommand{\lp}{\left}
\newcommand{\rp}{\right}
\DeclareMathOperator{\sech}{sech}

\DeclareMathOperator{\Var}{Var}
\DeclareMathOperator{\argmax}{argmax}

\newcommand{\norm}[1]{\left\lVert #1 \right\rVert}

\title{Learning Ising models from one or multiple samples}
\author{
    Yuval Dagan 
    \\EECS \& CSAIL, MIT\\
    \tt{dagan@mit.edu}
    \and
	Constantinos Daskalakis
	\\EECS \& CSAIL, MIT\\
	\tt{costis@csail.mit.edu}
	\and
	Nishanth Dikkala
	\\GOOGLE RESEARCH \& MIT\\
	\tt{nishanthd@google.com}\\
	\and
	Anthimos Vardis Kandiros
	\\EECS \& CSAIL, MIT\\
	\tt{kandiros@mit.edu}
	}
\date{\vspace{-5ex}}

\begin{document}

\maketitle
\thispagestyle{empty}

% !TEX root = ./main.tex

\begin{abstract}
\iffalse
\footnote{\yuval{Alternative titles:\\
	Learning Ising models from one or multiple samples: a unified framework.\\
	A unified framework to learn Ising models from one or multiple samples
}}
\fi

There have been two separate lines of work on estimating Ising models: (1) estimating them from multiple independent samples under minimal assumptions about the model's interaction matrix \cite{bresler2015efficiently,vuffray2016interaction,klivans2017learning,hamilton2017information,wu2019sparse}; and (2) estimating them from one sample in restrictive settings \cite{chatterjee2007estimation,bhattacharya2018inference,ghosal2018joint,daskalakis2019regression}. We propose a unified framework that smoothly interpolates between these two settings, enabling significantly richer estimation guarantees from one, a few, or many samples.

Our main theorem provides guarantees for one-sample estimation, quantifying the estimation error in terms of the metric entropy of a family of interaction matrices. As corollaries of our main theorem, we derive bounds when the model's interaction matrix is a (sparse) linear combination of known matrices, or it belongs to a finite set, or to a high-dimensional manifold. In fact, our main result handles multiple independent samples by viewing them as one sample from a larger model, and can be used to derive estimation bounds that are qualitatively similar to those obtained in the afore-described multiple-sample literature.
Our technical approach benefits from sparsifying a model's interaction network, conditioning on subsets of variables that make the dependencies in the resulting conditional distribution sufficiently weak.
We use this sparsification technique to prove strong concentration and anti-concentration results for the Ising model,
which we believe have applications beyond the scope of this paper.

\end{abstract}

\newpage
\setcounter{page}{1}

\section{Introduction}

%\yuval{Current structure: (most parts should be shortened)\\
%Motivation and problem definition, first 3 pages \\
%Dobrushin's condition \\
%Results \\
%Prior work (should be revised) \\
%Techniques
%}

%\yuval{Meetings' comments:
%Convince that we are overcoming technical challenges.\\
%}

%----------------- PROSE ----------------------
% stats from dependent data motivation

Markov Random Fields (MRFs) are a popular framework for representing high-dimensional distributions with conditional independence structure, represented via an undirected graph~\cite{Lauritzen96,WainwrightJ08}. The explicit representation of conditional independences allows for a more succinct representation of a distribution, decreasing the computational requirements to do inference. A special case of MRFs studied in this paper is the celebrated {\em Ising model}~\cite{ising1925beitrag}, which samples a binary vector, $x = (x_1,\ldots,x_n) \in \{\pm 1 \}^n$, according to a measure of the following form:
\begin{align}
\label{eq:ising}
    &\Pr_{J^*}[x] = \exp\left(x^{\top}J^*x/2 - F(J^*) - n\log 2 \right),
    % \\
    % &~\text{where}~ 
    % F(J^*) = \log \left( 2^{-n}\sum_{x \in \{\pm 1 \}^n} \exp(x^{\top}J^*x/2 )\right). \notag
\end{align}
where $J^*$ is an $n \times n$ symmetric matrix with zero diagonal and $F(J^*) = \log \left( 2^{-n}\sum_{x} \exp(x^{\top}J^*x/2 )\right)$ is the so-called log-partition function. Notice that the term $J^*_{ij} x_i x_j$ in the exponent of the density encourages $x_i$ and $x_j$ to have equal or opposite signs depending on the sign and magnitude of $J^*_{ij}$, but this ``local encouragement'' can be overwritten by indirect interactions arising through paths between $i$ and $j$ in the undirected graph defined by the non-zero entries of $J^*$. Whenever $i$ and $j$ are disconnected in this graph, $x_i$ and $x_j$ are independent.

Since its introduction, the Ising model has found profound applications in a range of disciplines, including Statistical Physics, Computer Vision, Computational Biology, and the Social Sciences; see e.g.~\cite{GemanG86,Ellison93,Felsenstein04,chatterjee2005concentration,DaskalakisMR11,daskalakis2017concentration}. 
%For example, it has been used to model binary behavior on a social network~\cite{Ellison93,young2006diffusion,montanari2010spread}, wherein nodes choose between, say, whether to own an Android phone ($+1$) or an iPhone ($-1$), and $\sum_j J^*_{ij}x_ix_j$ captures the utility derived by node $i$ depending on her and the other nodes' selection of phones. In this setting, it is easy to show that, when nodes update their strategies using logit best response dynamics, the stationary distribution of their joint behavior converges to~\eqref{eq:ising}.
These applications have motivated a long line of research aiming at estimating Ising models using samples. Some exciting progress on this front has appeared in recent years, including~\cite{santhanam2012information,ravikumar2010high,bresler2015efficiently,vuffray2016interaction,klivans2017learning,hamilton2017information,wu2019sparse}. Importantly, most prior work assumes access to {\em multiple independent samples}, targeting estimating the interaction matrix $J^*$ of a model under some conditions on $J^*$.
% , which typically amount to placing a lower bound on the temperature of the model, i.e.~an upper bound on~$||J^*||_{\infty}$. 
Instead our focus in this work is to estimate Ising models from a {\em single} sample, which as we will shortly explain is  a {\em more general problem}:
% is a more gener \costis{add emphasis that this is more general than multiple samples}: 

\vspace{5pt}\noindent ~~\framebox{\begin{minipage}[h]{16cm} {\bf Single-Sample Ising Model Estimation:} {Given a family of interaction matrices ${\cal J} \subseteq \mathbb{R}^{n\times n}$ and one sample $X$ from \eqref{eq:ising}, where $J^* \in {\cal J}$, compute an estimate $\hat{J}(X)$ to minimize $\|\hat{J}(X)-J^* \|_F$.}
\end{minipage}
}

\medskip Notice that estimating Ising models from one sample generalizes estimating them from multiple samples. This is because $\ell$ independent samples from an $n$-node Ising model with interaction matrix $J^*$ can be viewed as one sample from an Ising model with $n\ell$ nodes, which belong to $\ell$ disconnected subnetworks that each have interaction matrix $J^*$.

% Indeed, given an Ising model with $n$ nodes and interaction matrix $J^*$, we may consider an Ising model with $\ell \times n$ nodes, which belong to $\ell$ disconnected networks of $n$ nodes, each subnetwork having interaction matrix $J^*$. Obviously, one sample from the big Ising model gives us $\ell$ independent samples from the small Ising model.

Moreover, single-sample estimation is motivated by many applications where we may realistically only collect a single independent sample from a distribution. E.g., in applications of the Ising model in social network analysis, a sample from the model represents some binary behavior of the nodes in a social network, such as using an Android phone or an iPhone, voting for Democrats or Republicans, etc. In such applications, if we take a snapshot of the nodes' behaviors tomorrow, chances are that very little would  change compared to their behavior today, and we certainly would not collect an independent sample. More broadly, lack of access to independent samples is ubiquitous in financial, meteorological, and geographical data, as well as social-network data~\cite{manski1993identification,bramoulle2009identification}, where it has been studied in topics as diverse as criminal activity~\cite{glaeser1996crime}, welfare participation~\cite{bertrand2000network}, school achievement~\cite{sacerdote2001peer},  retirement plan participation~\cite{duflo2003role}, and obesity~\cite{christakis2013social,trogdon2008peer}. Moreover, it has motivated a growing literature on single-sample statistical estimation, including~\cite{besag1974spatial,yu1994rates,chatterjee2007estimation,kontorovich2008concentration,berti2009rate,mohri2009rademacher,pestov2010predictive,mohri2010stability,london2013collective,kuznetsov2015learning,london2016stability,McDonald2017,kontorovich2017concentration,bhattacharya2018inference,ghosal2018joint,bresler2018optimal,daskalakis2019regression,dagan2019learning,valiant2019worstcase}. 
%Our goal is to advance the algorithmic and probabilistic foundations of this important and emerging field.

Of course, one sample from~\eqref{eq:ising} only carries $n$ bits, while the matrix $J^*$ to be estimated has $\Omega(n^2)$ real entries. Thus, one cannot hope to estimate $J^*$ well from one sample without placing  constraints on $J^*$. Said differently, the error in estimating $J^*$ from one sample should depend on how complex $J^*$ might be. This is the role played by $\cal J$ in the definition of our estimation problem. Our main result, presented shortly as Theorem~\ref{thm:informal-main}, is that there exists an estimator whose error depends on the metric entropy of $\cal J$. Instantiating $\cal J$ in different ways, we obtain strong estimation guarantees when: (i) $\cal J$ is finite; (ii) it contains linear combinations of known matrices; (iii) it contains sparse linear combinations of known matrices; and (iv) it is a high-dimensional manifold. These are respectively Corollaries~\ref{cor:finite}, \ref{cor:linear}, \ref{cor:sparse}, and \ref{cor:manifold}. 

% \yuval{modified this a bit according to our discussion}
Prior to our work, the single-sample Ising model estimation literature had only studied quite restrictive special cases of our problem, namely the case $J^*=\beta J$, where $J$ is a known matrix, and $\beta$ is an unknown scalar strength parameter~\cite{chatterjee2007estimation,bhattacharya2018inference}, or slightly more general cases studied by follow-up work~\cite{bhattacharya2018inference,ghosal2018joint,daskalakis2019regression}.
Restricted to this special case, our bounds provide quantitative improvements in the estimation error, as discussed in Section~\ref{sec:intro-one-param}. However, our general theorem, as well as its corollaries in Settings (i)--(iv) discussed in the previous paragraph, provide vast extensions. E.g.~(ii) and (iii) capture settings wherein we might know various social networks that individuals belong to (Facebook, LinkedIn, etc.) and expect that these all contribute to their behavior at different strengths. Setting (iv) captures settings of interest to Spatial Econometrics~\cite{anselin2001spatial,lesage2008introduction,anselin2012new,anselin2013spatial} wherein we might be able to postulate a functional form for the interaction matrix and might be interested in estimating its parameters. 
% \costis{add a bit more on this}

On the other hand, multiple-sample Ising model estimation is a widely studied problem with a long literature, going back to at least 
% As we have already discussed, another well studied line of work concerns estimation of Ising models from multiple independent samples. 
% While algorithms for learning Ising models go b \cite{chow1968approximating}, 
% While this literature goes back to at least
~\cite{chow1968approximating}. Yet, an efficient algorithm that learns Ising models on general (bounded-degree) graphs was only recently given in breakthrough work by~\cite{bresler2015efficiently}, which has incited a renaissance of work on this topic~\cite{vuffray2016interaction,klivans2017learning,hamilton2017information,wu2019sparse}.
%In fact, estimating Ising models from one sample generalizes estimating them from multiple  samples. Indeed, given an Ising model with $n$ nodes and interaction matrix $J^*$, we may consider an Ising model with $\ell \times n$ nodes, which belong to $\ell$ disconnected networks of $n$ nodes, each subnetwork having interaction matrix $J^*$. Obviously, one sample from the big Ising model gives us $\ell$ independent samples from the small Ising model.
Since single-sample estimation generalizes multiple-sample estimation, as we have already discussed, our results for single-sample estimation allow us to obtain reconstruction guarantees for the following problem for any value of $\ell$:

\vspace{5pt}\noindent ~~\framebox{\begin{minipage}[h]{16cm} {\bf $\ell$-Sample Ising Model Estimation:} {Given a family of interaction matrices ${\cal J} \subseteq \mathbb{R}^{n\times n}$ and $\ell$ independent samples from \eqref{eq:ising}, where $J^* \in {\cal J}$, compute an estimate $\hat{J}$ to minimize $\|\hat{J}-J^* \|_F$.}
\end{minipage}
}

\medskip \noindent Corollary~\ref{cor:multiple-samples} of Theorem~\ref{thm:informal-main} quantifies that access to multiple samples typically decreases the reconstruction error by a factor of $\widetilde{\Omega}(\sqrt{\ell})$. As such, we get reconstruction guarantees which smoothly interpolate between the single-sample estimation setting considered by~\cite{chatterjee2007estimation,bhattacharya2018inference,ghosal2018joint,daskalakis2019regression} and the more common $\omega(1)$-sample estimation setting considered by~\cite{bresler2015efficiently,vuffray2016interaction,klivans2017learning,hamilton2017information,wu2019sparse}. Interestingly, instantiating our result to the latter setting we obtain guarantees which are competitive to that work, as shown in Corollary~\ref{cor:learn-no-assuption} and the middle row of Table~\ref{tbl:results}. 
Our sample complexity is typically higher, 
%Indeed, without any assumption on $J^*$, besides the standard assumption that $||J^*||_{\infty} = O(1)$ \costis{(which is made by all these works to avoid exponential sample complexity)}, we show that ${O}(n^2 \log n/\varepsilon^2)$ samples suffice to learn $J^*$ to within Frobenius norm $\varepsilon$. This is off by a factor of $O(n)$ compared to these works,
yet we derive it as a corollary of our main theorem which does not utilize independence between the samples. {This further enables us to obtain similar bounds given two or more \emph{dependent} samples as demonstrated by Corollary~\ref{cor:weakly-dep}. See Table~\ref{tbl:results} for a summary of our results together with a comparison to prior work on estimation from a single and multiple samples.} 

 \begin{table}
 \def\arraystretch{1}
\begin{tabular}{ |c| c| c| }
\hline
 \bf{Family $\J$ of matrices} & \bf{Single sample} & $\ell$ \bf{samples} \\
 \hline
 Arbitrary family $\J$  *
 & $\sqrt{f(\J,1/n)}$ (Theorem~\ref{thm:informal-main}) & $\sqrt{f(\J,1/n\ell) \over \ell}$ (Corollary~\ref{cor:multiple-samples}) \\[1ex]
 Finite $\J$ & $\sqrt{\log |\J|}$ (Corollary~\ref{cor:finite}) & $\sqrt{\log |\J|/\ell}$ (Cor \ref{cor:finite} \& \ref{cor:multiple-samples})
 \\[1ex]
 Linear combination of $k$ known matrices &  $\sqrt{k}$ (Corollary~\ref{cor:linear}) & $\sqrt{k/\ell}$ (Cor \ref{cor:linear} \& \ref{cor:multiple-samples})
 \\[1ex]
 \shortstack{$s$-sparse linear combination\\ of $k$ known matrices} & $\sqrt{s\log k}$ (Corollary~\ref{cor:sparse}) & $\sqrt{s\log k/\ell}$ (Cor \ref{cor:sparse} \& \ref{cor:multiple-samples})
 \\
 \hline
 All matrices (unconstrained) & impossible & \shortstack{$n\sqrt{\log(n\ell)/\ell}$ (Corollary~\ref{cor:learn-no-assuption})\\ $\sqrt{n}(\log n/\ell)^{1/4}$ \cite{klivans2017learning}} \\
 \hline
 Scalar multiples of a known matrix ** & \shortstack{$1/\sqrt{F(J^*)}$ (Corollary \ref{cor:one-param-informal})\\ $1/\sqrt{F(J^*)}$ (under additional\\ assumptions) \cite{chatterjee2007estimation,bhattacharya2018inference}} & \shortstack{$1/\sqrt{\ell F(J^*)}$\\(follows from our \\one-sample result)} \\
 \hline
\end{tabular}
\caption{We state the estimation error $||\hat{J}-J^*||_F$ obtained by our work and prior work in different settings, ignoring some logarithmic factors. We present bounds under the standard assumption that $\|J^*\|_{\infty}$ is bounded by some constant $M$. Under this assumption, since $\|J\|_F \le \sqrt{n}\|J\|_\infty \le M\sqrt{n}$, a rate smaller than $M\sqrt{n}$ is non-trivial ; see Definition~\ref{def:covering}/Theorem~\ref{thm:informal-main}.\\
* In the first row, $f(\J,\epsilon)=\log N(\mathcal{J}, \|\cdot\|_2, \epsilon)$ is the metric entropy of family $\cal J$ under $\|\cdot\|_2$.\\
** In the last row, we consider the setting $\J = \{\beta J \colon |\beta| = O(1)\}$, where $J$ is fixed and the estimation error is with respect to the parameter $\beta$. Here, $F(J^*)$ is the log-partition function, defined earlier.}  \label{tbl:results}
\end{table}

\section{Our Results} \label{sec:intro-results}
%\yuval{removed overview to before this section.}

\subsection{A general upper bound}

In this section, we present a general upper bound that is a function of the covering numbers of the set $\mathcal{J}$, which represents the smallest number of elements from $\mathcal{J}$ that can approximate all elements of  set $\mathcal{J}$. We begin with a definition.

\begin{definition}\label{def:covering}
	Given a normed space $(\mathcal{X},\|\cdot\|)$, a set $\mathcal{V}\subseteq \calX$ and $\epsilon > 0$, we say that a set $\mathcal{N} \subseteq \mathcal{V}$ is an $\epsilon$-cover of $\mathcal{V}$ if for any $v \in \mathcal{V}$ there exists $u \in \mathcal{N}$ such that $\|u-v\|\le \epsilon$.
	The $\epsilon$-covering number of $\mathcal{V}$ with respect to the norm $\|\cdot\|$, denoted by $N(\mathcal{V}, \|\cdot\|,\epsilon)$, is the minimum cardinality of an $\epsilon$-cover.
\end{definition}

Our main result is stated below. As is standard in prior work, we parametrize our error in terms of a bound $M$ on the infinity norm of the interaction matrices, $\|J\|_\infty = \max_i \sum_j |J_{ij}|$, which is called ``width'' in~\cite{klivans2017learning} and relaxes placing a bound on the maximum degree~\cite{bresler2015efficiently,vuffray2016interaction}. As shown in prior work~\cite{santhanam2012information}, our single exponential dependence on $M$ is necessary.\footnote{While \cite{santhanam2012information} provide a lower bound for multiple-sample estimation, their lower bound applies to our case as well because as we have explained single-sample estimation is more general than multiple-sample.}

% and our estimation error has a (single) exponential dependence on $M$. This matches the dependence on $M$ obtained by prior work for multiple-sample estimation, and is tight as shown by~\cite{santhanam2012information}. This lower bound carries over to our setting as single-sample estimation is more general than multiple-sample estimation.

% It is easy to see (because ) that the (single) exponential dependence of the estimation error on $M$, shown in for multiple samples, carries over to our setting.

\begin{theorem}[Follows from Theorem~\ref{t:general}] \label{thm:informal-main}
Let $M>0$ and let $\mathcal{J} \subseteq \{J \colon \|J\|_\infty \le M \}$ denote a collection of interaction matrices.
There is an algorithm which, given a single sample $x \sim \Pr_{J^*}$ where $J^* \in \mathcal{J}$, outputs $\hat{J}$ such that with probability $\ge 1-\delta$:
\begin{align*}
\|\hat{J}-J^* \|_F 
%&\le \min \lp\{r \ge 0 \colon r \ge C\sqrt{\log N(\mathcal{J}, \|\cdot\|_2, cr/n) + \log (1/\delta)+\log\log n}\rp\}\\
&\le C(M) \sqrt{\log N(\mathcal{J}, \|\cdot\|_2, 1/n) + \log (1/\delta)+\log\log n},
\end{align*}
where $C(M)$ is an (single) exponential function of $M$ and $\|\cdot\|_2$ denotes the spectral norm on matrices. Moreover, $\hat{J}$ is the minimizer over ${\cal J}$ of a convex function on the space of matrices, $\mathbb{R}^{n\times n}$. It can be computed in polynomial time if ${\cal J}$ is convex and projection onto ${\cal J}$ is efficiently computable.
\end{theorem}

Theorem~\ref{thm:informal-main} guarantees that we can find a matrix $\hat{J}$ that is  close to the true interaction matrix $J^*$ in Frobenius norm. In this general formulation, the error depends on the covering numbers of the set $\mathcal{J}$. 
In many interesting scenarios, the $\epsilon$-cover of $\mathcal{J}$ will have size of the order of $(1/\epsilon)^k$, where $k$ is a notion of dimension that is specific to each case. By applying Theorem~\ref{thm:informal-main}, we obtain an error of the order of $\sqrt{k\log n}$ for constant $M$. If $k$ is significantly less than $n$, this is a non-trivial bound, since both matrices $\hat{J},J^*$
can have a Frobenius norm as high as $\Omega(\sqrt{n})$. 
We present examples where this is the case in the next section.

\begin{remark}[Tightness of the bound]
	It is reasonable to expect that Theorem~\ref{thm:informal-main} is not completely tight. Tight upper bounds based on covering numbers are usually proved via the technique of chaining. However, technical difficulties arise once one tries to apply it in our scenario. Still, in all examples presented in the next section, this technique could remove only logarithmic factors, as our near-tight lower bounds provided in Section~\ref{sec:lower} establish.
\end{remark}

% A final remark should be made about the computational cost of computing the estimator $\hat{J}$. If $\mathcal{J}$ is a convex set, then since $\hat{J}$ is the optimizer of a convex function over $\cal J$, it is usually possible to find $\hat{J}$ in polynomial time via optimization algorithms such as projected gradient descent.\footnote{In order to compute in polynomial time the optimum of a convex function over a convex set $\J$, one needs to be able to perform some operations in polynomial time, e.g.~to project any point into the set $\J$.} A detailed analysis in a special case of $\mathcal{J}$ appears in Section~\ref{sec:optimization} of the appendix.

\subsection{Applications of the upper bound}
To showcase the power of Theorem~\ref{thm:informal-main}, we now apply it to some concrete families $\mathcal{J}$. The families we consider capture both single-sample and multiple-sample Ising model estimation problems, in Sections~\ref{sec:single sample examples} and~\ref{sec:expose multiple samples} respectively.
% consist of both estimation from a single or multiple samples. 
In all cases, we parametrize our bounds in terms of a bound $M$ on the infinity norm of the matrices in $\J$ and a function $C(M)$ of which appears in our estimation error, as in Theorem~\ref{thm:informal-main}. 
The detailed statements and proofs are provided in Section~\ref{s:applications}.

\subsubsection{Estimation from a single sample} \label{sec:single sample examples}
The simplest case is when $\mathcal{J}$ is finite. Then, $N(\mathcal{J},\|\cdot\|_2,\epsilon) \le |\mathcal{J}|$ for all $\epsilon\ge 0$ and we have:
\begin{corollary}\label{cor:finite}
	If $\mathcal{J}$ is finite and all its elements $J$ satisfy $\|J\|_{\infty}\le M$, our estimator satisfies $\|\hat{J}-J^*\|_F \le C(M) \sqrt{\log |\mathcal{J}| + \log(1/\delta)+\log\log n}$, with probability $\ge 1-\delta$. Moreover, $\hat{J}$ can be computed in time $\poly(|\J|,n)$ (i.e. polynomial time in $|\J|$ and $n$).
\end{corollary}

Next, we consider settings where $J^*$ is a linear combination of $k$ known matrices, with unknown coefficients. 
\begin{restatable}{corollary}{linsubspace}\label{cor:linear}
	Let $J_1,\dots,J_k$ be fixed matrices and let $\J = \{ J = \sum_{i=1}^k \beta_i J_i \colon \|J\|_\infty \le M, \vec{\beta}\in\mathbb{R}^k \}$. Then, our estimator $\hat{J}$ satisfies
	$
	\|\hat{J}-J^*\|_F \le C(M)\sqrt{k\log n + \log(1/\delta)}$, with probability~$\ge 1-\delta$,
	and $\hat{J}$ can be computed in time $\poly(n,k)$.
\end{restatable}
This can be extended to when $J^*$ is a $s$-sparse linear combination of $k$ known matrices, which enables us to obtain a bound with only a logarithmic dependence on $k$. For any $\vec{\beta} \in \mathbb{R}^k$ denote by $\|\vec{\beta}\|_0$ the number of nonzero coordinates of $\vec{\beta}$. The result is given below.
\begin{restatable}
{corollary}{sparse}\label{cor:sparse}
	Let $J_1,\dots,J_k$ be fixed matrices, $s > 0$, and let $\J = \{ J = \sum_{i=1}^k \beta_i J_i \colon \|J\|_\infty \le M, \|\vec{\beta}\|_0 \le s \}$. Then, our estimator $\hat{J}$ satisfies
	$
	\|\hat{J}-J^*\|_F \le C(M)\sqrt{s(\log n+\log k) + \log(1/\delta)},
	$
	with probability~$\ge 1-\delta$, and $\hat{J}$ can be computed in time $\poly(n,s)\cdot \binom{k}{s}$.
\end{restatable}
%Here, the algorithm has to enumerate over all subsets of $\{J_1,\dots,J_k\}$ of size $s$ and this increases the runtime by a factor of $\binom{k}{s}$. This yields a polynomial time algorithm whenever $s = o(1)$. It is likely possible to remove the need for enumeration by a relaxation that considers all the linear combinations with $\|\vec{\beta}\|_1 \le O(s)$, however, this is not implied from Theorem~\ref{thm:informal-main}.

While Corollary~\ref{cor:linear} considers linear combinations of $k$ known matrices, one can also consider non-linear settings, where, in general, the matrices lie in a $k$-dimensional manifold. We consider manifolds that are images of Lipschitz functions from convex subsets of $\mathbb{R}^k$ to the set of matrices. For this class, the following bound can be derived (see Section~\ref{sec:manifolds} for a general argument):

\begin{restatable}{corollary}{manifold}\label{cor:manifold}
	Let $h(\vec\beta)$ be a function from $[-1,1]^k$ to the set of $n\times n$ matrices, that satisfies
	$\|h(\vec\beta)-h(\vec\beta')\|_2 \le L \|\vec\beta-\vec\beta'\|_\infty$ for some $L>0$. Define $\J = \{ J=h(\vec{\beta}) \colon \vec{\beta}\in[-1,1]^k,\|J\|_\infty \le M \}$. Then our estimator $\hat{J}$ satisfies
	$
	\|\hat{J}-J^*\|_F \le C(M)\sqrt{k(\log n+\log L) + \log(1/\delta)},
	$
	with probability $\ge 1-\delta$.
\end{restatable}

\subsubsection{Estimation from several samples} \label{sec:expose multiple samples}

When we are given access to several independent or dependent samples, we can utilize them to obtain stronger guarantees. This is done via a reduction to the single-sample setting. As a first example, assume that $\ell$ independent samples from an $n$-dimensional Ising model are obtained. Notice that these can be viewed as a single sample from an $n\ell$ dimensional model. Thus, an application of  Theorem~\ref{thm:informal-main} results in a gain of approximately $\sqrt{\ell}$ in the rate.
\begin{restatable}[Special case of Corollary~\ref{cor:iid-formal}]{corollary}{multisample}\label{cor:multiple-samples}
 Let $M>0$ and let $\mathcal{J} \subseteq \{J \colon \|J\|_\infty \le M \}$ denote a collection of interaction matrices. Assume that $\ell$ independent samples are obtained from $\Pr_{J^*}$ where $J^*\in \J$. There is an estimator $\hat{J}$ such that, with probability $\ge 1-\delta$,
	\[
	\|\hat{J}-J^*\|_F \le C(M)\sqrt{\frac{\log N(\mathcal{J}, \|\cdot\|_2, 1/(n\ell)) + \log (1/\delta)+\log\log n}{\ell}},
	\]
	where the same comments for $C(M)$ and the complexity of computing $\hat{J}$ made in Theorem~\ref{thm:informal-main} apply.
\end{restatable}
%This can be compared to standard bounds from literature on learning from multiple independent samples. There, the interaction matrix $J^*$ is either assumed either to satisfy $\|J^*\|_\infty \le O(1)$ as in this paper, or the slightly stronger assumption that $J^*$ is the adjacency matrix of some low-degree graph. Apart from that, no further assumptions are made. Interpreting their bounds in terms of the Forbenius norm, one has $\|\hat{J}-J^*\|_F \le O(\sqrt{n/\ell})$. The bound of Corollary~\ref{cor:multiple-samples} is better in the regime where the number of samples $\ell$ is not large, once more  restrictions on $J^*$ are posed.
Notice that Corollary~\ref{cor:multiple-samples} is phrased in terms of a general set $\mathcal{J}$. 
In particular, it can be applied to learn Ising models from multiple samples in the same setting studied by \cite{klivans2017learning}, where they learn $J^*$ while only assuming that $\|J^*\|_\infty \le M$. Utilizing the fact that the space of interaction matrices is an $O(n^2)$-dimensional vector space, one obtains (similarly to Corollary~\ref{cor:linear}): 
\begin{restatable}{corollary}{multinoassump} \label{cor:learn-no-assuption}
	Let $\J = \{ J \colon \|J\|_\infty \le M \}$ and assume that $\ell$ independent samples from $\Pr_{J^*}$ where $J^* \in \J$are obtained. Then, there is a polynomial time algorithm that finds $\hat{J} \in \J$ such that, w.p.~$\ge 1-\delta$,
	\[
	\|\hat{J}-J^*\|_F \le 
	C(M)\lp(\sqrt{\frac{n^2\log (n\ell) + \log (1/\delta)}{\ell}} \rp).
	\]
\end{restatable}

%This enables non-trivial bounds when $\ell \gg n$ while the prior work could already derive such bounds for $\ell \gg 1$. However, since Corollary~\ref{cor:learn-no-assuption} is a direct application of the bound on learning from a single sample, it can be modified to handle structured dependencies between the samples.
\noindent This provides a new polynomial-time algorithm for this problem. Comparing to our error bound, \cite{klivans2017learning} achieved an error of $\sqrt{n}(\log n/\ell)^{1/4}$, as also stated in Table~\ref{tbl:results}. 

\medskip Interestingly, as we discuss next, our results can be extended to settings where the samples are not independent. 

\paragraph{Beyond Independent Samples.} 
In many applications the learning task involves either a few or many dependent samples. For the sake of presentation, we assume time-series dependencies although other dependencies of a more complex structure can be studied in a similar fashion.  Given an interaction matrix $J_0$ that controls the dependencies within each sample and $J_1$ that controls dependencies between consecutive samples, we define the following joint distribution over samples $x^1,\dots,x^\ell \in \{-1,1\}^n$:
\[
\Pr_{J_0,J_1,\ell}\lp[x^1\cdots x^\ell\rp] \propto 
\prod_{t=1}^\ell \exp\lp(-(x^t)^\top J_0 x^t/2\rp) \prod_{t=1}^{\ell-1} \exp\lp(-(x^t)^\top J_1 x^{t+1}/2\rp).
\]
The following statement bounds the learning error, that can be meaningful even for $\ell=2$:
\begin{restatable}[Special case of Corollary~\ref{cor:weakmulti}]{corollary}{multiweak}
\label{cor:weakly-dep}
	Let $\ell \ge 2$, let $\J_0$ and $\J_1$ be collections of interaction matrices of infinity norm bounded by $M$, and let $(x^1,\dots,x^\ell)\sim \Pr_{J_0^*,J_1^*,\ell}$ for some $J_0^*\in \J_0$ and $J_1^*\in\J_1$. Then, there exists an estimator $(\hat{J}_0,\hat{J}_1)$ such that, w.p. $\ge 1-\delta$, both $\| J_0^*-\hat{J}_0\|_F$ and $\|J_1^*-\hat{J}_1\|_F$ are bounded by
	\[
    \frac{C(M)}{\sqrt{\ell}} \sqrt{\log N\lp(\mathcal{J}_0, \|\cdot\|_2, \frac{1}{n\ell}\rp)
		+ \log N\lp(\mathcal{J}_1, \|\cdot\|_2, \frac{1}{n\ell}\rp)
		+\log\log n + \log(1/\delta)}.
	\]
	%where the same comments for $C(M)$ and the complexity of computing $(\hat{J}_0,\hat{J}_1)$ made in Theorem~\ref{thm:informal-main} apply.
\end{restatable}

\subsection{Lower bounds} \label{sec:lower}
We first present a general lower bound based on the metric entropy of $\J$ and then we show that our lower bound is strong enough to provide nearly tight results for the cases of linear subspaces and finite sets. The following is shown in Section~\ref{s:lower_bound}.
% (the formal statement is given in Theorem~\ref{t:lb}): 
\begin{restatable}
{theorem}{lb}\label{t:lb}
	Let $r>0$ and suppose there exists some $R,\alpha > 0$ and a family  $\mathcal{J}$ of interaction matrices such that: (1) for all $J \in {\cal J}$ the infinity norm of $J$ is bounded by $1-\alpha$ and the diameter\footnote{A set $\mathcal{K}$ has diameter at most $R$ if for any $A,B \in \mathcal{K}$ we have $\|A-B\|_F \leq R$.} of $\mathcal{J}$ is bounded by $R$; and (2) it holds that
	\[
	\frac{\log N(\J, \|\cdot\|_F, 2r)}{2}
	\ge C(\alpha) R^2 + \log 2,
	\]
	where $C(\alpha)$ is a specific constant determined in the proof. 
	Then, any estimator $\hat{J}(x)$ based on a single sample attains a minimax error of
	$
	\max_{J^* \in \J} \E_{x \sim P_{J^*}}[\|\hat{J}(x)- J^*\|_F]
	\ge r/2.
	$
\end{restatable}

%Specifically, in the common setting where $N(\J,\|\cdot\|_F,\epsilon) \approx \lp(\frac{1}{\epsilon}\rp)^\kappa$ for some $\kappa \ge 1$, one can obtain a lower bound of $r = \Omega(\sqrt{\kappa})$.\footnote{In the discussed setting, is possible to find a subset $\J'\subseteq \J$ of diameter $R = O(\sqrt{\kappa})$ with covering numbers $N(\J',\|\cdot\|_F,r) \approx (R/r)^\kappa$. Applying Theorem~\ref{t:lb} on this value of $\J'$ derives the desired bound.} This is nearly tight since the upper bound from Theorem~\ref{thm:informal-main} is of the order $\sqrt{\kappa \log n}$. 
%While the guarantee of Theorem~\ref{t:lb} requires the diameter of $\J$ to be small, it is possible derive lower bounds for sets $\J$ of large diameter by applying Theorem~\ref{t:lb} on a subset $\mathcal{V}\subseteq \J$ of smaller diameter.
Using Theorem~\ref{t:lb}, one can derive a nearly-tight lower bound on the estimation error for linear combinations of $k$ known matrices $J_1,\dots,J_k$:
\begin{restatable}{corollary}{lowerbound}\label{cor:lb-linspace}
	Let $k \in\mathbb{N}$, let $J_1,\dots,J_k$ be interaction matrices with disjoint supports\footnote{The support of a matrix $J$ is defined as the set of its non-zero elements.} such that $\|J_i\|_\infty \le 1$ and $\|J_i\|_F \geq k$ for all $i$.  Define $\J = \{ J=\sum_i \alpha_i J_i \colon \alpha_i \in \mathbb{R}, \|J\|_\infty \le 1 \}$. Then, any one-sample estimator $\hat{J}(x)$ has a minimax error of
	$
	\sup_{J^* \in \J} \E_{x\sim \Pr_{J^*}}\lp[ \| \hat{J}-J^* \|_F \rp] \ge c \sqrt{k}.
	$
\end{restatable}
In the proof of Corollary~\ref{cor:lb-linspace}, one constructs a lower bound for a family of size $\exp(O(k))$. Hence, we derive the following tight lower bound of $\Omega(\sqrt{\log |\mathcal{J}|})$ on estimation from finite families of distributions:
\begin{corollary}
	Let $m > 0$. There exists a family $\mathcal{J}$ of cardinality $|\mathcal{J}| = m$ that satisfies $\sup_{J \in \J} \|J\|_\infty \le 1/2$, such that the minimax error satisfies 
	$
	\max_{J^* \in \J} \E_{x \sim P_{J^*}}[\|\hat{J}(x)- J^*\|_F]
	\ge c\sqrt{\log m}
	$
	(where $c>0$ is a universal constant).
\end{corollary}

\subsection{Improved bounds for estimating a single parameter} \label{sec:intro-one-param}

We further present an application to the single-sample setting studied in prior work \cite{chatterjee2007estimation,bhattacharya2018inference} on estimating a single parameter $\beta$ (this follows from the main lemmas in the proof of  Theorem~\ref{thm:informal-main}):
\begin{restatable}{corollary}{oneparam}
\label{cor:one-param-informal}
	Let $M>0$, let $J_0$ be a fixed matrix with $\|J_0\|_\infty \le 1$ and let $\beta^*$ be some unknown parameter satisfying $|\beta^*|\le M$. Then, there exists an estimator $\hat{\beta}$ from a single sample $x \sim \Pr_{\beta^* J_0}$such that w.p. $\ge 1-\delta$,
	$
	|\hat\beta-\beta^*|
	\le C(M) F({\beta^*J_0})^{-1/2}\left(\log\log n + \log(1/\delta)\right)
	$
	where $F(\cdot)$ is defined as in~\eqref{eq:ising}.
\end{restatable}
Notice that the bound is inversely proportional to the square root of the partition function $F(J^*)$, which captures the strength of dependencies between the nodes and this bound is generally stronger than the one obtained using the Frobenius norm.
Corollary~\ref{cor:one-param-informal} improves over prior work that required further assumptions to hold and obtained no guarantees at the vicinity of some phase transitions (see Section~\ref{sec:priorwork} for a comparison).

\section{Overview of Techniques}\label{sec:technique}

We start by presenting the main techniques used in this paper in Section~\ref{sec:technical-lemmas} and proceed with a proof sketch in Section~\ref{sec:pr-sketch}.

\subsection{Key Technical Insights and Vignettes}\label{sec:technical-lemmas}
\paragraph{From Low-Temperature to High-Temperature (Dobrushin).}
While nodes of the Ising model can be complexly dependent, when the correlations are sufficiently weak, the model shares important similarities to product measures. A well-studied mathematical formulation of weak dependencies for general random vectors is Dobrushin's uniqueness condition, defined formally in Section~\ref{sec:dobrushin}. For Ising models, a sufficient condition implying Dobrushin's is $\|J^*\|_{\infty} = \alpha < 1$, where $\alpha$ is a constant; see e.g.~\citep{dobrushin1987completely,stroock1992logarithmic}.\footnote{Dobrushin's condition is slightly more general and defined in terms of a bound on the total influence exercised to any one node by the other nodes. See Section~\ref{sec:dobrushin} for the general form of the condition. However, as is often done in the literature, we use the slightly stronger but easier to interpret bound on $\|J^* \|_{\infty}$.}  While Dobrushin's condition implies multiple desirable properties (see e.g. \cite{chatterjee2005concentration,weitz2005combinatorial}), we will specifically use the fact that functions of the Ising model concentrate well under this condition; see e.g.~\cite{chatterjee2005concentration,daskalakis2017concentration,gheissari2017concentration,gotze2019higher,adamczak2019note}. Unfortunately, the regimes we are considering in this paper may lie well outside Dobrushin's condition, and the tools available to handle Ising models that do not satisfy Dobrushin's condition are significantly weaker and restricted, and concentration does not hold in general.

In this work, we prove concentration inequalities for Ising models outside of Dobrushin's condition via reductions to the Dobrushin regime: we show that we can condition on a subset of the variables, such that in the conditional distribution, the unconditioned variables satisfy Dobrushin. 
A basic example where we can see such behavior is when $J$ is the incidence matrix of a bipartite graph, namely, there exists a set $I \subseteq [n]$ such that $J_{ij} = 0$ whenever either $i,j\in I$ or $i,j\in [n]\setminus I$. If we condition on $x_{-I} := x_{[n]\setminus I}$, then $\{x_i \colon i \in I\}$ are conditionally independent and particularly, satisfy Dobrushin.
The following lemma generalizes this intuition. For the purposes of this lemma, we  work with Ising models with {\em external fields}. Given an interaction matrix $J^*$ and a vector $h$ of external fields, we define the distribution over $x\in \{\pm 1\}^n$ by $\Pr_{J^*,h}(x) \propto \exp(x^{\rm T}J^*x/2 + h^T x)$.
\begin{informallemma}[Conditioning Trick] 
\label{lem:informal-subsample}
	Let $p_{J^*,h}(x)$ be an Ising model with interaction matrix $J^*$ satisfying $\|J^*\|_\infty = M$ and any external field vector $h$. Then there exist $\ell=O(\log n)$ sets $I_1,\dots,I_\ell \subseteq [n]$ such that:
	\begin{enumerate}
		\item Each $i\in [n]$ appears in exactly $\ell'=\lceil\ell/(16M)\rceil$ different sets $I_j$.
		\item For all $j \in [\ell]$, the conditional distribution of $x_{I_j}$, conditioning on any setting of $x_{-I_j}$, satisfies Dobrushin's condition.
	\end{enumerate}
\end{informallemma}

We apply this lemma repeatedly in our proof, %While it is used to replace Chatterjee's exchangeable pair technique, a key ingredient used in~\cite{chatterjee2005concentration,ghosal2018joint,bhattacharya2018inference,daskalakis2019regression}, it also 
as it allows us to tap into the flexibility of dealing with weakly dependent random variables.
As a first application, given a vector $a \in \mathbb{R}^n$, we obtain a lower bound on the variance of $a^\top x$. It is well known that if $x$ is an $i.i.d.$ vector of binary random variables, each with variance $v$, then $\Var(a^\top x) = v\|a\|_2^2$. Furthermore, if $x$ satisfies Dobrushin's condition, then the entries of $x$ are nearly independent and we can also show that $\Var(a^\top x) \ge \Omega(\|a\|_2^2)$.
%using the fact that its covariance matrix is diagonally dominant. 
We will use Informal Lemma~\ref{lem:informal-subsample} to show that a similar lower bound holds even beyond Dobrushin's condition. 
\begin{informallemma}[Anti-Concentration]
\label{lem:informal-var}
	Suppose that $x$ is sampled from an Ising model whose interaction matrix satisfies $\|J^*\|_\infty = O(1)$ and whose external field vector satisfies $\|h\|_\infty = O(1)$. Then, for all $a \in \mathbb{R}^n$,
	\[
	\Var(a^\top x) \ge \Omega(\|a\|_2^2).
	\]
\end{informallemma}
\begin{proof}[Proof sketch]
To prove this lemma, consider the sets $I_1,\dots,I_\ell$ from Informal Lemma~\ref{lem:informal-subsample}. First, we claim that there exists $j \in [\ell]$ such that $\|a_{I_j}\|_2^2 \ge \Omega(\|a\|_2^2)$. Indeed, by linearity of expectation, if we draw $j \in [\ell]$ uniformly at random then,
\[
\E_j[\|a_{I_j}\|_2^2]
= \E\left[\sum_{i=1}^n \mathbf{1}(i \in I_j) a_i^2\right]
= \sum_{i=1}^n \frac{\ell'}{\ell} a_i^2
= \frac{\ell'}{\ell}\|a\|_2^2
\ge \Omega(\|a\|_2^2).
\]
Hence, there exists a set $I_j$ that achieves this expectation, namely, $\|a_{I_j}\|_2^2 \ge \Omega(\|a\|_2^2)$. Now using that, conditioning on $x_{-I_j}$, $x_{I_j}$ has a low Dobrushin coefficient, as implied by Informal Lemma~\ref{lem:informal-subsample}, we can bound $\Var[a^\top x \mid x_{-I_j}] \ge \Omega(\|a_{I_j}\|_2^2)$ as discussed above, using weak dependence. Since conditioning reduces the variance on expectation, we conclude that
\[
\Var(a^\top x) 
\ge \E_{x_{-I_j}}[\Var[a^\top x | x_{-I_j}]] 
\ge \Omega(\|a_{I_j}\|_2^2)
\ge \Omega(\|a\|_2^2).
\]
\end{proof}
\paragraph{Measure Concentration for Non-Polynomials.}

There are multiple recent works studying the concentration of polynomial functions of the Ising model~\cite{daskalakis2017concentration,gheissari2017concentration,gotze2019higher,adamczak2019note,}. Here, we would like to bound the tails of general functions, in terms of their polynomial Taylor approximations. 
By a simple modification to the proof of~\cite{adamczak2019note}, we can derive the following:

\begin{theorem}\label{thm:informal-concentration}
	Let $f:\{0,1\}^n \mapsto \R$ be an arbitrary function and
	$X$ be sampled from an Ising model which satisfies Dobrushin's condition. Then
	\[
	\Pr[|f(X)-\E f(X)|>t]
	\le \exp\lp(-c\min\lp(\frac{t^2}{\|\E_X Df(X)\|_2^2 + \max_x\|Hf(x)\|_F^2}, \frac{t}{\max_x \|Hf(x)\|_2}\rp)\rp).
	\]
	Here $D_if(x) = (f(x_{i+}) - f(x_{i-}))/2$ is the \emph{discrete derivative}, where $x_{i+}$ and $x_{i-}$ are obtained from $x$ by replacing the value of $x_i$ with $1$ and $-1$, respectively. The vector of discrete derivatives is denoted by $Df$ and $Hf$ is the $n\times n$ matrix of second discrete derivatives. 
\end{theorem}
Theorem~\ref{thm:informal-concentration}
can be trivially extended to derive bounds based on higher order Taylor expansion,
extending \cite[Theorem~2.2]{adamczak2019note} for multi-linear polynomials.
%In that setting, $Hf(x)$ is a constant matrix. \cite[Theorem~2.2]{adamczak2019note} in fact derives concentration inequalities for polynomials of any degree and this result can also be generalized in a similar fashion. 

\subsection{Proof Sketch of our Upper Bound}\label{sec:pr-sketch}
Using the  tools from Section~\ref{sec:technical-lemmas}, we present a sketch of the proof of our main results. We start by describing the algorithm that is going to be used. 
A standard approach is \emph{maximum likelihood estimation} (MLE), which outputs the maximizer $\hat{J}$ of the probability of the given sample $x$, namely, $\hat{J} := \argmax_{J} \Pr_{J}[x]$.
% \costis{replaced this: $\widehat{\beta}$ of the probability of the given sample $x$, namely, $\widehat{\beta} := \argmax_{\beta} \Pr_{\beta J}[x]$}. 
%where $x$ is the sample from the distribution $\Pr_{\beta^*}$ parametrized by an unknown $\beta^*$, and $\Pr_{\beta}[x]$ is the probability of $\beta$ to draw $x$. 
Unfortunately, for Ising models, the MLE requires computing the partition function which is computationally hard to approximate \cite{sly2014counting}. A recourse, suggested by \cite{chatterjee2007estimation}, is to compute the \emph{maximum pseudo-likelihood estimator} (MPLE) of \cite{besag1974spatial,besag1975statistical} instead. 
One typically minimizes the negative log pseudo-likelihood,
\begin{equation}
\varphi(x;J)
:=  -\sum_{i=1}^n \log\Pr_{J }[x_i \mid x_{-i}],
\end{equation}
where $\Pr_{ J}[x_i\mid x_{-i}]$ is the probability of $\Pr_{J}$ to draw $x_i$ conditioned on the remaining entries of $x$, denoted $x_{-i}$.
If $\mathcal{J}$ is a convex set, then this is a convex function  which can be optimized 
%global optima $\hat{\beta}$ can be subsequently found 
using appropriate first-order optimization techniques to find an optimum $\hat J$.

A bound on the error can then be proved by the following steps. First, we show that for every $J_0 \in \mathcal{J}$ that is far from $J^*$ we have
\begin{equation}\label{eq:gap}
\phi(x;J_0) \geq \phi(x;J^*) + \Omega(1)
\end{equation}
with high probability.
One can prove this using a Taylor approximation of $\varphi$, while utilizing the first directional derivatives of $\varphi$ that we define as
\[
\frac{\partial\varphi(x;J)}{\partial A} 
:= \lim_{t\to 0} \frac{\varphi(x;J+At) - \varphi(x;J)}{t} 
\]
and the second directed derivatives that we similarly define.
Evaluating the Taylor approximation of $t \mapsto J^* + t(J_0-J^*)$ at $t=1$, one obtains that
\begin{equation}\label{eq:taylor}
\phi(x;J_0) = \phi(x;J^*) + \|J_0-J^*\|_F \frac{\partial \phi(x;J^*)}{\partial A} + \frac{1}{2} \|J_0-J^*\|_F^2 \frac{\partial^2 \phi(x;J_x)}{\partial^2 A}; \quad
\text{where }
A = \frac{J_0 - J^*}{\|J_0-J^*\|_F}
\end{equation}
and $J_x$ is a point in the segment connecting $J_0$ with $J^*$.
Hence, to show a large gap between $\phi(x;J_0),\phi(x;J^*)$ we need a good upper bound on the absolute value of the first derivative and a good lower bound on the second derivative. 
% However, in our setting it is difficult to
% prove these high probability bounds for specific 
% $L,\mu$. The reason is that
% in previous works (\cite{daskalakis2019regression, bhattacharya2018inference, chatterjee2007estimation}) this is obtained by assuming
% that the model
% satisfies Dobrushin's condition, or the frobenius norm of $J^*$ is of maximal order \yuval{partition function bounds}.\yuval{why is this the reason. In fact, the true reason is that the functions may not concentrate.}
% \yuval{not clear how the next sentence contributes} Furthermore, the derivative is not a polynomial function,
% rendering the task of showing concentration even harder.
% Therefore, we will instead show
% that even if these two quantities vary significantly,
% their ratio will be bounded with high probability, \emph{conditionally} on the values of some nodes. 

We now turn to the specific challenges encountered when
trying to prove these bounds.
The derivative $\varphi'(x;J^*)$ takes the form
$$
\frac{\partial \varphi(x;J^*)}{\partial A} = \sum_{i=1}^n \varphi'_i(x;J^*)\quad;\quad
\varphi'_i(x;J^*) := -\frac{\partial}{\partial A}\log\Pr_{ J}[x_i \mid x_{-i}]\big|_{J = J^*}.
$$
We notice that $\E[\varphi'_i(x;J^*)\mid x_{-i}] = 0$,
hence it suffices to show concentration of the derivative around its mean to obtain a good upper bound.
However, tail bounds on the gradient from prior work do not lead us to the optimal bound on the derivative in our setting.
Instead, we use Lemma~\ref{lem:informal-subsample}
to select a number of subsets $I_1,\ldots,I_l$
of $[n]$, such that conditioned on $x_{-I_j}$, $x_{I_j}$
satisfies Dobrushin's condition. The lemma also
guarantees that each $i \in [n]$ belongs to $\ell'$ different subsets $I_j$ where $\ell'$ is a constant fraction of $\ell$,
which means we can write
\begin{equation}\label{eq:1}
\left|\frac{\partial \varphi(x;J^*)}{\partial A}\right|
= \left|\sum_{i=1}^n \varphi'_i(x;J^*)\right| = 
\left|\frac{1}{\ell'} \sum_{j=1}^\ell \sum_{i \in I_j} \varphi'_i(x;J^*)\right| 
\le \frac{\ell}{\ell'} \max_j \left|\sum_{i \in I_j} \varphi'_i(x;J^*)\right|
\le O\left( \max_{j\in [\ell]} \left|\sum_{i \in I_j} \varphi'_i(x;J^*)\right| \right).
\end{equation}
Hence, it suffices to bound each one of the terms
that appear in the maximum. In fact, since each term $\sum_{i \in I_j} \varphi'_i(x;J^*)$ has zero mean conditioned on $x_{-I_j}$, it suffices to show that it concentrates around its expectation conditioned on $x_{-I_j}$. Given that conditioning on $x_{-I_j}$, $x_{I_j}$ satisfies Dobrushin's condition, we can use the concentration inequality from Informal Theorem~\ref{thm:informal-concentration}, to derive that 
\[
\left|\sum_{i\in I_j} \varphi'_i(x;J^*)\right| \le O\lp(\lp\|\E\lp[Ax~\middle|~x_{-I_j}\rp]\rp\|_2 + \|A\|_F\rp),
\]
with high probability. Applying \eqref{eq:1} and union bounding over $j \in [\ell]$, we deduce that with high probability,
\begin{equation}\label{eq:644}
|\varphi'(x;J^*)| \le  \widetilde{O}\lp(\max_{j\in [\ell]}\lp\|\E\lp[Ax~\middle|~x_{-I_j}\rp]\rp\|_2 + \|A\|_F\rp).
\end{equation}

We now show a lower bound on $\partial ^2\varphi(x;J_x)/\partial^2 A$, where $J_x$ is in the segment connecting $J^*$ and $J_0$.
Some simple calculations show that for every $J$ in this segment,
\begin{equation}\label{eq:665}
\frac{\partial^2\varphi(x;J)}{\partial^2 A} \ge \Omega\lp(\|Ax\|_2^2\rp).
\end{equation}
We then proceed by showing that: (a) the expectation of $\|Ax\|_2^2$ is lower bounded appropriately; and (b) it concentrates around its expectation.
Note that (a) reduces to showing an expectation
bound for a sum of squares of linear functions.
This can also be phrased as a variance bound for linear 
functions of the Ising model, which is exactly the type of
result that
Informal Lemma~\ref{lem:informal-var} provides.
Using it, we manage to prove that the expectation of 
the second derivative conditioned on $x_{-I_j}$ is at least 
\begin{equation}\label{eq:77}
\E\lp[\|Ax\|_2^2 \middle| x_{-I_j}\rp]
\ge \Omega\lp(\lp\|\E\lp[Ax~\middle|~x_{-I_j}\rp]\rp\|_2^2 + \|A\|_F^2\rp).
\end{equation}
By concentration of polynomials under Dobrushin's condition \cite{adamczak2019note}, we will show that $\|Ax\|_2^2$ is at least the right hand side of \eqref{eq:77} with high probability, and taking a union bound over $j\in [\ell]$, we derive that w.h.p.,
\begin{equation}\label{eq:643}
\frac{\partial^2\varphi(x;J_x)}{\partial^2 A}
\ge \|Ax\|_2^2
\ge \Omega\lp(\max_{j\in[\ell]}\lp\|\E\lp[Ax~\middle|~x_{-I_j}\rp]\rp\|_2^2 + \|A\|_F^2\rp).
\end{equation}
If $\|J^*-J_0\|_F = \tilde{\Omega}(1)$, we derive by \eqref{eq:taylor}, \eqref{eq:644} and \eqref{eq:643} that that inequality~\eqref{eq:gap} holds w.h.p. Moreover, the further $J_0$ is from $J^*$, the higher is the probability.
% Fortunately, for polynomial
% functions of the Ising model, we have very sharp results
% from the literature \cite{adamczak2019note}. These give us the desired
% concentration radius, which is smaller in magnitude that
% the lower bound we proved for the expectation. 
% This way, the ratio of the first and second derivatives
% remains bounded, which concludes the proof for the single parameter case.\vardis{maybe mention that $\varphi' \le \sqrt{\varphi''}$ instead? I'm not sure}

We now have to use \eqref{eq:gap} to derive the error bound. To do that, we would like to show that for all $J$ that are far from $J^*$ in Frobenius norm, $\phi(x;J)>\phi(x;J^*)$. Since $\phi(x;\hat J) \leq \phi(x;J^*)$, this would imply that $\hat J$ is close to $J^*$. Proving that this statement holds with high probability for all far enough points requires more than a union bound, since there might be infinitely many points. Instead, we will construct a finite subset $\mathcal{U}$
of these points such that
every point is $\epsilon$ close to one in $\mathcal{U}$ ($\mathcal{U}$ forms an $\epsilon$-net). By a union bound over $\mathcal{U}$ we prove that with high probability \eqref{eq:gap} holds for all points in $\mathcal{U}$. Since $\phi$ is Lipschitz as a function of the matrix $J$, this suffices to argue that for all far enough points, their $\phi$ value is much larger than that of $J^*$. 
We note that union bounding \eqref{eq:gap} over $|\mathcal{U}|$ events corresponding to all possible $J \in \mathcal{U}$, requires each event to hold with sufficiently high probability, and this holds whenever $\|J^*-J\|_F \ge \Omega(\sqrt{\log |\mathcal{U}|})$.

\section{Comparison to Prior work} \label{sec:priorwork}

\paragraph{Comparison with multiple-sample bounds.}
An important line of previous work focuses on learning Ising models from multiple independent samples. The first work that gives a polynomial-time algorithm for this problem is\citet{bresler2015efficiently} and improved results were obtained by \cite{vuffray2016interaction,hamilton2017information,klivans2017learning} and others. \cite{klivans2017learning} showed that under the common assumption $\|J^*\|_\infty \le O(1)$, it is possible, using $\ell$ samples, to learn each row of $J^*$ up to an error of $O((\log(n)/\ell)^{1/4})$, which translates to a Forbenius norm error of $O(n^{1/2}(\log(n)/\ell)^{1/4})$.  
In comparison, Corollary~\ref{cor:multiple-samples} can derive better guarantees even with one or a few samples, assuming additional structural assumptions on $J^*$. Further, Corollary~\ref{cor:learn-no-assuption} that assumes the same setting as \cite{klivans2017learning}, retains polynomial-time learnability, while reducing to a single-sample algorithm that does not utilize independence. This enables to consider dependent samples with only a small overhead.

\paragraph{Comparison with single-sample bounds.}
Another interesting line of work involves learning the Ising model from a single sample of the distribution. The first to work on this problem was \citet{chatterjee2007estimation}, who assumed a single-parameter family, $\J = \{\beta J_0 \colon |\beta|\le M\}$ where the goal is to learn $\beta$. In subsequent work, \cite{bhattacharya2018inference} derived an improved bound and \citet{ghosal2018joint} presented an algorithm that jointly learns $\beta$ and an \emph{external field} $\theta$, assuming that $\Pr[x]\propto \exp(-\beta x^\top J x/2 + \theta \sum_i x_i)$. Further, \citet{daskalakis2019regression} studied linear regression with Ising model dependencies, which corresponds to learning $\beta$ together with multiple external field parameters. In comparison, Theorem~\ref{thm:informal-main} is the first to learn Ising models using one-sample from a complex family of matrices.

We further discuss the improvements over the prior work on single-sample estimation that are apparent in Corollary~\ref{cor:one-param-informal} and are essential for obtaining the results of this paper: (1) Removal of additional assumptions that require the log partition function $F(J^*)$ to be \emph{well behaved}, yielding no guarantees in scenarios such as at the vicinity of some phase transitions. (2) Obtaining high probability estimates on single-parameter families that enables generalizing to arbitrary families via a union bound. These two improvements necessitates a new proof approach as presented in Section~\ref{sec:technique}.

%\newpage
%\tableofcontents
%\newpage
%\input{estimating-parameters}
\section{Further Notation and Definitions}\label{s:prelims}
This section establishes the notational conventions used and presents some background definitions used throughout the paper. We start with the notational conventions.

\paragraph{Standard notations and definitions.}
\begin{itemize}
    %\item Matrix notation: We use capital letters such as $J,A$ to denote matrices.
    \item Sets of indices: denote by $[n]= \{1,\dots,n\}$. Given $I \subseteq [n]$, let $-I := [n]\setminus I$, and given $i \in [n]$ let $-i = -\{i\} = [n]\setminus \{i\}$.
    \item Indexed vectors and matrices: Given a vector $a = (a_1,\dots, a_n)$ and a subset $I \subseteq [n]$, let $a_I$ denote the $|I|$ coordinate vector$\{a_{i} \colon i \in I\}$. Similarly, for a matrix $A$ of dimension $n\times m$, $I \in [n]$ and $I' \in [m]$, let $A_{II'}$ denote the corresponding submatrix. Similarly, let $A_{I} := A_{I[n]}$ and $A_{\cdot I} = A_{[n] I}$, and let $A_{i} := A_{\{i\}}$.
    \item Standard mathematical sets: Let $\mathcal{S}^{k-1}$ denote the $k$-dimensional unit sphere, $\{x \in \mathbb{R}^k \colon \|x\|_2 = 1\}$. And given $m,n > 0$ integers, let $M_{n\times m}(\mathbb{R}):= M_{n\times m}$ denote the space of real matrices of dimension $n\times m$.
    \item Ising model distributions: Given a symmetric matrix $J$ with zeros on the diagonal, let $\Pr_J$ denote the Ising model with interaction matrix $J$, defined as in \eqref{eq:ising}.
    We say that a random variable $x$ with interaction matrix $J$ and external field $h$ is $(M,\gamma)$-bounded, if $\|J\|_{\infty} \le M$, and if $\min_{i \in [n]}\Var(x_i|x_{-i}) \ge \gamma$ with probability $1$. In other words, if for all $x$ and all $i$, $\Pr[x_i=1|x_{-i}](1-\Pr[x_i=1|x_{-i}]) \ge \gamma$.
    \item Absolute constants: We let the notations $C,c',C_1,\dots$ denote constants that depend only on $M$ and $\gamma$, and are bounded whenever $M$ is bounded from above and $\gamma$ from below (unless the dependence on $\gamma$ and $M$ is stated explicitly).
    \item Conditional variance: given two random variables $X$ and $Y$, define $\Var[X|Y] := \E_X[(X-\E[X|Y])^2| Y]$. Since the conditional expectation $\E[X|Y]$ is a random variable which is a function of $Y$, so is $\Var[X|Y]$.
\end{itemize}

\paragraph{Matrix norms.}
Given a real matrix $A$ of dimension $m\times n$, let $\|A\|_F^2 = \sum_{ij} A_{ij}^2$ denote the Frobenius norm, let
\[
\|A\|_2 = \max_{u \in \mathbb{R}^n \setminus \{0\}} \frac{\|Au\|_2}{\|u\|_2}
\]
and let 
\[
\|A\|_\infty = \max_{u \in \mathbb{R}^n \setminus \{0\}} \frac{\|Au\|_\infty}{\|u\|_\infty} 
= \max_{i=1,...,m} \sum_{j=1}^n |A_{ij}|.
\]
The following inequalities are known for any symmetric matrix $A$: $\|A\|_2 \le \|A\|_F$ and $\|A\|_2 \le \|A\|_\infty$.

\paragraph{Dobrushin's condition.}
Next we define a variant of Dobrushin's uniqueness condition (high-temperature condition) for Ising models that we use. The more general form of the condition is presented in Section~\ref{sec:dobrushin}.
 \begin{definition}[Dobrushin's condition]
 \label{def:dob}
 Given an Ising model $x$ with interaction matrix $J$, we say that it satisfies \emph{Dobrushin's condition} if $\|J\|_\infty < 1$, where $\alpha:= \|J\|_\infty$ is called the \emph{Dobrushin's coefficient}. 
 \end{definition}

\paragraph{Optimization over a vector space $\mathcal{V}$.}
Notice that replacing the interaction matrices $J_1,\dots,J_k$ with other matrices $J_1',\dots,J_k'$ that span the same linear subspace of $M_{n\times n}(\mathbb{R})$ does not change the studied problem. Hence, we will forget about $J_1,\dots,J_k$ and replace them with their span $\mathcal{V}$.
\section{Analyzing the Maximum Pseudo-Likelihood Estimator}\label{sec:pr-multiparam}
This section contains the proof of our main result which is the following theorem.
\begin{restatable}{theorem}{general}\label{t:general}
	Let $\mathcal{J}$ be some set of matrices of infinity norm bounded by a constant $M$, and let
	\begin{align*}
	R &= \min \lp\{r \ge 0 \colon r \ge C(M)\sqrt{ \log\log n + \log N(\mathcal{J}, \|\cdot\|_2, cr^2/n) + \log (1/\delta)}\rp\}\\
	&\le C(M) \sqrt{ \log\log n + \log N(\mathcal{J}, \|\cdot\|_2, c/n) + \log (1/\delta)}.
	\end{align*}
	Then, with probability $1-\delta$, it holds that any point
	$
	J \in \mathcal{J}
	$
	that satisfies 
	$
	\|J - \hat{J}\|_F \ge R,
	$
	also satisfies
	$
	\varphi(J) \ge \varphi(J^*) + cR^2
	$.
	In particular, there exists an algorithm that, given one sample $x\sim\Pr_{J^*}$ where $J^* \in \mathcal{J}$, outputs $\hat{J}=\hat{J}(x)$ such that
	\[
	\|\hat{J}-J^*\|_F \le R.
	\]
\end{restatable}
Theorem~\ref{t:general} guarantees that we can find a matrix $\hat{J}$ that is $R$ close to the true interaction matrix $J^*$ in Frobenius norm. A detailed discussion about applications of Theorem~\ref{t:general}, including the case where $\mathcal{J}$ is a $k$-dimensional linear subspace of $\R^{n \times n}$, is given in Section~\ref{s:applications}.

Section~\ref{s:main_part} contains an overview of the proof, while the main lemmas are presented in the following sections. 

\begin{comment}
\begin{theorem}
\label{thm:main}
    \yuval{this should be replaced with the more general theorem. Linear subspaces are moved to the applications}
	Let $\mathcal{V}$ denote a subspace of $M_{n\times n}(\mathbb{R})$ of dimension $k \ge 1$ containing symmetric matrices with zero diagonals, fix $M > 0$, let $\mathcal{J} := \{J \in \mathcal{V} \colon \|J\|_\infty \le M\}$ and fix $\delta > 0$. There exists an algorithm that runs in time $\mathrm{poly}(n)$, that given one sample $x\sim\Pr_{J^*}$ where $J^* \in \mathcal{J}$, outputs $\hat{J}=\hat{J}(x)$, such that with probability $1-\delta$:
	\[
	\|J^*-\hat{J}\|_F \le C(M) \sqrt{k\log n + \log(1/\delta)},
	\]
	where $C(M)>0$ is bounded when $M$ is bounded.
\end{theorem}

Theorem~\ref{thm:main} guarantees that we can find a matrix $\hat{J}$ that is 
$\tilde{O}(\sqrt{k})$ close to the true interaction matrix
$J^*$ in Frobenius norm. Notice that both matrices
can have a Frobenius norm as high as $\sqrt{n}$, which
makes this a non-trivial guarantee. Section~\ref{s:main_part} contains an overview of the proof, while the main lemmas are presented in the following sections. 
%\yuval{Removed the organization to the bottom of section 3.1 after people have already read it.}
\end{comment}

\subsection{Overview of the proof}\label{s:main_part}
The algorithm used to estimate $J^*$ will be the maximum pseudo-likelihood estimator (MPLE), as mentioned in Section~\ref{s:prelims}:
\begin{equation}\label{eq:pmle}
\arg\max_{J \in \mathcal{J}} PL(J ; x) := \arg\max_{J \in \mathcal{J}} \prod_{i\in [n]} \Pr_J[x_i | x_{-i}].
\end{equation}

In fact, the above maximization problem is concave and we are able to find an approximate solution using first-order methods, if $\mathcal{J}$ is a convex set. We will address the question of how to do this efficiently in Section~\ref{sec:optimization}. For convenience
in calculations, the function we will actually optimize is the negative log pseudo-likelihood:
\begin{align}\label{eq:pmleform}
\varphi(J) 
&:= - \log PL(J ; x)
= \sum_{i=1}^n (\log\cosh(J_i x) - x_i J_i x + \log 2).
\end{align}

A standard approach for showing consistency of the MPLE is by showing high probability upper bounds on the first derivatives of $\varphi$ combined with a high probability lower bound on the smallest eigenvalue of the Hessian.
To formalize this intuition in our setting, we begin by reviewing the definition of differentiation with respect to a matrix.

\begin{definition}
If $f:M_{n\times n}(\mathbb{R}) \mapsto \R$ is a twice 
continuously differentiable function and
$A \in M_{n\times n}(\mathbb{R}) \setminus \{0\}$, we define for
all $J \in M_{n\times n}(\mathbb{R})$
\[
\frac{\partial f(J)}{\partial A} = 
\lim_{t\to 0} \frac{f(J + tA)-f(J)}{t} = \frac{d f(J+tA)}{d t}\bigg|_{t=0}.
\]
\end{definition}
In the case of the MPLE, some simple calculations show that
\begin{align}\label{eq:pmle_der}
	\frac{\partial \varphi(J)}{\partial A}
	= \frac{1}{2}\sum_{i=1}^n 
	(A_ix) (\tanh(J_i x)-x_i) ; \quad
	\frac{\partial^2 \varphi(J)}{\partial A^2}
	= \frac{1}{2}\sum_{i=1}^n
	(A_ix)^2 \sech^2(J_ix).
\end{align}

The general strategy for proving that $\hat{J}$ is a good estimator of $J^*$ is the following. First, since the MPLE finds the minimum of $\phi$ on the set $\mathcal{J}$, we know that $\phi(\hat{J}) \leq \phi(J^*)$.
Next, we will show that any point $J$ that is far from $J^*$ in Frobenius norm should have a much larger value of $\phi$ than $J^*$. If we can show that this holds for all such matrices $J$ with high probability, then certainly $\hat{J}$ should lie close to $J^*$ in Frobenius norm, because $\phi(\hat{J}) \leq \phi(J^*)$. Hence, the bulk of the argument is in proving this gap between $\phi(J)$ and $\phi(J^*)$ if $J$ is far from $J^*$. 
We will do this in two steps. First, we prove that this gap exists with high probability for a single $J \in \mathcal{J}$, namely, that $\varphi(J) > \varphi(J^*)$ with high probability. This is described in more detail in Section~\ref{s:single}. The second step is to show that this holds for all $J$ that are far from $J^*$ with high probability. This requires finding a suitable $\epsilon$-net of these matrices and taking a union bound over all the elements in this net. A crucial property in this step is the Lipschitzness of $\phi$, which allows us to control the value of $\phi$ for all points that are close to a point in the net. The argument is described in Section~\ref{s:multiple}.

\begin{comment}

If we had the usual partial derivatives for a function in $\R^n$, we would try to bound the gradient and Hessian of 
this function. In our case, a similar strategy is to bound 
 the first and second derivatives of $\varphi$ with respect to matrices $A\in  M_{n\times n}(\mathbb{R})$ of unit norm, namely, those residing in $\mathcal{A} := \{ A \in  M_{n\times n}(\mathbb{R}) \colon \|A\|_F = 1\}$,
 as defined in Section~\ref{s:prelims}.\yuval{$\mathcal{A}$ should only contain the directions within the vector space.} 
Specifically, we 
would like a high probability upper bound on the derivative
of $\varphi$ at the true matrix $J^*$
and a lower bound on the second derivative of $\varphi$
at any point. 
 This is exactly the content of the following Lemma,
which is of central importance to the proof.
\vardis{Say that the k dimensional appears in later section}
\end{comment}

\subsubsection{A single dimensional problem}\label{s:single}
In this Section, we focus on a single $J_1 \in \mathcal{J}$. We would like to show that if $\|J_1 - J^*\|_F$ is large, then $\phi(J_1) - \phi(J^*)$ will also be large. This is formalized in the following Lemma.

\begin{restatable}
{lemma}{oneJ}\label{lem:one-element}
	Let $M>0$, let $x$ be drawn from the distribution parametrized by $J^*$ and let $J_1 \ne J^*$ be such $\|J_1\|_\infty,\|J^*\|_\infty \le M$. Then, there are constants $c,c'>0$ depending only on $M$ such that with probability $1-\log n\exp(-c\|J^*-J_1\|_F^2)$, it holds that
	\[
	\varphi(J_1)
	\ge \varphi(J^*) + c' \|J_1-J^*\|_F^2.
	\]
\end{restatable}

In order to prove this lemma, we have to somehow be able to compare $\phi(J^*)$ with $\phi(J_1)$. 
A common way to make this comparison is to view $\phi$ as a function on the line connecting $J_1,J^*$. Specifically, we define for all $t \in \lp[0,\|J_1 - J^*\|_F\rp]$,
\[
g(t) = \phi\lp(J(t)\rp) \quad \text{where }
J(t) = J^* + t\frac{J_1 - J^*}{\|J_1 - J^*\|_F}.
\]
Let $A = \frac{J_1 - J^*}{\|J_1 - J^*\|_F}$ and notice that it has a unit Forbenius norm and 
\begin{align*}
g'(t) = \frac{\partial \phi(J(t))}{\partial A}; \quad
g''(t) = \frac{\partial^2 \phi(J(t))}{\partial^2 A}.
\end{align*}

Thus, the problem becomes $1$-dimensional and to compare these two values, we can just use the Taylor expansion of $g$ around $J^*$. Specifically, we have that
\begin{align*}
\phi(J_1) &= g(\|J_1 - J^*\|_F) = 
g(0) + \|J_1 - J^*\|_F g'(0) + \frac{\|J_1 - J^*\|_F^2}{2} g''(\xi)\\
&= \phi(J^*) + \|J_1 - J^*\|_F  \frac{\partial \phi(J^*)}{\partial A} + \frac{\|J_1 - J^*\|_F^2}{2} \frac{\partial^2 \phi(J(\xi))}{\partial^2 A}
\end{align*}
where $\xi \in [0,\|J_1 - J^*\|_F$. Based on the previous expression, to show 
that $\phi(J_1) > \phi(J^*)$, we need an upper bound on $|\partial \phi(J^*)/\partial A|$ and also, a lower bound on $\partial^2 \phi(J)/\partial^2 A$ for all $J$ in the segment connecting $J_1, J^*$. Moreover, we would like these two bounds to be comparable to each other, so that they can be combined in the Taylor formula to prove the desired inequality.

It would be desirable to prove such bounds with high probability. From the concentration inequalities literature on Ising models, we know that such inequalities hold when the model is in high temperature, namely, $\|J^*\|_\infty < 1$ (we need to extend some of these inequalities to apply to our case). However, our assumption is just that $\|J^*\|_\infty$
is bounded by a constant, which might be greater than $1$. Hence, we need to reduce concentration in this more general case into concentration in high temperature. To do that, we first present a simple but powerful lemma that converts Ising models $\Pr_J$ with $1 < \|J\|_\infty \le C$ to models satisfying Dobrushins condition, namely $\|J\|_\infty < 1$, by conditioning on a set of nodes. While in the first setting standard concentration inequalities are not guaranteed to hold and fundamental quantities of the distribution such as the partition function are generally hard to approximate, the second setting resembles the i.i.d. scenario. While this reduction is simple and has many limitations, it suffices to show that the learning rate obtained in the constant influence regime is at least as good as the optimal rate achievable under Dobrushin's condition.
This optimality is made precise in Section~\ref{s:lower_bound}.

\begin{restatable}{lemma}{lemsubsample} \label{lem:ising-subsample}
	Let $x=(x_1,\dots,x_n)$ be an $(M,\gamma)$-Ising model, and fix $\eta \in (0,M]$. Then, there exist subsets $I_1,\dots,I_\ell \subseteq [n]$ with $\ell \le CM^2 \log n/\eta^2$ such that:
	\begin{enumerate}
		\item For all $i \in [n]$, 
		$$|\{j \in [\ell] \colon i \in I_j\}| =\lp\lceil \frac{\eta\ell}{8M} \rp\rceil
		$$.
		\item For all $j \in [\ell]$ and any value of $x_{-I_j}$, the conditional distribution of $x_{I_j}$ conditioned on $x_{-I_j}$ is an $(\eta,\gamma)$-Ising model.
	\end{enumerate}
	Furthermore, for any non-negative vector $\theta \in \mathbb{R}^n$ there exists $j \in [\ell]$ such that 
	$$
	\sum_{i \in I_j} \theta_i \ge \frac{\eta}{8M} \sum_{i=1}^n \theta_i.
	$$
\end{restatable}
The proof of Lemma~\ref{lem:ising-subsample} is presented
in Section~\ref{s:condition}.
It is a simple application of the probabilistic method and
will be used multiple times throughout this work. We can now state the two lemmas bounding the first and the second derivatives of $\phi$.

\begin{restatable}{lemma}{firstder}\label{lem:derivative-one-concentration}
Let $x$ be drawn from the distribution parametrized by $J^*$ and $A \in \R^{n\times n}$. Assume $\|J^*\|_\infty \leq M$. Let $I_1, \dots, I_\ell$ be the subsets
	obtained by Lemma~\ref{lem:ising-subsample} for
	$\eta = 1/2$. Then, there exist constants $C,c$ depending only on $M$ such that for any $t > 0$ we have
	\[
	\left|\frac{\partial\varphi(J^*)}{\partial A}\right|\leq t\left(\|A\|_F + \max_{j \in [l]} \|\E[Ax|x_{-I_j}]\|_2\right)
	\]
with probability at least
\[
1-C \log n \exp\lp(-c\min\lp(t^2,\frac{t\|A\|_F}{\|A\|_2}\rp)\rp).
\]
\end{restatable}

\begin{restatable}{lemma}{secondder}\label{lem:hessian-oneparam}
Let $x$ be drawn from the distribution parameterized by $J^*$ and $A \in \R^{n\times n}$ be a symmetric matrix with zeros on the diagonal. Assume $\|J^*\|_\infty \leq M$. Let $I_1, \dots, I_\ell$ be the subsets
	obtained by Lemma~\ref{lem:ising-subsample} for
	$\eta = 1/2$. Then, there exist constants $C,c$ depending only on $M$ such that for any $t > 0$ we have that
	\begin{align*}
	\|Ax\|_2^2 \geq c\|A\|_F^2 + c\max_{j \in [l]} \|\E[Ax\mid x_{-I_j}]\|_2^2 - t \left(\|A\|_F + \max_{j \in [l]} \|\E[Ax\mid x_{-I_j}]\|_2\right)
	\end{align*}
	with probability at least
	\begin{align*}
	1 - C \log n \exp\left(-c\frac{\min(t^2,t\|A\|_F)}{\|A\|_2^2}\right).
	\end{align*}
	Consequently, for any $J \in \mathcal{J}$ we have 
	\[
	\Pr\left[\frac{\partial^2 \phi(J)}{\partial^2 A} < c'\|A\|_F^2 + c'\max_{j \in [l]} \|\E[Ax\mid x_{-I_j}]\|_2^2\right]
	\le C \log n \exp\left(-c\frac{\|A\|_F^2}{\|A\|_2^2}\right).
	\]
\end{restatable}

In both of those Lemmas, we use the technique of conditioning on subsets $I_i$ defined in Lemma~\ref{lem:ising-subsample}.
Essentially, Lemma~\ref{lem:hessian-oneparam} states that with high probability, $\|Ax\|_2^2$ is lower bounded by some specific quantity. The same quantity appears as an upper bound for the first derivative in Lemma~\ref{lem:derivative-one-concentration}. This is not a coincidence. Indeed, we will easily show in the following sections that $\|Ax\|_2^2$ is a lower bound for $\partial^2 \phi(J)/ \partial^2 A$ for all matrices $J$ with bounded infinity norm. This means that Lemma~\ref{lem:hessian-oneparam} is essentially a lower bound of the second derivative. Thus, combining Lemmas~\ref{lem:derivative-one-concentration} and ~\ref{lem:hessian-oneparam}, we can prove Lemma~\ref{lem:one-element}. This is done in Section~\ref{s:additional}.

Notice that the derivative is upper bounded by a value which is not constant, but it is rather bounded in terms of conditional expectations with respect to $x_{-I_j}$. Similarly, the second derivative is lower bounded in terms of the same quantity. Since in the Taylor expansion we are interested in the difference between the second and first derivatives, such a non-constant bound suffices to derive Lemma~\ref{lem:one-element}. Furthermore, one might not, in general, replace these bounds with constant bounds, since the term $\max_{j \in [\ell]} \|\E[A x|x_{-I_j}]\|_2$ might not concentrate outside of Dobrushin's condition, and, $\partial\varphi(J^*)/\partial A$ and $\partial^2\varphi(J)/\partial A^2$ will fluctuate along with it.
Next, we present the proofs of Lemmas ~\ref{lem:derivative-one-concentration} and ~\ref{lem:hessian-oneparam}.

\subsubsection{Completing the proof for multiple $J$'s}\label{s:multiple}

We will now explain how to use Lemma~\ref{lem:one-element} to conclude the proof of Theorem~\ref{t:general}.
First of all, we know that the value of $\hat{J}$ is smaller than $\phi(J^*)$, since the algorithm minimizes the negative log-pseudolikelihood $\phi$.
On the other hand, by Lemma~\ref{lem:one-element}, we know that if $J$ is far from $J^*$, then with high probability it's value $\phi(J)$ will be significantly larger than $\phi(J^*)$.  And the further $J$ is, the higher is the probability.
If we could prove that this holds with high probability for all points that are far from $J^*$, then we would be able to conclude that $\hat{J}$ is close to $J^*$. 
This suggests the following plan of attack: we should pick an $R>0$ such that with high probability, all points $J$ with $\|J-J^*\|_F\geq R$ satisfy $\phi(J) > \phi(J^*)$.
This would imply that $\|J^* - \hat{J}\|_F \leq R$.

Suppose $\mathcal{A}_R$ is the set of points $J$ with $\|J-J^*\|_F\ge R$.
One obstacle in proving such a statement is that there is possibly an uncountable amount of matrices $J$ that are $R$ far from $J^*$. Hence, a simple union bound over these matrices would yield meaningful result. A common way to reason about uncountable families of objects is to define an $\epsilon$-net on this set of objects. An $\epsilon$-net of the set $\mathcal{A}_R$ is a finite subset of it, such that any point in $\mathcal{A}_R$ is $\epsilon$-close to some point in the net. 
Clearly, a small $\epsilon$ means that the size of the $\epsilon$-net will be larger. Suppose $N$ is the size of the $\epsilon$ net. As a first step in our proof, 
we could try to prove that for all points $J$ in the net, $\varphi(J)>\varphi(J^*)+\Omega(R)$. By Lemma~\ref{lem:one-element} and a simple union bound, this happens with probability at least
$$
1 - N \log n e^{-cR^2}
$$
Notice that a larger $R$ implies a higher probability that our claim is true. On the other hand, $R$ will be our final bound for $\|J^* - \hat{J}\|_F$, so we would like to make it as small as possible. 
It is clear that for a sufficiently large probability of success we need $R = \Omega(\sqrt{\log N +\log\log n})$.

Suppose we chose such an $R$. Notice that we have only shown the desired claim for the points in the net. Our hope is that since the remaining points are close to the points in the net, their values will also be close. This property is characterized by the Lipschitzness of the function $\phi$. Using some straightforward computations, we bound the Lipschitzness of $\phi$ by $n$ in Section~\ref{s:additional}. 
This means that if $\epsilon = O(R/n)$, then all points in $\mathcal{A}_R$ will have a larger value of $\varphi$ than $J^*$. Thus, the size $N = N(\mathcal{A}_R,\|\cdot\|_2,R/n)$ of the net will depend on $R$. Thus, in order to design a net that takes advantage of the Lipschitzness and is guaraneed to work with high probability, we should pick an $R$ such that
$$
R \geq \sqrt{\log N(R/n)}.
$$
Since we would like to show that $\|J^* - \hat{J}\|_F$ is small, we would like to choose the smallest such $R$. Thus, our final error rate will be of the form
$$
\inf\{R : R \geq \sqrt{\log N(\E,\|\cdot\|_2,R/n)+\log\log n}\}
$$
This is exactly the guarantee that Theorem~\ref{t:general} gives us. 

\paragraph{Section organization.}
Section~\ref{s:derivative} contains the proof of Lemma~\ref{lem:derivative-one-concentration} that provides
an upper bound for the derivative of the MPLE.
In Section~\ref{s:second_der} we present
the bound on the second derivative of the MPLE(Lemma~\ref{lem:hessian-oneparam}). 
Lastly, in Section~\ref{s:additional} we give a proof of Lemma~\ref{lem:one-element}
that compares the value of a single $J$ to $J^*$ and conclude with the proof of Theorem~\ref{t:general}.
 Note that Lemma~\ref{lem:ising-subsample} will be presented in Section~\ref{s:condition}.

% Section~\ref{s:condition} contains the self-contained proof
% of Lemma~\ref{lem:ising-subsample}, which is of central
% importance and is used
% throughout the proof. 
% Sections~\ref{s:concentration} and ~\ref{s:conc_aux}
% have a collection of concentration results for
% general Boolean functions in the weak dependence regime,
% which follow directly from the existing literature. 
% Finally, in Section~\ref{sec:pr-lem-main} 
% we describe the algorithm used and prove Lemma~\ref{lem:optimization} regarding its performance.
% We also give a detailed proof of Lemma~\ref{lem:main-general}
% that connects the bounds on the derivatives
% with the estimation rate.
\subsection{Bounding the Derivative of Log Pseudo-Likelihood}\label{s:derivative}
%\subsubsection{Outline of the proof}
The central goal of this section is to prove Lemma~\ref{lem:derivative-one-concentration}.
The general form of the derivative is a sum of $\tanh$ functions. The first thing one notices is that
its mean is $0$. Hence, it is enough to prove a strong
enough concentration bound for this quantity. 
There are many challenges with this approach, which we 
will now present.

First, most concentration bounds that are known hold for Ising
models that satisfy Dobrushin's condition. However,
our model may or may not be in this state.
Thus, we utilize Lemma~\ref{lem:ising-subsample} to find
a small number of sets $I_j \subseteq [n]$, such that each
$i \in [n]$ belongs to a small number of these sets.
This allows us to write the derivative sum as a sum
over all the sets $I_j$, thus focusing on the behavior of the
function in each set. 
Lemma~\ref{lem:ising-subsample} also guarantees that
if we condition on $x_{-I_j}$, the resulting
Ising model will satisfy Dobrushin's condition. 
Hence, the problem reduces to bounding the terms of the sum belonging to $I_j$ when we condition on $x_{-I_j}$.

This brings us to the next challenge, which involves
the concentration bound. Most of the existing results on concentration inequalities for the Ising model focus on
the case where the function is a multilinear polynomial.
Since this is not the case for the derivative, 
we introduce a simple modification to the already existing techniques so that we obtain the desired
concentration in our case. This will give us a bound
that depends on the conditional expectation of a quadratic
form conditioned on $x_{-I_j}$. This might seem insufficient at first, since we are not getting a uniform bound
for the derivative, but rather one that depends on
the values of $x_{-I_j}$.
However, an analogous lower bound is proven for the strong convexity in Lemma~\ref{lem:hessian-oneparam}. Hence,
the two bounds match in all instantiations of $x_{-I_j}$,
allowing us to complete the proof for the concentration of the derivative in a single direction $A $.

\subsubsection{Proof of Lemma~\ref{lem:derivative-one-concentration}}
First, we decompose the derivative according to terms corresponding to the sets $I_1,\dots,I_\ell$ as specified by
Lemma~\ref{lem:ising-subsample} for $\eta = 1/2$. Recall that each element $i \in [n]$ appears in exactly
$\ell'=\lceil \eta \ell/8\rceil$ sets:
\begin{equation} \label{eq:3}
\left| \frac{\partial\varphi(J^*)}{\partial A} \right|
= \frac{1}{2}\left|\sum_{i\in [n]} A_i x(x_i - \tanh(J^*_i x))\right|
\le \frac{1}{\ell'} \sum_{j \in [\ell]} \left|\sum_{i \in I_j} A_i x(x_i - \tanh(J^*_i x))\right|
:= \frac{1}{\ell'} \sum_{j \in [\ell]} |\psi_j(x;A)|.
\end{equation}
We will bound the terms $\{\psi_j(x;A)\colon j \in [\ell]\}$ separately, and show that each term concentrates around zero. In order to do so, we will show that conditioned on any value for $x_{-I_j}$, $\psi_j(x;A)$ concentrates around zero, while the radius of concentration can depend on the specific value of $x_{-I_j}$. First, we would like to claim that this term is conditionally zero mean:
\begin{claim} \label{cla:expectation-zero}
	For any $j \in [\ell]$,
	$
	\E\left[\psi_j(x;A) ~\middle|~ x_{-I_j}\right] = 0.
	$
\end{claim}
\begin{proof}
	First, fix $i \in I_j$, and notice that since $A$ and $J^*$ have zeros on the diagonal, both $A_ix$ and $J_ix$ are constant conditioned on $x_{-i}$, hence 
	\[
	\E[A_ix(x_i-\tanh(J_ix))|x_{-i}] = A_ix(\E[x_i|x_{-i}] - \tanh(J^*_ix)) = 0,
	\]
	where the last equality follows from definition of the Ising mode.
	Next, notice that
	\begin{align*}
	\E\left[\psi_j(x;A) ~\middle|~ x_{-I_j}\right]
	&= \sum_{i \in I_j} \E\left[A_ix(x_i-\tanh(J^*_ix)) ~\middle|~ x_{-I_j} \right]\\
	&= \sum_{i \in I_j} \E_{x_{I_j}}\lp[\E\left[A_ix(x_i-\tanh(J^*_ix)) ~\middle|~ x_{-i} \right]\rp]=0.
	\end{align*}
\end{proof}
% \begin{remark}
%     In fact, Claim~\ref{cla:expectation-zero} holds in more generality: if $\{P_{\theta} \colon \theta \in \Theta\}$ is a parametric family of distributions with $\Theta \subseteq \mathbb{R}$, then it holds that
%     \[
%     \E_x\left[\frac{dP_{\theta}[x]}{d\theta}\right] = 0
%     \]
%     under standard regularity assumptions, which implies that 
% \end{remark}

Next, we will bound the radius of concentration of $\psi_j(x;A)$ conditioned on $x_{-I_j}$, and show that it is roughly bounded by $O\lp(1 + \lp\|\E\lp[Ax~\middle|~x_{-I_j}\rp]\rp\|_2\rp)$. Additionally, we derive concentration inequalities for the above term.
In order to achieve this task, we utilize Lemma~\ref{lem:ising-subsample}, which states that $x_{I_j}$ is conditionally Dobrushin, conditioned on $x_{-I_j}$. Hence, we can utilize concentration inequalities for Dobrushin random variables, and achieve the following bound:
\begin{restatable}{lemma}{lemmaderivativeconcentration} \label{lem:derivative-concentration}
	For any symmetric matrix $A$ with zeros on the diagonal, we have
	\[
	\Pr\left[|\psi_j(x;A)| \ge t ~\middle|~ x_{-I_j} \right]
	\le \exp\left(-c\min\left(\frac{t^2}{\|\E\lp[Ax~\middle|~x_{-I_j}\rp]\|_2^2+\|A\|_F^2}, \frac{t}{\|A\|_2}\right)\right).
	\]
\end{restatable}
Most existing concentration bounds for the Ising model 
focus on multilinear polynomials, which is obviously not
the form $\psi_j$ has.
Hence, in order to prove this bound, we modified slightly
the existing proof of Theorem~\ref{t:note} to present a concentration inequality for general functions of ising models satisfying Dobrushin's condition, presented in Section~\ref{s:concentration}.
The concentration radius we get depend on bounds of the first and second discrete derivatives of the function. 
The full proof of Lemma~\ref{lem:derivative-concentration} can be found in Section~\ref{s:conc_aux}.

Combining Lemma~\ref{lem:derivative-concentration} with \eqref{eq:3}, we obtain the proof of Lemma~\ref{lem:derivative-one-concentration}, which we restate here for convenience.
\firstder*

\begin{proof}
	For any $j \in [\ell]$, Lemma~\ref{lem:derivative-concentration} implies that
	\[
	\Pr\left[|\psi_j(x;A)| \ge t \lp(\|A\|_F + \lp\|\E[Ax|x_{-I_j}]\rp\|_2\rp)\right]
	\le \exp\lp(-c \min\lp( t^2, \frac{t \|A\|_F}{\|A\|_2}\rp)\rp).
	\]
	By a union bound over $j \in [\ell]$, with probability at least $1-C\log n\exp(-c \min( t^2, t \|A\|_F/\|A\|_2))$, 
	\[
	|\psi_j(x;A)| \le t (\|A\|_F + \|\E[Ax|x_{-I_j}]\|_2)
	\]
	holds for all $j \in [\ell]$. By
	\eqref{eq:3}, whenever this holds, we also have
	\[
	\left|\frac{\partial\varphi(J^*)}{\partial A}\right|
	\le \frac{1}{\ell'} \sum_{j \in [\ell]} |\psi_j(x;A)|
	\le \frac{\ell}{\ell'} t \left(\|A\|_F + \max_{j\in[\ell]}\|\E[Ax|x_{-I_j}]\|_2\right).
	\]
	Since $\ell/\ell'$ is at most a constant, the proof
	is complete.
\end{proof}
\subsection{Strong Convexity of Log Pseudo-Likelihood}\label{s:second_der}
The central goal of this section is to prove Lemma~\ref{lem:hessian-oneparam}.
We break it into two parts.
Section~\ref{s:strong_conv} deals with the
overall proof of the lemma, while Section~\ref{subsec:exp-bound} contains 
an auxiliary lemma utilized in the proof.
The approach here is quite similar
to the one for the derivative. However, the specific tools
differ, since this is an anticoncentration result.
The argument begins by lower bounding the second derivative
by a quadratic form depending only on the direction $A$ of the derivative, namely $\|Ax\|_2^2$. 
Then, the problem reduces to showing anticoncentration for
this quadratic form. 
This will be accomplished by establishing two claims.

First, we show that the mean
of this form is bounded away from $0$ by a quantity 
that depends on the Frobenius norm of $Ax$, conditional on the values $x_{-I_j}$.
This would trivially be true if the spins were independent.
In the case of the Ising model, we use Lemma~\ref{lem:ising-subsample} to find a subset of the nodes
that has a small Dobrushin constant. This means that these
nodes will be weakly correlated conditional on the rest, hence
close to independent.
This translates to the covariance matrix being diagonally
dominant, from which the claim follows.

Second, we need to show that $\|Ax\|_2^2$ concentrates
around its mean in a radius that is of the same or less order
that the mean, conditional on $x_{-I_j}$.
This concentration will be easier than the one obtained
for the derivative, since we are now dealing with
a quadratic function, to which known concentration 
results apply \cite{adamczak2019note}.
Therefore, we have shown that the second derivative
in a particular direction $A$ is sufficiently large.

\subsubsection{Proof of Lemma~\ref{lem:hessian-oneparam}}\label{s:strong_conv}

The first step will be to lower bound the
second derivative by a 
 quadratic form. This way, the task of proving strong convexity becomes significantly easier. For all $J \in \mathcal{J}$ and any direction $A$, we obtain that
\begin{align} \label{eq:second_der_lower}
4\frac{\partial^2 \varphi(J)}{\partial A^2}
&= \sum_{i=1}^n 
(A_ix)^2 \sech^2(J_ix)
\ge \sum_{i=1}^n (A_i x)^2 \sech^2(\|J x\|_\infty)\notag\\
&= \lp\|Ax\rp\|_2^2 \sech^2(\|J x\|_\infty)
\ge \|A x\|_2^2 \sech^2(\|J\|_\infty \|x\|_\infty)\\
&\ge \|A x\|_2^2 \sech^2(M),\label{eq:40}
\end{align}
where we used the facts that $\sech x \ge \sech y$ whenever $|x| \le |y|$ and $\|J\|_\infty \le M$ for all $J \in \mathcal{J}$. Notice that the bound we get only depends on the direction $A$ of the derivative and not on the point $J$ where
we are calculating it. Hence, any bounds we manage to prove
on this quantity will hold for all $J \in \mathcal{J}$. Our goal is to show that with high probability $\|Ax\|_2^2$ is sufficiently large.

The general strategy to showing this anticoncentration property
will be to bound it's expectation away from $0$ and then
show that the function concentrates well around that value.
We now focus on the first goal.
 Since $\|Ax\|_2^2$ is a sum of squares of linear functions, we proceed with a lower bound on the second moment of any linear function. What we need essentially is a variance lower bound. Therefore, we have the following lemma.
\begin{lemma} \label{lem:expectation-prob}
	Let $y$ be an $(M,\gamma)$-Ising model. Then, for any vector $a \in \mathbb{R}^m$,
	\[
	\Var(a^\top x) \ge \frac{c \gamma^2\|a\|_2^2}{M}.
	\]
\end{lemma}
The proof is presented in Subsection~\ref{subsec:exp-bound}. Here now give a short outline of it.
\begin{proof}[Proof sketch.]
First, let's examine what would happen if all the coordinates of $x$ were independent, each coordinate having a variance of at least $\gamma$. Then, 
\[
\Var(a^\top x) = \sum_i \Var(a_i x_i) = \sum_i a_i^2 \Var(x_i) \ge \gamma \|a\|_2^2.
\]
We would like to apply the same logic in our situation.
Specifically, we would like the coordinates of $x$ to be as close to independent as possible. Therefore, we will use Lemma~\ref{lem:ising-subsample} to find a set of coordinates $I$ such that $x_I$ is an $(\alpha,\gamma)$-Ising model conditioned on any value for $x_{-I}$, where $\alpha = c\gamma$. In this regime, the interactions between the nodes are weak, resulting in the covariance matrix of $x_I$ conditioned on $x_{-I}$ being diagonally dominant. This allows us to derive the same variance bound as if $x_I$ was independent:
\[
\Var\lp[a^\top x~\middle|~ x_{-I}\rp]
= \Var\lp[a_I^\top x_I~\middle|~x_{-I}\rp]
\ge c\gamma \|a_I\|_2^2.
\]
Lemma~\ref{lem:ising-subsample} guarantees that we can select the set $I$ such that $\|a_I\|_2^2 \ge \alpha/(8M) \|a\|_2^2$, by substituting $\theta_i$ with $a_i^2$ in the lemma. Hence,
\[
\Var\lp[a^\top x~\middle|~x_{-I}\rp] \ge c\gamma\|a_I\|_2^2 
\ge \frac{c' \gamma \alpha\|a\|_2^2}{M} 
\ge \frac{c'' \gamma^2\|a\|_2^2}{M}.
\]
Finally, since conditioning can only decrease the conditional variance on expectation,
the same bound holds for $\Var[a^\top x]$.
\end{proof}
Now, we would like to apply Lemma~\ref{lem:expectation-prob}
to obtain a lower bound for the expectation of $\|Ax\|_2^2$.
A simple application of the Lemma shows that: 
\begin{align} \label{eq:37}
\E[\|Ax\|_2^2] 
&= \sum_{i} \E[(A_ix)^2] 
\ge \sum_{i} (\Var[(A_ix)] + \E[A_ix]^2)\\
&\ge \sum_i \Omega(\|A_i\|_2^2) + \|\E[Ax]\|_2^2
= \Omega(\|A\|_F^2) + \|\E[Ax]\|_2^2.
\end{align}
Recall, however, that our goal in Lemma~\ref{lem:hessian-oneparam} is to lower bound the second derivative of $\varphi$ in terms of the conditional expectation of $Ax$, conditioned on $x_{-I_j}$, for $j \in [\ell]$. The following lemma presents a
variant of \eqref{eq:37} when we condition on $x_{-I_j}$.
Its proof is an application of Lemma~\ref{lem:ising-subsample} along with some simple
calculations with conditional expectations.

\begin{lemma} \label{lem:bnd-expectation}
	Fix $A \in \mathcal{A}$. For any $I_j$ and any $x_{-I_j}$, it holds that
	\[
	\E\lp[\|Ax\|_2^2 ~\middle|~ x_{-I_j}\rp]
	\ge c\gamma^2 \|A_{\cdot I_j}\|_F^2  + \lp\|\E\lp[Ax~\middle|~x_{-I_j}\rp]\rp\|_2^2.
	\]
	Furthermore, there exists some $j \in [\ell]$ such that 
	$$
	\E\lp[\|Ax\|_2^2 ~\middle|~ x_{-I_j}\rp] \ge c\gamma^2/M
	$$
	with probability $1$.
\end{lemma}
\begin{proof}
	Since conditioned on $x_{-I_j}$, $x_{I_j}$ is a $(1/2,M)$ Ising model, we derive by Lemma~\ref{lem:expectation-prob} that
	\begin{align*}
	\E\lp[\lp\|Ax\rp\|_2^2 | x_{-I_j}\rp]
	&= \sum_{i=1}^n \E\lp[(A_ix)^2|x_{-I_j}\rp]
	= \sum_{i=1}^n \lp(\E\lp[A_ix|x_{-I_j}\rp]^2 + \Var\lp[A_ix|x_{-I_j}\rp]\rp) \\
	&= \lp\|\E\lp[Ax|x_{-I_j}\rp]\rp\|_2^2 + \sum_{i=1}^n \Var\lp[A_{i,I_j}x_{I_j}|x_{-I_j}\rp]
	\ge \lp\|\E\lp[Ax|x_{-I_j}\rp]\rp\|_2^2 + \sum_{i=1}^n c\gamma^2\lp\|A_{i,I_j}\rp\|_2^2\\
	&\ge \lp\|\E\lp[Ax|x_{-I_j}\rp]\rp\|_2^2 + c\gamma^2\lp\|A_{\cdot I_j}\rp\|_F^2.
	\end{align*}
	To prove the second part of the lemma, it suffices to show that there exists $j\in [\ell]$ such that $\|A_{\cdot I_j}\|_F^2 \ge c' \|A\|_F^2/M = c'/M$. 
	Recall that the sets $\{I_j\}_{j \in [\ell]}$ were obtained from Lemma~\ref{lem:ising-subsample} with $\eta = 1/2$. By the last part of this lemma, if we substitute $\theta_i =\|A_{\cdot i}\|_2^2$, we derive that there exists a $j \in [\ell]$ such that
	\[
	\|A_{\cdot I_j}\|_F^2
	= \sum_{i \in I_j} \theta_i
	\ge \frac{c}{M} \sum_{i \in [n]} \theta_i
	= \frac{c \|A\|_F^2}{M}
	= \frac{c}{M},
	\]
	recalling that $\|A\|_F=1$ for all $A \in \mathcal{A}$.
\end{proof}
According to Lemma~\ref{lem:bnd-expectation}, if
we can show that $\|Ax\|_2^2$ is concentrated
at a radius of 
$$O\lp(\lp\|A_{I_j}\rp\|_F^2 + \lp\|\E[Ax]|x_{-I_j}\rp\|_2^2\rp)$$
around its mean, conditional
on $x_{-I_j}$, then anticoncentration follows.
The next Lemma contributes to the proof in this direction.
\begin{lemma}\label{lem:conc-mainpart}
	Fix symmetric matrix $A$ with zeros on the diagonal and $j \in [\ell]$. Then, for any fixed value of $x_{-I_j}$ and
	any $t>0$
	\[
	\Pr\left[\lp\|Ax\rp\|_2^2 < \E\left[\|Ax\|_2^2~\middle|~x_{-I_j}\right] - t ~\middle|~ x_{-I_j}\right]
	\le \exp\left(-\frac{c}{\|A\|_2^{2}}\min\left(\frac{t^2}{\|A\|_F^2+\|\E\lp[Ax~\middle|~x_{-I_j}\rp]\|_2^2},t\right)\right).
	\]
\end{lemma}
The proof is given in Section~\ref{s:conc-mainpart}.
Note that by Lemma~\ref{lem:ising-subsample} we know that $x_{I_j}$ is conditionally Dobrushin, conditioned on $x_{-I_j}$. The reason why Lemma~\ref{lem:conc-mainpart}
does not directly follow from the concentration inequality for quadratic forms (Theorem ~\ref{t:note}) is that the matrix
$A^\top A$ does not necessarily have zeroes on the diagonal.
Hence, the proof consists of a simple argument that reduces to the concentration of a polynomial of the form $x^\top \tilde{A}x$ where $\tilde{A}$ has zeros on the diagonal.

By combining Lemmas~\ref{lem:bnd-expectation} and ~\ref{lem:conc-mainpart}, we are now able to prove
that $\|Ax\|_2^2$ is lower bounded with high probability,
conditioned on $x_{-I_j}$.
This allows us to complete the proof of Lemma~\ref{lem:hessian-oneparam}.
\secondder*

\begin{proof}
	Let $E_j$ denote the event that 
	\begin{equation}\label{eq:332}
	\|Ax\|_2^2 > \E[\|Ax\|_2^2\mid x_{-I_j}] - t \left(\|A\|_F + \|\E[Ax\mid x_{-I_j}]\|_2\right).
	\end{equation}
	By Lemma~\ref{lem:conc-mainpart}, 
	\[
	\Pr[E_j]
	= \E_{x_{-I_j}}\left[\Pr[E_j \mid x_{-I_j}]\right]
	\ge 1- \exp\left(-c \min(t^2,t\|A\|_F)/\|A\|_2^2\right).
	\]
	By a union bound, $\Pr\left[\bigcap_{j}E_j\right] \ge 1 - C \log n \exp\left(-c \min(t^2,t\|A\|_F)/\|A\|_2^2\right)$.
	
	From now onward, assume that $\bigcap_j E_j$ holds. Lemma~\ref{lem:bnd-expectation} implies that 
	\[
	\E\left[\|Ax\|_2^2\mid x_{-I_j}\right] 
	\ge c\|A_{\cdot I_j}\|_F^2  + \|\E[Ax|x_{-I_j}]\|_2^2
	\]
	and additionally, that there exists some $j$ such that $\|A_{\cdot I_j}\|_F^2 \ge c\|A\|_F^2$. We derive that 
	\begin{equation}\label{eq:333}
	\|Ax\|_2^2 > c\|A_{\cdot I_j}\|_F^2  + \|\E[Ax|x_{-I_j}]\|_2^2 - t \left(\|A\|_F + \|\E[Ax\mid x_{-I_j}]\|_2\right)
	\end{equation}
	holds for all $j$. Let $j_1 = \max_j \|\E[Ax\mid x_{-I_j}]\|_2$ and $j_2 = \max_j \|A_{\cdot I_j}\|_F$. Then, substituting $j_1$ in \eqref{eq:333}, we derive that
	\[
	\|Ax\|_2^2 > \max_j \|\E[Ax\mid x_{-I_j}]\|_2^2 - t \left(\|A\|_F + \max_j\|\E[Ax\mid x_{-I_j}]\|_2\right).
	\]
	By substituting $j_2$ in \eqref{eq:333} we derive that
	\[
	\|Ax\|_2^2 > c'\|A\|_F^2  - t \left(\|A\|_F + \max_j\|\E[Ax\mid x_{-I_j}]\|_2\right).
	\]
	Taking an average of the above two inequalities, we derive that
	\[
	\|Ax\|_2^2 > c'\|A\|_F^2/2 + \max_j \|\E[Ax\mid x_{-I_j}]\|_2^2/2 - t \left(\|A\|_F + \max_j\|\E[Ax\mid x_{-I_j}]\|_2\right).
	\]
	This concludes the proof for the first part of the lemma. For the second part, we substitute $t = c''\|A\|_F$ and use the inequality\eqref{eq:second_der_lower} for a lower bound on the second derivative. 
\end{proof}

\subsubsection{Proof of Lemma ~\ref{lem:expectation-prob}} \label{subsec:exp-bound}
We begin by proving Lemma~\ref{lem:expectation-prob}.
The idea is that if the total interaction strength of a node
with all of its neighbors is small compared to the variance
of each node, then the behavior is similar to that of independent
samples. To prove that, we employ the following Lemma, the 
proof of which can be found in \cite[Lemma~4.9]{dagan2019learning}, and is also a direct corollary of the earlier paper \cite[Theorem~3.1]{bresler2019stein}.
	\begin{lemma}\label{l:wasser}
		Let $x$ be an $(M,\gamma)$ Ising model with $M < 1$. Fix $i \in [n]$ and let $P_{x_{-i}|x_i=1}$ denote the conditioned distribution over $x_{-i}$ conditioned on $x_i = 1$ and $P_{x_{-i}|x_i=-1}$ denote the conditioned distribution conditioned on $x_i = -1$. Then,
		\[W_1(P_{x_{-i}|x_i=1},P_{x_{-i}|x_i=-1}) \le \frac{2M}{1-M},\]
		where $W_1$ is the $\ell_1$-Wasserstein distance, namely
		$
		W_1(P,Q)=\min_{\pi} \E_{(x,y)\sim \pi}[\|x-y\|_1].
		$
		The minimum is taken over all distributions $\pi$ over $\{-1,1\}^n \times \{-1,1\}^n$, such that the marginals are $P$ and $Q$, respectively.
	\end{lemma}
This lemma essentially tells us that if $M$ is small enough,
then changing the spin of a single node is unlikely to influence
the remaining ones. 
Using this fact, we can proceed as in the proof sketch of Section~\ref{s:strong_conv} to prove the claim.
To prove it in the general case, we employ once more the
conditioning trick to reduce the Dobrushin constant
of the model. Specifically, Lemma~\ref{lem:ising-subsample} guarantees
that we can find a subset $I$ of nodes that
are conditionally very weakly dependent, while still
maintaining a constant fraction of the total variance
of the linear function.
The details are given below.
\begin{proof}[Proof of Lemma~\ref{lem:expectation-prob}]
	Fix $a \in \mathbb{R}^d$. 
	We first assume that $M \le \gamma/4$, and then we prove it for all $M$. We start by arguing that for all $i\in [n]$, 
	$$
	\sum_{j \in [n]\setminus \{i\}}
	|\mathrm{Cov}(x_i,x_j)|
	\le \frac{M}{1-M}.
	$$
	The strategy will be to bound the Wasserstein
	distance between two Ising models
	conditional on different values of a single node.
    Using a tensorization argument, we have that if $X=(X_1,\dots, X_n), Y=(Y_1,\dots,Y_n)$ are random vectors over the same domain, then $W_1(X,Y) \ge \sum_{i \in [n]} W_1(X_i, Y_i)$. Applying this on the distributions $P_{x_{-i} | x_i=1}$ and $P_{x_{-i} | x_i=-1}$, we derive
    by Lemma~\ref{l:wasser} that
	\begin{align*}
	\sum_{j \in [n]\setminus \{i\}}
	\lp|\E[x_j | x_i=1] - \E[x_j | x_i=-1]\rp|
	&=\sum_{j \in [n]\setminus \{i\}} W_1(P_{x_j | x_i=1}, P_{x_j|x_i=-1})\\
	&\le W_1(P_{x_{-i} | x_i=1}, P_{x_{-i} | x_i=-1})
	\le \frac{2M}{1-M},
	\end{align*}
	where we also used the easily established property that for any two random
	variables $U,V$ supported in a set of size $2$,
	$W_1(U,V) = |\E[U] - \E[V]|$.
	We will use the above bound to get a bound on $\mathrm{Cov}(x_i,x_j)$. Indeed,
	\begin{align*}
	\mathrm{Cov}(x_i,x_j)&=
	\E[(x_i-\E x_i)x_j]
	= \E_{x_i}[(x_i-\E x_i)\E_{x_j}[x_j|x_i]]\\
	&= \Pr[x_i = 1] (1 - \E x_i) \E[x_j| x_i=1]
	+\Pr[x_i = -1] (-1 - \E x_i) \E[x_j| x_i=-1]\\
	&= \frac{1}{2}(1+\E[x_i]) (1 - \E x_i) \E[x_j| x_i=1]
	+\frac{1}{2}(1-\E[x_i]) (-1 - \E x_i) \E[x_j| x_i=-1] \\
	&= \frac{1}{2}(1 - \E [x_i]^2)(\E[x_j| x_i=1]- \E[x_j| x_i=-1]).
	\end{align*}
	Combining the above inequalities, we obtain that
	\[
	\sum_{j \in [n]\setminus \{i\}}
	|\mathrm{Cov}(x_i,x_j)|
	\le \frac{\sum_{j \in [n]\setminus \{i\}}
	|\E[x_j| x_i=1]- \E[x_j| x_i=-1]|}{2}
	\le \frac{M}{1-M}.
	\]
	To conclude the proof for the setting where $M \le \gamma/4$:
	\begin{align*}
	\Var(a^\top x)
	&= \E[a^\top (x-\E x) (x-\E x)^\top a]
	= \sum_{i,j} a_{ij} \mathrm{Cov}(x_i,x_j)
	\ge \sum_{i=1}^n a_i^2 \Var(x_i)
	- \sum_{i \ne j} |a_i a_j \mathrm{Cov}(x_i,x_j)| \\
	&\ge \sum_{i=1}^n a_i^2 \Var(x_i)
	- \sum_{i \ne j} \frac{(a_i^2 + a_j^2)|\mathrm{Cov}(x_i,x_j)|}{2}
	= \sum_{i =1}^n a_i^2 \left( \Var(x_i) - \sum_{j \ne i} |\mathrm{Cov}(x_i,x_j)| \right)\\
	&\ge \sum_{i =1}^n a_i^2 \left( \gamma - \frac{M}{1-M} \right)
	\ge \|a\|_2^2 \frac{\gamma}{2},
	\end{align*}
	where we used the fact that $M \le \gamma/4 < 1/2$, which implies that $M/(1-M) \le 2M \le \gamma/2$.
	
	In order to prove the Lemma without the above assumption, we again use the conditioning trick.
	Specifically, we find a subset of the nodes
	that satisfies Dobrushin's condition with a small
	enough $M$ and then apply our result. 
	To make this formal, notice that by Lemma~\ref{lem:ising-subsample}, there exists a subset $I$
	of nodes such that conditioned on any $x_{-I}$, $x_I$ is a $(\gamma/4,\gamma)$-Ising model. Moreover, it is guaranteed that 
	$$\sum_{i \in I} a_i^2 \ge \frac{c'\gamma}{M}\sum_{i \in [n]} a_i^2.
	$$ 
	Hence, if we apply the previous result when
	conditioning on $x_{-I}$ we get
	\[
	\Var[a^\top x | x_{-I}]
	= \Var[a_I^\top x_I | x_{-I}]
	\ge \frac{\gamma \|a_I\|_2^2}{2}
	\ge \frac{c\gamma^2 \|a\|_2^2}{M}.
	\]
	Since conditioning decreases the variance in expectation, we have that
	\[
	\Var[a^\top x]
	\ge \E_{x_{-I}}[\Var[a^\top x | x_{-I}]]
	\ge \frac{c\gamma^2\|a\|_2^2}{M},
	\]
	as required.
\end{proof}

\subsection{Proof of Lemma~\ref{lem:one-element} and Theorem~\ref{t:general}\label{s:additional}}

The goal of this section is to prove the main Theorem~\ref{t:general}. We do so in two steps. First, we use Lemmas~\ref{lem:derivative-one-concentration} and ~\ref{lem:hessian-oneparam} to complete the proof of Lemma~\ref{lem:one-element}.  The strategy here is to use a lower bound on the Taylor approximation of $\phi$ around $J_1$ to prove that it's value will be larger than $J^*$. To prove the lower bound, we utilize the lower bound on $\frac{\partial^2 \phi(J)}{\partial^2 A}$ fro all $J$ and the upper bound of $\frac{\partial \phi(J^*)}{\partial A}$. 
The precise argument is carried out in the following proof.

\oneJ*
\begin{proof}
	%Let $A = J-J^*/\|J-J^*\|_F$. 
	Define the function 
	$
	J \colon [0,1] \to \mathbb{R}
	$
	by $J(t) = J^* + t (J_1-J^*)$ such that $J(0)=J^*$ and $J(1)=J_1$. Let $A = J_1-J^*$.
	Notice that
	\[
	\frac{d \varphi(J(t))}{d t}
	= \frac{\partial \varphi(J)}{\partial A}\Bigg|_{J=J(t)};\quad
	\frac{d^2 \varphi(J(t))}{d t^2}
	= \frac{\partial^2 \varphi(J)}{\partial A^2}\Bigg|_{J=J(t)}.
	\]
	Let $\mu = \min_{t \in (0,1)} \frac{d^2 \varphi(J(t))}{d t^2}$ and let $r = \frac{d\varphi(J(t))}{dt}|_{t=0}$. Then, by the fundamental theorem of calculus,
	\[
	\frac{d\varphi(J(t))}{dt}
	\ge r
	+ t \mu.
	\]
	This implies by the fundamental theorem of calculus that
	\[
	\varphi(J(t)) \ge \varphi(J(0)) + tr+ \frac{t^2\mu}{2}.
	\]
	Substituting $t=1$, we derive that
	\[
	\varphi(J_1) \ge \varphi(J^*) + r + \frac{\mu}{2}.
	\]
	
	From Lemma~\ref{lem:hessian-oneparam}, with probability $1-C \log n \exp\left(-c\frac{\|A\|_F^2}{\|A\|_2^2}\right)$, it holds that $\mu \ge c_0\|A\|_F^2 + c_0\max_j \|\E[Ax\mid x_{-I_j}]\|_2^2$. By applying Lemma~\ref{lem:derivative-one-concentration} with $t = O(\|A\|_F)$, we derive that with probability $1-C \log n \exp\left(-c\frac{\|A\|_F^2}{\|A\|_2^2}\right)$, $|r| \le \mu/4$. Whenever these two events hold, we have
	\[
	\varphi(J_1) - \varphi(J^*)
	\ge r + \frac{\mu}{2}
	\ge \frac{-\mu}{4} + \frac{\mu}{2}
	\ge c_0 \|A\|_F^2/4.
	\]
	This holds with probability 
	\[
	1-C \log n \exp\left(-c\frac{\|A\|_F^2}{\|A\|_2^2}\right)
	=1-C \log n \exp\left(-c\frac{\|J_1-J^*\|_F^2}{\|J_1-J^*\|_2^2}\right)
	\]
	and notice that $\|J_1-J^*\|_2 \le \|J_1\|_2 + \|J^*\|_2 \le \|J_1\|_\infty + \|J^*\|_\infty$ which is bounded by a constant.
\end{proof}
To complete the proof of Theorem~\ref{t:general}, we need the bound of Lemma~\ref{lem:one-element} to hold for all points
$J \in \mathcal{J}$ uniformly.
So far we have obtained a bound
for comparing a single $J$ with $J^*$ that holds with probability proportional
to $1 - \log n e^{-cr^2}$. Hence, by taking a union bound over a set
of points of cardinality $N$, we obtain a 
probability concentration bound of the form $1 - N\log n e^{-cr^2}$. 
Since we would like this bound to hold for all points
,
we should choose this set so that it \emph{covers}
the set $\mathcal{J}$.
A natural candidate for such a set is an \emph{$\epsilon$-net}
of $\mathcal{J}$. Intuitively, it is a set of points such that
every $J \in \mathcal{J}$ is close to at least one of
them.
Ideally, if two points are close under some notion of distance, we would like to conclude that the values of $\phi$ on these two points will also be close. This is exactly the content of the following auxiliary Lemma.

\begin{lemma}\label{lem:lip-mple}
	Let $M> 0$ and $J_1, J_2$ be two matrices. Then,
	$$
	|\phi(J_1) - \phi(J_2)| \leq n\|J_1 - J_2\|_2
	$$
\end{lemma}
\begin{proof}
	We have that
	$$\phi(J) = 
	\sum_{i=1}^n (\log\cosh(J_i x) - x_i J_i x + \log 2).
	$$
	We can define the function $g:[0,1]\mapsto \R$ by
	$$
	g(t) = \phi\lp(tJ_1 + (1-t)J_2\rp)
	$$
	Clearly, $g$ is computing the values of $\phi$ across the line segment connecting $J_1, J_2$. Hence, the desired inequality is equivalent to
	$$
	|g(1) - g(0)| \leq n \|J_1 - J_2\|_2
	$$
	Computing the derivative of $g$ will thus give us a suitable bound for this difference. By the calculations done in the previous lemmas, it is clear that
	$$
	g'(t) = \frac{\partial \phi(J)}{\partial A}\Bigg|_{J = tJ_1 + (1-t)J_2}
	$$
	where $A = J_1 - J_2$. We set $J(t) = tJ_1 + (1-t)J_2$. By the calculations done in previous parts of the paper, we have
	$$
	\frac{\partial \phi(J)}{\partial A}\Bigg|_{J = tJ_1 + (1-t)J_2} =
	\frac{1}{2} \sum_{i=1}^n \lp((J_1 - J_2)_i x \rp)\lp(\tanh(J(t)_i x) - x_i\rp)
	$$
	Using the Cauchy-Schwarz inequality, we obtain
	\begin{align*}
	\lp|\sum_{i=1}^n \lp((J_1 - J_2)_i x \rp)\lp(\tanh(J(t)_i x) - x_i\rp)\rp| &\leq \sqrt{\sum_{i=1}^n \lp((J_1 - J_2)_i x \rp)^2 \sum_{i=1}^n \lp(\tanh(J(t)_i x) - x_i\rp)^2}\\ &= 
	\|(J_1 - J_2)x\|_2 \sqrt{\sum_{i=1}^n \lp(\tanh(J(t)_i x) - x_i\rp)^2}
	\end{align*}
	Since the $\tanh$ funtion is bounded in $[-1,1]$, we have
	$$
	\sqrt{\sum_{i=1}^n \lp(\tanh(J(t)_i x) - x_i\rp)^2} \leq 2\sqrt{n}
	$$
	Also, by the infinity norm bound of $J_1,J_2$ we have
	$$
	\|(J_1 - J_2)x\|_2 \leq \|J_1-J_2\|_2 \|x\|_2 =  \sqrt{n}\|J_1-J_2\|_2
	$$
	We conclude that for any $t \in [0,1]$
	$$
	|g'(t)| \leq \frac{1}{2} 2\sqrt{n}\sqrt{n}\|J_1-J_2\|_2 =
	n\|J_1-J_2\|_2
	$$
	By the mean value theorem, this implies
	$$
	|g(1) - g(0)| \leq n \|J_1 - J_2\|_2
	$$
	and the claim is complete.
	
\end{proof}

We now use the preceding observations to conclude the proof of Theorem~\ref{t:general}. The idea is to consider the set of points $\mathcal{A}_r$ whose distance from $J^*$ is at least $r$ in Frobenius norm. We would like to show that all of these points will have a value of $\phi$ that is $r$ larger than $J^*$ with high probability.
To do this, we find an $r/n$-net of $\mathcal{A}_r$.
By taking a union bound over this set, we ensure that
the concentration bound will hold for all points in this set. Also, by the Lipschitzness argument, this implies that this holds for
all points in $\mathcal{A}_r$. The failure probability is thus $N(r/n)\log n e^{-cr^2}$. We choose an $r$ small enough to make this probability smaller than $\delta$ and this gives us the final guarantee.

\general*
\begin{proof}
	Let $r>0$.
	The idea is to find an $\epsilon$-cover $\mathcal{N}$ for the set of elements $\mathcal{J}_{\ge r} := \{J \in \mathcal{J} \colon \|J^*-J\|_F \ge r \}$. Then, apply Lemma~\ref{lem:one-element} for each element in this cover and argue by a union bound, that with high probability all elements in $\mathcal{N}$ exceed $J^*$ in the cost function $\varphi$. Then, use Lipschitzness of this function $\varphi$ to generalize from the net to $\mathcal{J}_{\ge r}$.
	
	Let $\epsilon = c'r^2/(2n)$ where $c'$ is the constant from Lemma~\ref{lem:one-element}, and we take $\mathcal{N}$ to be an $\epsilon$-cover of $\mathcal{J}_{\ge r}$ with respect to the matrix norm $\|\cdot\|_2$. 
	
	By applying Lemma~\ref{lem:one-element}, each $J \in \mathcal{N}$ satisfies with probability $1-C\log n \exp(-cr^2)$,
	\[
	\varphi(J) \ge \varphi(J^*)+c'r^2.
	\]
	By a union bound over the net, with probability at least $1- |\mathcal{N}|\log n \exp(-cr^2)$, all these points satisfy the above inequality. 
	
	Next, assume that the above holds and we generalize from the cover to all of $\mathcal{J}_{\ge r}$.
	Take any point $J \in \mathcal{J}_{\ge r}$ and let $J_0$ be any point in $\mathcal{J}$ that satisfies $\|J_0-J\|_2 \le \epsilon$. Then, we have by Lemma~\ref{lem:lip-mple} that
	\[
	\varphi(J) \ge \varphi(J_0) - n \|J-J_0\|_{2}
	\ge \varphi(J_0) - c'r^2/2
	\ge \varphi(J^*) +c'r^2/2.
	\]
	
	The failure probability is
	\[
	|\mathcal{N}| \log n \exp(-cr^2)
	= N(\mathcal{J}_{\ge r}, \|\cdot\|_2,c'r^2/(2n))\log n \exp(-cr^2)
	\le N(\mathcal{J}, \|\cdot\|_2,c'r^2/(4n))\log n \exp(-cr^2),
	\]
	using the fact that for any sets $\mathcal{U} \subseteq \mathcal{U}$, any norm $\|\cdot\|$ and any $\epsilon>0$, it holds that $N(\mathcal{U}',\|\cdot\|,\epsilon) \le N(\mathcal{U},\|\cdot\|,\epsilon/2)$.	
	
	Let $N = N(\mathcal{J}, \|\cdot\|_2,c'r^2/(4n))$ and notice that $r$ has to be sufficiently large such that 
	\[
	N \log n e^{-cr^2} \le \delta,
	\]
	which holds whenever
	\[
	\log N + \log \log n - cr^2 \le \log \delta
	\]
	or 
	\[
	r \ge C\sqrt{\log \log n + \log N + \log(1/\delta)}.
	\]
\end{proof}
\section{Applications}\label{s:applications}
We begin with applications for learning from one sample, and then move to multiple samples.

\subsection{Applications for learning from one sample}

\subsubsection{A collection of finite candidates}
Assume that $\mathcal{J}$ is a finite set. Then, we have $N(\mathcal{J},\|\cdot\|_2,0) = |\mathcal{J}|$. In particular, we immediately obtain the bound of Corollary~\ref{cor:finite}:
\[
\|\hat{J}-J^*\|_F \le C \sqrt{\log|\mathcal{J}| + \log(1/\delta)+\log \log n}.
\]

\subsubsection{A Euclidean subspace of matrices}

In this section, we assume that the matrix $J^*$ is a linear combination of $k$ known candidate matrices, $J_1,\dots,J_k$, and further, that it has bounded infinity norm. In particular, define $\mathcal{J} = \{J:\sum_{i=1}^k \beta_i J_i, \|J\|_\infty \leq M, \beta_i \in \R\}$. Using standard arguments, we can obtain bounds on the covering numbers of this set:

\begin{lemma}\label{lem:subspace-cover}
	Let $\mathcal{J} = \{J:\sum_{i=1}^k \beta_i J_i, \|J\|_\infty \leq M, \beta_i \in \R\}$, where $J_1,\ldots,J_k$ are known matrices. Then
	$$
	N(\mathcal{J}, \|\cdot\|_\infty, \epsilon) \leq \lp(1 + \frac{2M}{\epsilon}\rp)^k
	$$
\end{lemma}
\begin{proof}
	Let $B(x,r)$ be the ball of center $x$ and radius $r$ w.r.t. $\|\cdot\|_\infty$ in the subspace spanned by $J_1,\ldots,J_k$.
	Then the set $\mathcal{J}$ is equal to $B(0,M)$. Let $|V|$ denote the volume of a subset $V$ of this $k$-dimensional subspace. Suppose we cover $B(0,M)$ in the usual greedy way: each time we choose a
	ball $B(x,\epsilon/2)$ that is disjoint with the previous balls, where $x \in B(M)$  and we add it to the cover, until we cannot add any more balls. Let $N$ be the number of centers selected using this process. Then clearly any point $x$ in $B(0,M)$ is within $\epsilon$ distance from some center(otherwise the ball $B(x,\epsilon/2)$ could be added in the cover). Hence, these centers form an $\epsilon$ net of $B(0,M)$. Moreover, notice that all the selected balls are disjoint and fit inside the ball $B(0,M + \epsilon/2)$. Hence, their total volume is at most $|B(0,M+\epsilon/2)$. It follows that
	$$
	N|B(0,\epsilon/2)| \leq |B(0,M+\epsilon/2)|
	$$
	By scaling we know that \begin{align*}
	&|B(0,\epsilon/2)| =  \lp(\frac{\epsilon}{2}\rp)^k |B(0,1)|\\
	&|B(0,M + \epsilon/2)| =  \lp(M + \frac{\epsilon}{2}\rp)^k |B(0,1)|
	\end{align*}
	Notice that it doesn't matter which norm defines the balls. As long as the balls scale exponentially with the dimension, the volume does so too.
	
	Consequently, we have that
	$$
	N \leq \frac{|B(0,M+\epsilon/2)|}{|B(0,\epsilon/2)|} = \lp(\frac{M + \epsilon/2}{\epsilon/2}\rp)^k = \lp(1 + \frac{2M}{\epsilon}\rp)^k
	$$
	
\end{proof}

By applying Theorem~\ref{t:general}, we immediately obtain Corollary~\ref{cor:linear} :
\linsubspace*
\begin{comment}
\begin{theorem}\label{thm:linear}
	Let $\mathcal{J} = \{J:\sum_{i=1}^k \beta_i J_i, \|J\|_\infty \leq M, \beta_i \in \R\}$ where $J_1,\dots,J_k$ are fixed matrices. Then there exists an estimator $\hat{J}$ based on a single sample from some $J^* \in \mathcal{J}$, such that for every $\delta>0$, it holds with probability $1-\delta$ that $\|\hat{J} -J^*\|_F \le C\sqrt{k \log n + \log(1/\delta)}$, where $C$ depends only on $M$. Further, the algorithm runs in polynomial time in $k$ and $n$.
\end{theorem}
\end{comment}
\begin{proof}
	For the statistical quarantee, apply Theorem~\ref{t:general}, using the covering numbers from Lemma~\ref{lem:subspace-cover}. The optimum can be found in polynomial time due to Lemma~\ref{lem:optimization}.
\end{proof}

\subsubsection{Sparse estimation}

Here, we assume that the true interaction matrix is a linear combination of just a few matrices in the subspace. 
Hence, we are dealing with $\ell_0$ sparsity.
Denote for any $\vec{\beta} \in \mathbb{R}^k$ by $\|\vec{\beta}\|_0$ the number of nonzero coordinates of $\vec{\beta}$. The goal is to have a learning rate that is dominated by the number of nonzero components of $\beta$. This is precisely the content of the following claim.
\sparse*
\begin{comment}
\begin{corollary}
	Let $s > 0$ and let $\J = \{ J = \sum_{i=1}^k \beta_i J_i \colon \|J\|_\infty \le M, \|\vec{\beta}\|_0 \le s \}$. Then,
	\[
	\|\hat{J}-J^*\|_F \le O\lp(\sqrt{s(\log n+\log k) + \log(1/\delta)}\rp).
	\]
	and $\hat{J}$ is computed in time $\poly(n,s)\cdot \binom{k}{s}$.
\end{corollary}
\end{comment}
\begin{proof}
To prove the statistical guarantee, we need to bound the covering number of the set $\mathcal{J}$. To do that, we use the fact that $\mathcal{J}$ is a union of $\binom{n}{s}$ collections of matrices, each of them is an intersection of some $s$-dimensional subspace with the set of matrices of bounded infinity norm. By Lemma~\ref{lem:subspace-cover} the $\epsilon$ covering number of each such set $\mathcal{J}'$ is bounded by $N(\mathcal{J}', \|\cdot\|_\infty, \epsilon) \leq \lp(1 + \frac{2M}{\epsilon}\rp)^s$. Taking the union over the $\epsilon$ coverings over all the $\binom{k}{s}$ collections, we get a covering for $\mathcal{J}$ of size at most $\lp(1 + \frac{2M}{\epsilon}\rp)^s \binom{k}{s}$. Hence, by plugging this number into the general rate of Theorem~\ref{t:general} we immediately obtain the result.

For the analysis of the runtime, notice that in order to optimize the MPLE, one has to iterate over all $\binom{k}{s}$ subspaces of dimension $s$ and find the minimizer in each of those, each in polynomial time due to Corollary~\ref{cor:linear}.
\end{proof}
\if 0
\subsubsection{$\ell_1$ sparsity}

Next, we study the $\ell_1$ sparse setting, and we have the following statement:
\begin{corollary}
	Let $M>0$, let $J_1,\dots,J_k$ denote matrices with $\|J_i\|_\infty \le 1$. Define $\J = \{ J = \sum_{i=1}^k \beta_i J_i \colon \|\vec{\beta}\|_1 \le M \}$. Then,
	\[
	\|\hat{J}-J^*\|_F \le C(M)\lp(n^{1/4}\sqrt{\log k + \log(1/\delta)}\rp).
	\]
	and $\hat{J}$ is computed in time $\poly(n,k)$.
\end{corollary}

To prove the above corollary we need to apply a bound on the covering numbers of $\J$. We view $\J$ as a high dimensional manifold. Let $B_1(M) = \{ v\in\mathbb{R}^k \colon \|v\|_1 \le M \}$ and define $h \colon B_1(M) \to M_{n\times n}(\mathbb{R})$. Notice that $\J = \{ h(v) \colon v \in B_1(M) \}$. As in Section~\ref{sec:manifolds}, we can bound the covering numbers of $\J$ in terms of those of $B_1(M)$ and the Lipschitz constant of $h$. To estimate the Lipschitz constant, notice that
\[
\|h(v) - h(u)\|_2
= \| \sum_{i=1}^k V-U\|_F
\]
\fi

\subsubsection{High dimensional manifolds}\label{sec:manifolds}

There are settings where the interaction matrix lies in a nonlinear space. One such case is when the space is
 a $k$-dimensional manifold of matrices. In this case, we have some continuous mapping $h \colon D \to M_{n\times n}(\mathbb{R})$ where $D \subseteq \mathbb{R}^k$ is some domain and $\mathcal{J} = \{ h(x) \colon x \in D, \|h(x)\|_\infty \le M \}$.
We also assume that $h$ is Lipschitz. 
With these assumption, we can obtain the following guarantee.
\manifold*
\begin{proof}
Since $h$ satisfies  $\|h(x)-h(y)\|_2 \le L \|x-y\|$ for some $L>0$,
 we have that for any set $D \subseteq \R^k$:
 $$N(h(D), \|\cdot\|_2,\epsilon) \le N(D, \|\cdot\|,\epsilon/L).$$
  In particular, if $D = [0,1]^k$ and $\|\cdot \| = \|\cdot\|_\infty$ is the infinity norm over vectors in $\mathbb{R}^k$, then $N(D,\|\cdot\|_\infty, \epsilon) \le \epsilon^{-k}$. Hence, by a direct application of Theorem~\ref{t:general} we derive the following bound with probability $1-\delta$:
\[
\|\hat{J}-J^*\|_F \le C\sqrt{k \log(Ln) + \log(1/\delta)}.
\]
\end{proof}
\subsection{Multiple independent samples}

In this Section, we assume that we have access to multiple samples, either independent or correlated. An application of  Theorem~\ref{t:general} gives us nontrivial bounds.

\subsubsection{A general statement for independent samples}
We start by giving a general bound for learning from $\ell$ independent samples. The following is a slightly more general formulation of Corollary~\ref{cor:multiple-samples}.

\begin{corollary}\label{cor:iid-formal}
	Let $\J$ be a set of interaction matrices of infinity norm bounded by $M$. Assume that $\ell$ independent samples are obtained from $\Pr_{J^*}$ for some $J^*\in \J$. There is an estimator $\hat{J}$ for $J^*$ such that for all $\delta>0$, with probability $1-\delta$,
	\begin{align*}
	\|\hat{J}-J^*\|_F &\le 
	\frac{1}{\sqrt{\ell}}\min \lp\{r \ge 0 \colon r \ge C(M)\sqrt{\log N\lp(\mathcal{J}, \|\cdot\|_2, \frac{cr^2}{n\ell}\rp)
		+\log\log n + \log(1/\delta)}\rp\}\\
	&\le \frac{C(M)}{\sqrt{\ell}} \sqrt{\log N\lp(\mathcal{J}, \|\cdot\|_2, \frac{1}{n\ell}\rp)
		+\log\log n + \log(1/\delta)}.
	\end{align*}
\end{corollary}
\begin{proof}
In the setting of $\ell$ independent samples, we can view the concatenation of $\ell$ independent $n$-bit samples as a single $n\ell$-bit sample. Given an interaction matrix $J$ of size $n\times n$, denote by $h(J)$ the interaction matrix of size $n\ell\times n\ell$ that corresponds to the joint distribution over $\ell$ samples. Formally, $h(J)$ is block-diagonal, that contains $\ell$ copies of $J$ in the diagonal and zeros otherwise: $h(J)_{i+nj,i'+nj} = J_{i,i'}$ for any $i,i'\in [n]$ and $j \in \{0,\dots, n-1\}$ and $h(J)_{i,j} = 0$ otherwise. We create the collection of interaction matrices that correspond to the joint distributions over $\ell$-independent samples by $\J^\ell=\{h(J) \colon J \in \J\}$. In order to learn $J^*\in \J$ from $\ell$ samples, we will instead view it as learning a matrix $J^{\ell*} \in \J^\ell$ using a single sample. We will use this to get an estimate for $J^*$.

First of all, we bound the error of learning a matrix from $J^\ell$. Start by noticing that $\|h(J)\|_\infty = \|J\|_\infty$, hence, assuming that $\|J\|_\infty \le M$ for all $J \in \J$, we derive that $\|J^\ell\|_\infty \le M$ for all $J^\ell \in \J^\ell$. Further, $\|h(J) - h(J')\|_2 = \|J-J'\|_2$ hence for any $\epsilon > 0$, we have $N(\J,\|\cdot\|_2,\epsilon) = N(\J^\ell, \|\cdot\|_2, \epsilon)$. Let $\hat{J}^\ell$ denote the estimator of $J^{\ell*} \in \J^\ell$ and the learning rate for learning rate is bounded by
\begin{align*}
\|\hat{J}^\ell-J^{\ell*}\|_F
&\le\min \lp\{r \ge 0 \colon r \ge C(M)\sqrt{ \log N\lp(\mathcal{J}^\ell, \|\cdot\|_2, \frac{cr^2}{n\ell}\rp) + \log\log n + \log (1/\delta)}\rp\}\\
&= \min \lp\{r \ge 0 \colon r \ge C(M)\sqrt{\log N\lp(\mathcal{J}, \|\cdot\|_2, \frac{cr^2}{n\ell}\rp)
+\log\log n + \log(1/\delta)}\rp\}.
\end{align*}
Lastly, notice $\|h(J)-h(J')\|_F = \sqrt{\ell}\|J-J'\|_F$, hence if we define $\hat{J} = h^{-1}(J^\ell)$ then we derive that
\[
\|\hat{J}-J^*\|_F \le \frac{1}{\sqrt{\ell}}\min \lp\{r \ge 0 \colon r \ge C(M)\sqrt{\log N\lp(\mathcal{J}, \|\cdot\|_2, \frac{cr^2}{n\ell}\rp)
	+\log\log n + \log(1/\delta)}\rp\}.
\]
This proves Corollary~\ref{cor:iid-formal}.
\end{proof}
\subsubsection{Estimating the complete matrix from independent samples}

Assume now that nothing is known about the matrix $J$, we have $\mathcal{J} = \{J \colon \|J\|_\infty \le M\}$. Then, $\mathcal{J}$ is an intersection of a vector space of dimension bounded by $n^2$ with the set of matrices of bounded infinity norm. We can apply Corollary~\ref{cor:iid-formal} in combination with the covering number bound of Lemma~\ref{lem:subspace-cover} to derive that
\[
\|J^*-\hat{J}\|_F \le C(M)\sqrt{n^2\log (n\ell)+\log(1/\delta)},
\]
where $\ell$ is the number of samples. Since the optimization is over a linear space, we derive a polynomial time algorithm. Hence, we proved the following.
\multinoassump*

\subsubsection{Multiple samples with dependencies}
We now extend the previous results in the case where the $\ell$ samples are not independent.
Specifically, assume that one obtains $\ell$ samples, but instead of being independent, they have some dependency structure. We show how to formulate this as a one-sample learning problem. For the sake of presentation, we limit the discussion to time-series dependencies. However, similar arguments can handle general dependency structures. Assume that $x^1,\dots,x^\ell$ are $n$-bit vectors, let $J_0,J_1$ denote two $n\times n$ matrices and define the joint distribution over $x^1,\dots,x^\ell$ by
\[
\Pr_{J_0,J_1,\ell}\lp[x^1\cdots x^\ell\rp] \propto 
\prod_{t=1}^\ell \exp\lp(-(x^t)^\top J_0 x^t/2\rp) \prod_{t=1}^{\ell-1} \exp\lp(-(x^t)^\top J_1 x^{t+1}\rp).
\]
Here, $J_0$ is the interaction matrix that controls each sample $x^t$ and $J_1$ controls the interaction between $x^t$ and $x^{t+1}$. We would like to present this distribution as an Ising model on $n\ell$ random bits. The joint interaction matrix, that we denote by $h(J_0,J_1)$, is a block matrix with $\ell$ copies of $J_0$ on the diagonal and $\ell-1$ copies of $J_1$ on each of the sub-diagonals.
We would like to estimate $J_0^*,J_1^*$ from $\ell$ dependent samples, under the assumption that $J_0^* \in \J_0$ and $J_1^* \in \J_1$ for some collections of interaction matrices of bounded infinity norm. The details are as follows. 
\begin{corollary}\label{cor:weakmulti}
	Let $M>0$, let $\ell \ge 2$, let $\J_0$ and $\J_1$ be collections of interaction matrices of infinity norm bounded by $M$ and let $(x^1,\dots,x^\ell)\sim \Pr_{J_0^*,J_1^*,\ell}$ for some $J_0^*\in \J_0$ and $J_1^*\in\J_1$. Then, there exists an estimator $(\hat{J}_0,\hat{J}_1)$ that satisfies w.p. $1-\delta$,
	\begin{align*}
	&\max(\| J_0^*-\hat{J}_0\|_F, \|J_1^*-\hat{J}_1\|_F)
	\le \\
	&\frac{1}{\sqrt{\ell}} \min \lp\{r \ge 0 \colon r \ge C(M)\sqrt{\log N\lp(\mathcal{J}_0, \|\cdot\|_2, \frac{cr^2}{n\ell}\rp)
		+ \log N\lp(\mathcal{J}_1, \|\cdot\|_2, \frac{cr^2}{n\ell}\rp)
		+\log\log n + \log(1/\delta)}\rp\}\\
	&\le \frac{C(M)}{\sqrt{\ell}} \sqrt{\log N\lp(\mathcal{J}_0, \|\cdot\|_2, \frac{1}{n\ell}\rp)
		+ \log N\lp(\mathcal{J}_1, \|\cdot\|_2, \frac{1}{n\ell}\rp)
		+\log\log n + \log(1/\delta)}.
	\end{align*}
\end{corollary}
\begin{proof}
 Let $\J_{01} = \{ h(J_0,J_1) \colon J_0\in \J_0, J_1\in \J_1\}$ and we will reduce to the one-sample problem of estimating $J_{01}^* \in \J_{01}$. We start by computing some properties of $\J_{01}$. First of all, it has matrices of bounded infinity norm. Indeed, it is easy to see that $\|h(J_0,J_1)\|_\infty \le \|J_0\|_\infty + 2\|J_2\|_\infty$. Secondly, we bound the covering numbers of $\J_{01}$. First, we bound the spectral norms of matrices $h(J_0,J_1)$ in terms of those of $J_0$ and $J_1$. Notice that for any $n\ell$ dimensional vector $\vec{u}$ denoted by $(u^t_i)_{t\in [\ell],i\in[n]}$ and we have
\begin{align*}
\vec{u}^\top h(J_0,J_1)\vec{u}
&= \sum_{t=1}^\ell (u^t)^\top J_0 u^t
+ \sum_{t=1}^{\ell-1} (u^t)^\top J_1 u^{t+1}
\le \sum_{t=1}^\ell \|u^t\|_2^2 \|J_0\|_2
+ \sum_{t=1}^{\ell-1} \|(u^t)\|_2 \|J_1\|_2 \|u^{t+1}\|_2\\
&\le \|J_0\|_2 \sum_{t=1}^\ell \|u^t\|_2^2 
+ \|J_1\|_2 \sum_{t=1}^{\ell-1} (\|(u^t)\|_2^2+ \|u^{t+1}\|_2^2)/2
\le (\|J_0\|_2 + \|J_1\|_2) \|\vec{u}\|_2^2.
\end{align*}
where we used the arithmetic and geometric means inequality. This implies that $\|h(J_0,J_1)\|_2 \le \|J_0\|_2 + \|J_1\|_2$. Since $h$ is a linear function, we have that for any $J_0,J_1,J_0',J_1'$, $\|h(J_0,J_1)-h(J_0',J_1')\|_2 = \|h(J_0-J_0',J_1-J_1')\|_2 \le \|J_0-J_0'\|_2 + \|J_1-J_1'\|_2$. This implies that for any $\epsilon > 0$, one has
\[
N(\J_{01},\|\cdot\|_2,\epsilon)
\le N(\J_{0},\|\cdot\|_2,\epsilon/2) N(\J_{1},\|\cdot\|_2,\epsilon/2).
\]
Indeed, given an $\epsilon/2$-cover $\mathcal{N}_0$ for $\J_0$ and $\mathcal{N}_1$ for $\J_1$, we can take $\mathcal{N}_{01} = \{ h(J_0,J_1) \colon J_0\in \mathcal{N}_0,J_1 \in \mathcal{N}_1 \}$, and it forms an $\epsilon$-cover for $\J_{01}$. We derive that $J_{01}^*\in \J_{01}$ can be estimated with an error of
\begin{align*}
&\|J_{01}^*-\hat{J}_{01}\|_F\le\\
&\min \lp\{r \ge 0 \colon r \ge C(M)\sqrt{\log N\lp(\mathcal{J}_0, \|\cdot\|_2, \frac{cr^2}{n\ell}\rp)
+ \log N\lp(\mathcal{J}_1, \|\cdot\|_2, \frac{cr^2}{n\ell}\rp)
	+\log\log n + \log(1/\delta)}\rp\}.
\end{align*}
Next, we translate back to guarantees on estimating $J_0^*$ and $J_1^*$. Notice that 
\[
\|h(J_0,J_1) - h(J_0',J_1')\|_F^2 
= \ell \|J_0-J_0'\|_F^2 + 2(\ell-1)\|J_1-J_1'\|_F^2.
\]
This immediately implies the rate given in the statement.
\end{proof}

\section{The Conditioning Trick}\label{s:condition}
The main purpose of this section is to prove the following Lemma, which is used multiple times in the proof of Theorem~\ref{t:general}.

\lemsubsample*
Intuitively, this lemma says that we can select a small number of subsets,
such that each element $i \in [n]$ is contained in at least
a constant fraction of these sets. 
At the same time, we want the variables in these subsets
to be weakly dependent when we condition on the rest.
The proof relies on the following technical lemma.
\begin{lemma}\label{l:aux}
	Let $A$ be a matrix with zero on the diagonal, and $\|A\|_\infty \le 1$. Let $0< \eta < 1$. Then, there exist subsets $I_1, \dots, I_\ell \subseteq [n]$, with $\ell \le C \log n/\eta^2$ such that:
	\begin{enumerate}
		\item For all $i \in [n]$, 
		\begin{equation}\label{eq:pr1}
		    |\{j \in [\ell] \colon i \in I_j\}| = \lceil \eta\ell/8 \rceil
		\end{equation}
		\item For all $j \in [\ell]$ and all $i \in I_j$, \begin{equation}\label{eq:pr2}
		    \sum_{k \in I_j} |A_{ik}| \le \eta
		\end{equation}
	\end{enumerate}
\end{lemma}
This Lemma lies in the heart of the proof. It essentially 
ensures that we can select subsets of nodes that conditionally
satisfy Dobrushin's condition with an arbitrarily small constant. It also shows that these subsets can cover all
the nodes. Thus, it allows breaking up a sum over all the nodes to a sum over all these subsets with only a logarithmic error factor. The proof is an application of the probilistic method, using a simple, but strong technique found in \cite{alon2004probabilistic}. 

\begin{proof}[Proof of Lemma~\ref{l:aux}]
The proof uses the \emph{probabilistic method}. To prove that
there exists an object with a specific property, 
we find a way to select the object at random so that with
positive probability the property is satisfied. 
In our case, the object is the collection of subsets $I_1,\ldots,I_\ell$. The properties are ~\ref{eq:pr1} and ~\ref{eq:pr2}. 

We first define a way to select these subsets at random.
This is done with the following sampling procedure:
\begin{algorithm}
\caption{Sampling Procedure}
\begin{algorithmic}
\FOR {j := $1,\ldots,\ell$}
\STATE Draw a random set $I_j'$ of coordinates (a subset of the set of all coordinates $[n]$), where each $i \in [n]$ is selected to be in $I_j'$ independently with probability $\frac{\eta}{2}$.
\STATE Set $I_j$ to be the set of all coordinates $i \in I_j'$ such that $\sum_{k \in I_j'} |A_{ik}| \le \eta$. 
\ENDFOR
\STATE Output $(I_1,\ldots,I_\ell)$
\end{algorithmic}
\end{algorithm}

Notice that each subset is selected independently of the others.
We are now going to prove that the output of this sampling
procedure satisfies ~\ref{eq:pr1} and ~\ref{eq:pr2}
with positive probability. By the nature of the procedure,
~\ref{eq:pr2} is automatically satisfied. It remains to 
check that ~\ref{eq:pr1} holds. First, for a fixed $i \in [n]$ and $j \in [\ell]$,
we calculate the probability that $i \in I_j$.
\begin{align*}
\Pr\lp[i \in I_j\rp]
&= \Pr\lp[i \in I_j'\rp] \Pr\lp[i \in I_j \middle| i \in I_j'\rp]
= \Pr\lp[i \in I_j'\rp] \Pr\lp[\sum_{k \in I_j'} |A_{ik}| \le \eta ~\middle|~ i \in I_j\rp]\\
&= \Pr\lp[i \in I_j'\rp] \Pr\lp[\sum_{k \in I_j'} |A_{ik}| \le \eta\rp]
= \frac{\eta}{2} \Pr\lp[\sum_{k \in I_j'} |A_{ik}| \le \eta\rp].
\end{align*}
By linearity of expectation we have:
\[
\E\lp[\sum_{k \in I_j'} |A_{ik}|\rp] = \frac{\eta}{2} \sum_{k \in [n]} |A_{ik}| \le \frac{\eta}{2}
\]
Thus, by applying Markov's inequality we get:
\[
\Pr\lp[\sum_{k \in I_j'} |A_{ik}| \ge \eta\rp] \le \frac{1}{2},
\]
which yields
$$
\Pr\lp[i \in I_j\rp] \ge \frac{\eta}{4}.
$$
Now, fix an $i \in [n]$ and define 
$$
S_i := |\{j \in [\ell] \colon i \in I_j\}|
$$
Notice that the events $\{i \in I_j\}_{j=1}^\ell$ are independent
from each other. 
Hence, we can write
$$
S_i = \sum_{j=1}^\ell \mathbf{1}(i \in I_j)
$$
which means that $S_i$ is a sum of independent Bernoulli
random
variables.
By the preceding calculation, we get:
$$
\E[S_i] = \sum_{j=1}^\ell\E\lp[\mathbf{1}(i \in I_j)\rp]=
\sum_{j=1}^\ell \Pr\lp[i \in I_j\rp] \geq \frac{\eta \ell}{4}.
$$
Now, using Hoeffding's inequality and setting
$\ell = 32\log4 \log n/\eta^2$ we get:
\begin{align*}
\Pr\lp[S_i \leq \frac{\eta \ell}{8}\rp] &\leq
\Pr\lp[|S_i - \E[S_i]| \geq \frac{\eta \ell}{8}\rp]
\leq 2\exp\lp(-2\frac{\lp(\frac{\eta \ell}{8}\rp)^2}{\ell}\rp)\\
&\leq 2\exp\lp(-\frac{\eta^2 \ell}{32}\rp)
\leq \frac{1}{2n}.
\end{align*}

By a simple union bound we get
$$
\Pr\lp[\exists i\in [n]:S_i \leq \frac{\eta \ell}{8}\rp] \leq n\frac{1}{2n} = \frac{1}{2}
$$
That means that with positive probability we have $S_i \geq \eta \ell/8$. Hence, there exists a collection of subsets $I_1,\ldots,I_\ell$ having this property.
Notice that we do not have Equation~\ref{eq:pr1} but an inequality instead. However, for each $i$ such that
$S_i > \eta \ell/8$, we can remove $i$ from some
sets so that the number of sets it appears is exactly 
$\lceil \eta\ell/8 \rceil$. Clearly, by removing elements
from a set, inequality~\ref{eq:pr2} still holds.
\end{proof}

\begin{proof}[Proof of Lemma~\ref{lem:ising-subsample}]
    Let $J$ be the interaction matrix of the Ising model 
    distribution of $X$ and let $h$ denote the external field. By the hypothesis,
    we have that $\|J\|_\infty/M \leq 1$.
    Hence, by applying Lemma~\ref{l:aux} for the matrix
    $\|J\|_\infty/M \leq 1$ with $\eta' = \eta/M$,
    we obtain that there is a collection of subsets
    $I_1,\ldots,I_\ell$ such that
    \begin{itemize}
        \item For all $i \in [n]$,
        $$
            |\{j \in [\ell] \colon i \in I_j\}| = \lceil \frac{\eta\ell}{8M} \rceil
            $$
        \item For all $j \in [l]$ and all $i \in I_j$
        \begin{equation}\label{eq:dobr}
        \sum_{k \in I_j} \frac{|A_{ik}|}{M} \leq \frac{\eta}{M} \implies \sum_{k \in I_j} |A_{ik}| \leq \eta 
        \end{equation}
    \end{itemize}
    We now argue that if inequality~\ref{eq:dobr} holds,
    then the conditional distribution of $X_{I_j}$ conditioned on $X_{-I_j}$ is an $(\eta,\gamma)$-Ising model.
    To show that, it suffices to argue that $\Pr[X_{I_j} = y | X_{-I_j} = x_{-I_j}] \propto \exp(y^\top J' y + h'^\top y)$, for some interaction matrix $J'$ and external field $h'$.
    Hence, we calculate the ratio of conditional probabilities for two configurations $y$ and $y'$:
    $$
    \frac{\Pr\lp[X_{I_j} = y|X_{-I_j} = x_{-I_j}\rp]}{\Pr\lp[X_{I_j} = y'|X_{-I_j}=x_{-I_j}\rp]} = 
    \frac{\exp\lp(\sum_{u \in I_j,v \in I_j} J_{uv}y_uy_v +
    \sum_{u \in I_j, v \notin I_j} J_{uv}y_u x_v + h^\top y\rp)}{\exp\lp(\sum_{u \in I_j,v \in I_j} J_{uv}y_u'y_v' +
    \sum_{u \in I_j, v \notin I_j} J_{uv}y_u' x_v + h^\top y'\rp)}
    $$
    By the preceding equality, it is clear that the conditional
    distribution of $X_{I_j}$ conditional on $X_{-I_j} = x_{-I_j}$ is an Ising model with interaction matrix
    $J' = \{J_{uv}\}_{u,v\in I_j}$ and external field $h_i' = h_i + \sum_{v \notin I_j}J_{iv}x_v$. 
    This proves the claim. 
    
    To conclude with the last part of the lemma, fix $a \in \mathbb{R}^n$, and drawing $j \in [\ell]$ uniformly at random, one obtains
	\[
	\E_j[\sum_{i \in I_j} a_i]
	= \E_j[\sum_{i=1}^n a_i \mathbf{1}(i \in I_j)]
	= \frac{\lceil \eta\ell/(8M) \rceil}{\ell} \sum_{i=1}^n a_i.
	\ge \frac{\eta}{8M} \sum_{i=1}^n a_i.
	\]
	In particular, there exists some $j \in [\ell]$ which achieves a value of at least this expectation.
\end{proof}

\section{Concentration Inequalities for General Functions}
\label{s:concentration}
One of the main goals of the proof is to show that the derivative of the log pseudo-likelihood concentrates around its mean value. To do this, we rely on suitably conditioning on some of the spins, which guarantees that the conditional distribution on the remaining spins satisfies Dobrushin's condition.
However, the task of showing concentration of the derivative in
this regime remains. We begin with an overview of the results regarding
concentration of functions on the boolean hypercube when Dobrushin's condition is satisfied. We then state and prove a modification of
these results that serves our purposes well.

Many concentration results in the weak dependence regime concern
functions that are multi-linear polynomials. Specifically, in \cite{adamczak2019note} they prove
general concentration results for arbitrary multilinear polynomials of weakly dependent random variables.
For polynomials of degree 2, they show the following:
\begin{theorem}[\cite{adamczak2019note}]\label{t:note}
	Let $A$ be a symmetric matrix with zeros in the diagonal and $X$ be an $(\alpha,\gamma)$-bounded random variable. Let $p(x) = x^\top A x$. Then, for any $t > 0$,
	\[
	\Pr\lp[\lp|p(X) - \E p(X)\rp| > t\rp]
	\le \exp\lp(-c \min\lp(\frac{t^2}{\|A\|_F^2+\|\E Ax\|_2^2}, \frac{t}{\|A\|_2}\rp)\rp).
	\]
where the constant $c$ only depends on $\alpha,\gamma$ and
not on the entries of $A$. 
\end{theorem}
This inequality is tight up to constants. A nice property of this result is that it explicitly connects
the radius of concentration with the Frobenius norm of the
matrix. 

Unfortunately, the derivative of the PMLE is not a polynomial.
Hence, we cannot directly apply Theorem~\ref{t:note}. Instead,
we will modify it's proof so that it holds for arbitrary functions over the hypercube. 
The original proof relied on the fact that
the Hessian matrix of a second degree polynomial is constant.
Despite this not being true in our case,
we will follow the same strategy and prove concentration 
using the second order Taylor approximation of the function.
First, we need to define these quantities precisely. In the following, for
a vector $x$ and an index $i$,
we denote by $x_{i+}$ the vector obtained from $x$ by replacing that $i$'th coordinate with $1$ and by $x_{i-}$ the one that is obtained by replacing this coordinate with $-1$.

\begin{definition}
For an arbitrary function $f:\{0,1\}^n \mapsto \R$,
we define the \emph{discrete derivative} of the function as 
\[
D_i f(x) := \frac{f(x_{i+}) - f(x_{i-})}{2}
\]
Let $Df(x)$ denote the $n$-coordinate vector of discrete derivatives. 
Similarly, the function $H:\{0,1\}^n \mapsto \R^{n\times n}$
defined as
$$
H_{ij}(x) = D_i(D_jf(x))
$$
is called the \emph{discrete Hessian} of $f$. 
\end{definition}
As we will see, our concentration bound will depend on the
discrete derivative and Hessian of a function.
However, in some cases it is more convenient to provide bounds
for these quantities rather than explicitly calculate them.
Therefore, we can replace these two quantities by the ones defined below.
\begin{definition}
Let $f:\{0,1\}^n \mapsto \R$ be an arbitrary function.
A function $\tilde{D}:\{0,1\}^{n_1} \mapsto \R^n$ is called a
\emph{pseudo discrete derivative} for $f$ if
\[
\|\tilde{D}(x)\|_2 \ge \|Df(x)\|_2
\]
for all $x \in \{0,1\}^n$.
Additionally, we say that a function $\tilde{H} \colon \{-1,1\}^n \to \mathbb{R}^{n_1 \times n_2}$ is a \emph{pseudo discrete Hessian} with respect to the pseudo discrete derivative $\tilde{D}$ if for all $u \in \mathbb{R}^{n_1}, x \in \R^n$,
\[
\|u^\top\tilde{H}(x)\|_2 \ge \|D(u^\top \tilde{D}(x))\|_2.
\]
\end{definition} 
We are now ready to state the modification of Theorem~\ref{t:note}.
\begin{theorem}\label{t:newnote}
    Let $f:\{0,1\}^n \mapsto \R$ be an arbitrary function and
    $X$ an $(\alpha,\gamma)$-bounded random variable. Then
	\[
	\Pr[|f(x)-\E f(x)|>t]
	\le \exp\lp(-c\min\lp(\frac{t^2}{\|\E \tilde{D}(x)\|_2^2 + \max_x\|\tilde{H}(x)\|_F^2}, \frac{t}{\max_x \|\tilde{H}(x)\|_2}\rp)\rp).
	\]
\end{theorem}
The proof of Theorem~\ref{t:newnote} follows the structure
of Theorem~\ref{t:note}, with only slight modifications.
Hence, we are now going to describe the machinery used in the proof of Theorem~\ref{t:note}.

The main ingredient of the proof is a discrete logarithmic Sobolev inequality, first proven in \cite{marton2003measure}. First, we need a definition.

\begin{definition}
For any function $f:\{0,1\}^n \mapsto \R$ and $i \in [n]$,
we define
\[
\mathfrak{d}_if(x) = \frac{1}{2} \lp(\E\lp(\lp(f(X) - f(X_1,\ldots,X_{i-1},\Tilde{X_i},X_{i+1},\ldots,X_n)\rp)^2\mid X = x\rp)
\rp)^{1/2}
\]
where the random variable $X = (X_1,\ldots,X_n)$ has distribution $\mu$ and the conditional distribution of $\Tilde{X_i}$ given that $X = x$ is $\mu(\cdot|\Bar{x}_i)$, 
which is the conditional distribution of $X_i$ given that 
$X_j = x_j$ for $j \neq i$.
We denote by $\mathfrak{d}f(x)$ the vector in $\R^n$
whose $i$-th coordinate is $\mathfrak{d}_if(x)$.
\end{definition}
This quantity intuitively captures how much function $f$
changes on average when we resample one of its input variables independently. By the classical theory of Concentration
Inequalities for i.i.d. random variables, we know that this \emph{average Lipschitzness}  directly
affects the concentration properties of the function. 
This is indeed the case here, as the next lemma shows.
\begin{lemma}[\cite{marton2003measure,gotze2019higher}]\label{l:logsob}
Let $f:\{0,1\}^n \mapsto \R$ be an arbitrary function on the hypercube. Suppose $X$ is an $(\alpha,\gamma)$-bounded random variable. Then for any $p \geq 2$ it holds
\[
\|f(X) - \E f(X)\|_p :=
\lp(\E \lp( f(X) - \E f(X)\rp)^p\rp)^{1/p} \leq \sqrt{2Cp}\lp(\E\lp(\|\mathfrak{d}f(X)\|_2^p\rp)\rp)^{1/p}
\]
where C is a function of $\alpha,\gamma$ which is bounded when $\alpha$ is bounded from $1$ and $\gamma$ is bounded from zero.
\end{lemma}
Lemma~\ref{l:logsob} essentially tells us that in order to bound the $p$-th moment of the function we wish to show concentration for, it is enough to control the $p$-th moment 
of $\|\mathfrak{d}f(X)\|_2$. In the proof of Theorem~\ref{t:note}, the authors bound this moment by the
corresponding moment of a multilinear form of gaussian random
variables. In doing so, they exploit the fact that $\mathfrak{d}_if(X)$ can be very conveniently bounded when $f$
is a multilinear polynomial. We will follow exactly the same technique, while relying on the discrete derivative of $f$ instead to bound $\mathfrak{d}_if(X)$. 

\begin{proof}[Proof of Theorem~\ref{t:newnote}]
As mentioned earlier, the proof follows the same strategy as \cite{adamczak2019note}, with a small adjustment. Our general strategy will be to bound the $p$-th
moment of $f(X) - \E f(X)$ by the $p$-th moment of a gaussian
multilinear form. 
By Jensen's Inequality, we have:
\begin{align*}
\E\lp(\|\mathfrak{d}f(X)\|_2^p\rp) &= 
\E \lp[\lp(\sum_{i=1}^n \lp(\frac{1}{2} \lp(\E\lp(\lp(f(X) - f(X_1,\ldots,X_{i-1},\Tilde{X_i},X_{i+1},\ldots,X_n)\rp)^2\mid X = x\rp)
\rp)\rp)\rp)^{p/2}\rp]\\
&\leq \E \lp[\lp(\sum_{i=1}^n \frac{1}{2} \lp(f(X) - f(X_1,\ldots,X_{i-1},\Tilde{X_i},X_{i+1},\ldots,X_n)\rp)^2
\rp)^{p/2}\rp]
\end{align*}
We now note that for each $i$, either $X_i$ and $\tilde{X_i}$
will be the same or they will have opposite signs.
This means that in any case
$$
\lp|f(X) - f(X_1,\ldots,X_{i-1},\Tilde{X_i},X_{i+1},\ldots,X_n)\rp| \leq \lp|D_if(X)\rp|
$$
This means that
$$
\E\lp(\|\mathfrak{d}f(X)\|_2^p\rp)
\leq \E \lp(\|\frac{1}{2}Df(X)\|_2^p\rp)
$$
Combining this with the log Sobolev inequality of Lemma~\ref{l:logsob} we get:
\begin{equation}\label{eq:step1}
\|f(X) - \E f(X)\|_p \leq \sqrt{Cp} \lp(\E \|Df(X)\|_2^p\rp)^{1/p}
\leq \sqrt{Cp} \lp(\E \|\tilde{D}(X)\|_2^p\rp)^{1/p}
\end{equation}
Now suppose $g \sim \mathcal{N}(0,I_n)$ is an
$n$-dimensional gaussian vector independent of $X$. 
For a fixed $X$, the random variable $\frac{\tilde{D}(X)^\top g}{\|\tilde{D}(X)\|_2}$
is a single dimensional gaussian $\mathcal{N}(0,1)$.
Hence, by elementary properties of the gaussian distribution, there exists a constant $M>0$ independent of $p$ such that 
$$
\lp(\E_g\lp(\frac{\tilde{D}(X)^\top g}{\|\tilde{D}(X)\|_2}\rp)^p\rp)^{1/p}
\geq
\frac{1}{M}\sqrt{p}
$$
The symbol $\E_g$ means that we are integrating only with
respect to the random variable $g$.
This implies that for a fixed $X$
$$
\sqrt{p} \|\tilde{D}(X)\|_2 \leq M \lp(\E_g\lp(\tilde{D}(X)^\top g\rp)^p\rp)^{1/p}
$$
Combining with Equation~\ref{eq:step1} , we get that for all functions $f$ and pseudo derivatives $\tilde{D}$:
\begin{equation}\label{eq:induction}
\|f(X) - \E f(X)\|_p 
\leq K\lp(\E_{X,g} \lp(\tilde{D}(X)^\top g\rp)^p\rp)^{1/p}
\end{equation}
where $K = M\sqrt{C}$.
Inequality~\ref{eq:induction} is a first step in proving the
bound on the $p$-th moment. We want to use this inequality
to make the second derivative appear on the right hand side. 
To do this, we "fix" the value of $g$. By  Minkowski's inequality we have:
$$
\lp(\E_{X,g} \lp(\tilde{D}(X)^\top g\rp)^p\rp)^{1/p} =
\|\tilde{D}(X)^\top g\|_p \leq
\|\tilde{D}(X)^\top g - \E_X \tilde{D}(X)^\top g\|_p
+ \|\E_X \tilde{D}(X)^\top g\|_p
$$
First, we bound the second term.
By linearity of expectation, we have:
$$
\|\E_X \tilde{D}(X)^\top g\|_p =
\lp( \E_g \lp( (\E_X \tilde{D}(X))^\top g\rp)^p \rp)^{1/p}
$$
Now, the variable $(\E_X \tilde{D}(X))^\top g$ is clearly
a single dimensional gaussian, which means that
it's $p$-th moment is bounded as:
\begin{equation}\label{eq:step2}
    \|\E_X \tilde{D}(X)^\top g\|_p
    \leq M\sqrt{p}\|\E_X \tilde{D}(X)\|_2
\end{equation}
and that concludes the bound for the second term.
For the first term, fix $g$ and define the function $h:\{0,1\}^n\mapsto \R$ as $h(x) = \tilde{D}(x)^\top g - \E_X \tilde{D}(X)^\top g$. 
Now, we apply inequality~\ref{eq:step1} to the function 
$h$, which gives us
\begin{align*}
    \lp(\E_X \lp(\tilde{D}(X)^\top g - \E_X \tilde{D}(X)^\top g\rp)^p\rp)^{1/p}
    &\leq\sqrt{Cp}\lp( \E_X \|D(\tilde{D}(X)^\top
    g)\|_2^p\rp)^{1/p}\\
    &\leq \sqrt{Cp}  \lp( \E_X \|g^\top \tilde{H}(X)\|_2^p\rp)^{1/p}\\
    &\leq \sqrt{C}M \lp( \E_{X,g'} \lp|g^\top \tilde{H}(X)g'\rp|^p\rp)^{1/p}
\end{align*}
We conclude that
\begin{align*}
    \|\tilde{D}(X)^\top g - \E_X \tilde{D}(X)^\top g\|_p
    &= \lp(\E_{X,g} \lp(\tilde{D}(X)^\top g - \E_X \tilde{D}(X)^\top g\rp)^p\rp)^{1/p}\\
    &\leq K \lp( \E_{X,g,g'} \lp|g^\top \tilde{H}(X)g'\rp|^p\rp)^{1/p}
\end{align*}
Notice now that we have to bound the $p$-th norm of a quadratic form $g^\top \tilde{H}(X)g'$ where $g,g'$ are independent gaussian vectors.
That is precisely what the Hanson-Wright inequality gives us. Note that the matrix $\tilde{H}$ need not have zeroes in
the diagonal, as mentioned in \cite{rudelson2013hanson,latala2006estimates}.
\footnote{The standard result in \cite{rudelson2013hanson, latala2006estimates} applies to square matrices $A$. 
Since $\tilde{H}$ might not be necessarily a square matrix,
we can just make it square by adding zeroes to the missing
dimensions. The quadratic form doesn't change and
the matrix norms $\|\cdot\|_F, \|\cdot\|_2$ remain
the same.}
For a fixed $X$ there is a constant $C'$ such that:
$$
\lp(\E_{g,g'}\lp(g^\top \tilde{H}(X)g'\rp)^p\rp)^{1/p}\leq C'\lp( \sqrt{p}\|\tilde{H}(X)\|_F + p \|\tilde{H}(X)\|_2\rp)
\leq C'\lp(\sqrt{p}\max_{x \in \{0,1\}^n}\|\tilde{H}(x)\|_F
+ p\max_{x \in \{0,1\}^n}\|\tilde{H}(x)\|_2\rp)
$$
This gives us the final inequality:
\begin{equation}\label{eq:step3}
    \|\tilde{D}(X)^\top g - \E_X \tilde{D}(X)^\top g\|_p
    \leq C'\lp(\sqrt{p}\max_{x \in \{0,1\}^n}\|\tilde{H}(x)\|_F
+ p\max_{x \in \{0,1\}^n}\|\tilde{H}(x)\|_2\rp)
\end{equation}
Putting together inequalities~\ref{eq:step2} and ~\ref{eq:step3} we conclude that there exists a universal
constant $C''$ such that 
$$
\|f(X) - \E f(X)\|_p \leq C''\lp(\sqrt{p}\|\E_X \tilde{D}(X)\|_2 + \sqrt{p}\max_{x \in \{0,1\}^n}\|\tilde{H}(x)\|_F
+ p\max_{x \in \{0,1\}^n}\|\tilde{H}(x)\|_2\rp)
$$
To conclude the proof, we notice that by Markov's inequality
and the preceding result:
$$
\Pr\lp[|f(X) - \E f(X)| > e  C''\lp(\sqrt{p}\|\E_X \tilde{D}(X)\|_2 + \sqrt{p}\max_{x \in \{0,1\}^n}\|\tilde{H}(x)\|_F
+ p\max_{x \in \{0,1\}^n}\|\tilde{H}(x)\|_2\rp)\rp] \leq e^{-p}
$$
We now set 
\[
p = \min\lp(\frac{t^2}{\|\E_X \tilde{D}(X)\|_2^2 + \max_{x \in \{0,1\}^n}\|\tilde{H}(x)\|_F^2},
\frac{t}{\max_{x \in \{0,1\}^n}\|\tilde{H}(x)\|_2}\rp)
\]
and the result follows.
\end{proof}
\section{Improved Bound for Estimating a Single Parameter}
\label{sec:one-parameter}
In this section, we prove Corollary~\ref{cor:one-param-informal}, which we now restate for convenience.
\oneparam*

Similarly to the proof of Theorem~\ref{t:general}, our estimator is the maximum pseudo likelihood. 
Define the negative log pseudo-likelihood as %\yuval{Make sure that we define what is $\Pr_J$}
\[
\varphi(\beta) = - \sum_{i=1}^n \log \Pr_{\beta J}[x_i \mid x_{-i}].
\]
Then, the following holds:
\begin{equation}\label{eq:oneparam-derivative}
\varphi'(\beta)
:= \frac{d\varphi(\beta)}{d\beta}
= \frac{1}{2}\sum_{i=1}^n 
(J_ix) (\tanh(\beta J_i x)-x_i)
\end{equation}
and further,
\begin{equation}\label{eq:oneparam-secondder}
\varphi''(\beta)
:= \frac{d^2\varphi(\beta)}{d\beta^2}
= \frac{1}{2}\sum_{i=1}^n
(J_ix)^2 \sech^2(\beta J_ix)
\ge c \|Jx\|_2^2,
\end{equation}
where the last inequality holds uniformly for all $|\beta|\le M$, since $\sech^2$ is lower bounded whenever its argument is bounded from above.
Similarly to the arguments in Lemma~\ref{t:general}, it suffices to bound the ratio between the first and second derivatives of $\varphi$ and we obtain that
\begin{equation}\label{eq:oneparam-finalbound}
|\beta^*-\hat\beta| 
\le \frac{|\varphi'(\beta^*)|}{\min_{|\beta|\le M} \varphi''(\beta^*)}
\le C \frac{|\varphi'(\beta^*)|}{\|Jx\|_2^2}.
\end{equation}

We prove a lower bound on $\|Jx\|_2^2$ based on arguments from \cite{chatterjee2007estimation,bhattacharya2018inference}, and use lemmas from the proof of Theorem~\ref{t:general} to show that the derivative is bounded in terms of $\|Jx\|_2$.

We begin with the second part. For this purpose, we provide a restatement of the lemmas from the proof of Theorem~\ref{t:general} that we will use here. First, the sub-sampling lemma, that shows how to reduce the correlations in the Ising model by subsampling:
\lemsubsample*

We create the sets $I_1,\dots,I_\ell$ from this lemma with $\eta = 1/2$. The following two lemmas would provide an upper bound on $\varphi'$ and a lower bound on $\|Jx\|_2$, respectively, both in terms of the same quantity, which is a function of the sets $I_1,\dots,I_\ell$:
\begin{lemma}[A special case of Lemma~\ref{lem:derivative-one-concentration}]\label{lem:derivative-special-case}
	For any $t >0$,
	\[
	\Pr\left[\left|\varphi'(\beta^*)\right|>t\left(\|J\|_F + \max_j \|\E[Jx|x_{-I_j}]\|_2\right)\right]
	\le C \log n \exp\lp(-c\min\lp(t^2,\frac{t\|J\|_F}{\|J\|_2}\rp)\rp).
	\]
\end{lemma}
\begin{lemma}[A special case of Lemma~\ref{lem:hessian-oneparam}] \label{lem:second-derivative-special-case}
	For any $t > 0$:
	\begin{align*}
	&\Pr\left[\|Jx\|_2^2 < c\|J\|_F^2 + c\max_j \|\E[Jx\mid x_{-I_j}]\|_2^2 - t \left(\|J\|_F + \max_j \|\E[Jx\mid x_{-I_j}]\|_2\right)\right]\\
	&\quad \le C \log n \exp\left(-c\frac{\min(t^2,t\|J\|_F)}{\|J\|_2^2}\right).
	\end{align*}
\end{lemma}
Lemmas \ref{lem:derivative-special-case} and \ref{lem:second-derivative-special-case} follow from Lemma~\ref{lem:derivative-one-concentration} and Lemma~\ref{lem:hessian-oneparam} by substituting $J^*$ with $\beta^* J$ and $A$ with $J$, and noting that the quantity $\frac{\partial \varphi(J\beta^*)}{\partial J}$ in Section~\ref{sec:pr-multiparam} corresponds to $\varphi'(\beta^*)$ in this section.
As a direct corollary, we can bound $\varphi'$ with respect to $\|Jx\|$:
\begin{lemma}\label{lem:oneparam-derivative}
	For any $t\ge1$, with probability at least $1-\log ne^{-ct}$,
	\[
	\left|2 \varphi'(\beta^*)\right|
	=\left|x^\top J x - \sum_{i=1}^n J_ix \tanh(\beta^*J_ix)\right| 
	\le Ct\|Jx\|_2 + Ct^2.
	\]
\end{lemma}
\begin{proof}
	The first equality follows from definition, hence we will prove the inequality. From Lemma~\ref{lem:derivative-special-case}, it holds that for any $t \ge 1$,
	\begin{align} \label{eq:222}
	&\Pr\left[\left|\varphi'(\beta^*)\right|>t\left(\|J\|_F + \max_j \|\E[Jx|x_{-I_j}]\|_2\right)\right]\\
	&\le C \log n \exp\lp(-c\min\lp(t^2,\frac{t\|J\|_F}{\|J\|_2}\rp)\rp)\\
	&\le C \log n \exp(-c\min(t^2,t))\\
	&= C \log n \exp(-ct),
	\end{align}
	since $\|J\|_F \ge \|J\|_2$ for all $J$ and since $t^2 \ge t$ for all $t\ge 1$. Second of all, from Lemma~\ref{lem:second-derivative-special-case}, for any $t \ge 1$, we have that
	\begin{align}\label{eq:223}
	&\Pr\left[\|Jx\|_2^2 < c\|J\|_F^2 + c\max_j \|\E[Jx\mid x_{-I_j}]\|_2^2 - t \left(\|J\|_F + \max_j \|\E[Jx\mid x_{-I_j}]\|_2\right)\right]\\
	&\le C \log n \exp\left(-c\frac{\min(t^2,t\|J\|_F)}{\|J\|_2^2}\right)\\
	&\le C \log n \exp(-c\min(t^2,t))\\
	&\le C \log n \exp(-ct).
	\end{align}
	using the fact that $\|J\|_F \ge \|J\|_2$ and that $\|J\|_2$ is at most some constant.
	
	Next, with probability at least $1-2\log ne^{-ct}$, both the events that their probabilities are estimated in \eqref{eq:222} and \eqref{eq:223} hold, and assume that they do hold till the rest of the proof.
	Let $\zeta = \|J\|_F + \max_j \|\E[Jx\mid x_{-I_j}]\|_2$ and we derive that $\left|\varphi'(\beta^*)\right|\le t\zeta$ and that $\|Jx\|_2^2 \ge c \zeta^2 -t\zeta$. It follows that
	\begin{equation} \label{eq:335}
	\left|\varphi'(\beta^*)\right| \le Ct\|Jx\|_2 + Ct^2.
	\end{equation}
	Indeed, if $t \le c\zeta/2$ then
	\[
	\|Jx\|_2^2
	\ge c\zeta^2/2
	\ge c (|\varphi'(\beta^*)|/t)^2/2
	\]
	hence
	\[
	|\varphi'(\beta^*)|
	\le \sqrt{2/c} t \|Jx\|_2.
	\]
	Otherwise,
	\[
	|\varphi'(\beta^*)|
	\le t \zeta
	\le 2t^2/c.
	\]
	This concludes the proof.
\end{proof}

Next, we lower bound $\|Jx\|_2^2$ in terms of the log partition function $F(\beta^*J)$, based on arguments from \cite{chatterjee2007estimation,bhattacharya2018inference}.
\begin{equation}\label{eq:22}
\|Jx\|_2^2
= \sum_{i=1}^n (J_ix)^2
\ge \sum_{i=1}^n \frac{J_ix \tanh(\beta^* J_ix)}{\beta^*},
\end{equation}
since for all $y \in \mathbb{R}$, $y$ and $\tanh(y)$ have the same sign and additionally, $|\tanh(y)| \le |y|$. From Lemma~\ref{lem:oneparam-derivative},
the right hand side of \eqref{eq:22} can be approximated by $x^\top J x/\beta^*$.
Further, the term $x^\top J x$ can be lower bounded by the log partition function:
\begin{lemma}
With probability at least $1-e^{-F(J\beta^*)/2}$, $x^\top J x \ge F(J\beta^*)/\beta^*$.
\end{lemma}
\begin{proof}
	Recall that
	\[
	\Pr_{J\beta^*}[x] = 2^{-n} e^{\beta^* x^\top J x/2 - F(J\beta^*)}
	\]
	hence,
	\[
	\E_{J\beta^*}\left[e^{-\beta^*x^\top J x/2}\right]
	= \sum_{x} e^{-\beta^*x^\top J x/2} \Pr_{J\beta^*}[x]
	= e^{-F(J\beta^*)}.
	\]
	Therefore, by Markov's inequality,
	\begin{align*}
	\Pr_{J\beta^*}\left[\beta^*x^\top J x < F(J\beta^*)\right]
	&= \Pr_{J\beta^*}\left[e^{-\beta^*x^\top J x/2} > e^{-F(J\beta^*)/2}\right]\\
	&\le \E_{J\beta^*}[e^{-\beta^*x^\top J x/2}]/e^{-F(J\beta^*)/2}
	= e^{-F(J\beta^*)/2}.
	\end{align*}
\end{proof}

By the above lemmas, we derive the following:
\begin{lemma}\label{lem:bnd-partition}
	For any $1 \le t \le c\sqrt{F(J\beta^*)}$, with probability at least $e^{-ct}$,
	\[
	\frac{\varphi'(\beta^*)}{\|Jx\|_2^2}
	\le C t \beta^*/\sqrt{F(J\beta^*)}.
	\]
\end{lemma}
\begin{proof}
We derive from \eqref{eq:22}, Lemma~\ref{lem:oneparam-derivative} and Lemma~\ref{lem:bnd-partition}, and from $t \le c \sqrt{F(\beta^*J)}$,
that with probability at least $\log n e^{-ct}$,
\begin{align}
\|Jx\|_2^2 
\ge \frac{1}{\beta^*}\sum_{i=1}^n J_ix \tanh(\beta^* J_ix)
&\ge \frac{x^\top J x - Ct\|Jx\|_2 - Ct^2}{\beta^*}
\ge \frac{F(J\beta^*)-Ct^2\beta^*}{(\beta^*)^2} - \frac{Ct\|Jx\|_2}{\beta^*}\notag\\
&\ge \frac{F(J\beta^*)/2}{(\beta^*)^2} -\label{eq:2second} \frac{Ct\|Jx\|_2}{\beta^*},
\end{align}
where the last inequality follows from the above assumption that $1\le t \le c\sqrt{F(J\beta^*)}$ for a sufficiently small $c >0$ and the $|\beta|$ is bounded by a constant.
Assume that \eqref{eq:2second} holds from now onward.
We derive that
\[
(\|Jx\|_2\beta^*)^2 + \|Jx\|_2\beta^* Ct - F(J\beta^*)/2 \ge 0.
\]
By solving the quadratic inequality for $\|Jx\|_2\beta^*$, we derive that
\[
\|Jx\|_2\beta^* \ge \frac{-Ct + \sqrt{C^2t^2 + 2F(J\beta^*)}}{2} 
\ge \sqrt{F(J\beta^*)/2} - Ct.
\]
Applying the bound $t\le c\sqrt{F(J\beta^*)}$, we derive that \begin{equation}\label{eq:Jx-bnd}
\|Jx\|_2 \ge \sqrt{F(J\beta^*)}/(2\beta^*).
\end{equation}

Further, from Lemma~\ref{lem:oneparam-derivative}, with probability at least $\log ne^{-ct}$,
\begin{equation}\label{eq:2der}
\varphi'(\beta^*) \le Ct\|Jx\|_2 + Ct,
\end{equation}
and assume that this holds for the rest of the proof.
By \eqref{eq:Jx-bnd} and \eqref{eq:2der},
\begin{equation} \label{eq:12}
\frac{\varphi'(\beta^*)}{\|Jx\|_2^2}
\le C \frac{t\|Jx\|_2+t}{\|Jx\|_2^2}
= \frac{Ct}{\|Jx\|_2} + \frac{Ct}{\|Jx\|_2^2}
\le \frac{2\beta^*Ct}{\sqrt{F(J\beta^*)}} + \frac{4(\beta^*)^2Ct}{F(J\beta^*)}
\le \frac{C't\beta^*}{\sqrt{F(J\beta^*)}},
\end{equation}
by the assumption of this lemma that $\sqrt{F(J\beta^*)}$ is at least a constant and since $\beta^*$ is bounded by a constant.
\end{proof}

By \eqref{eq:oneparam-finalbound} and by Lemma~\ref{lem:bnd-partition}, Corollary~\ref{cor:one-param-informal} follows, by taking $t = \log \log n + \log(1/\delta)$.
\section{The Lower Bound on $\|\hat{J} - J^*\|_F$}\label{s:lower_bound}
The estimation error of Theorem~\ref{t:general} is optimal in many interesting applications. In this Section, we present the proof of the general lower bound on the estimation rate of $\|\hat J - J^*\|$ that only depends on the packing numbers of the set $\mathcal{J}$. It is well known that the packing numbers are very related to the covering numbers of a set. We note that the expression for the lower bound does not match the one from the upper bound. However, this is to be expected, since for arbitrary sets $\mathcal{J}$ we might need more information other than the covering number to characterize learnability. However, in the particular case where $\mathcal{J}$ is a linear $k$-dimensional
subspace of matrices, 
the lower bound becomes (nearly)-optimal up to constants when $\|J^*\|<1$. This implies that doing MPLE in high temperature is optimal. 
We first present the general tool for proving our lower bounds, which is going to be Fano's method. The following result is adapted from \cite{duchi2016lecture}. 

\begin{lemma}[\cite{khas1979lower}, Fano's method] \label{l:fano}
Let $\mathcal{P} = \{P_J \colon J \in \mathcal{J}\}$ be a family of distributions indexed by matrices $J \in \mathcal{J}$. For a distribution $P$ from this family, denote $J(P)$ the interaction matrix corresponding to it. Let $\{P_v\}_{v \in \mathcal{V}}$ be a finite subset of distributions from $\mathcal{P}$ such that for any $v \neq v' \in \mathcal{V}$ we have $\|J(P_v) - J(P_{v'})\|_F \geq 2r$ for some $r > 0$ (this family is also called a $2r$-packing for $\|\cdot\|_F$). Assume that $V$ is a random variable that is uniform on the set $\mathcal{V}$, and conditional on $V = v$, and we draw a sample $X \sim P_v$. Then the miminax risk for estimation with respect to the Frobenius norm is lower bounded by
$$
r\lp(1 - \frac{I(V;X) + \log 2}{\log|\mathcal{V}|}\rp)
$$
\end{lemma}

A standard way of using Fano's method is finding a $r > 0$ and a family $\{P_v\}_{v \in \mathcal{V}}$ such that 
$$
\frac{I(V;X) + \log 2}{\log|\mathcal{V}|} \leq \frac{1}{2}
$$
This would immediately imply the lower bound $r/2$. 
To prove this statement, we need to upper bound the mutual information $I(V;X)$. If the mutual information is small, then this means that by seeing the sample $X$ we cannot infer with large probability the distribution $V$ of the family that generated it. 
To bound this quantity, we will use the following inequality, whose simple proof can be found in Chapter 7 of \cite{duchi2016lecture}.
\begin{equation} \label{eq:mutual-kl}
I(V;X) \leq \frac{1}{|\mathcal{V}|^2}\sum_{v,v'}D_{KL}(P_v\|P_{v'})
\end{equation}
Thus, if we can bound the KL-divergence between all pairs in the family, we also have the same bound for $I(V;X)$. The following Lemma aims to provide such a bound, when the two models are in high temperature.

\begin{lemma}\label{l:klbound}
Suppose $P_{J_1},P_{J_2}$ be the distributions of two Ising models with 
interaction matrices $J_1$ and $J_2$, respectively. Suppose furthermore that $\|J_1\|_\infty, \|J_2\|_\infty < 1 - \alpha$ for some $\alpha \in (0,1)$. Then, there is a constant $C = C(\alpha)$ such that
$$
D_{KL}(P_{J_1}\|P_{J_2}) \leq C \|J_2 - J_1\|_F^2
$$
\end{lemma}
\begin{proof}
In the following computations, the concept of the log partition function will be useful. We thus define 
$$
F(J) = \ln \sum_{y \in \{-1,+1\}^n}\exp\lp(\frac{1}{2} y^\top Jy\rp)
$$
which is the log partition function of a model with interaction matrix $J$. 
We also define the following function $g:[0,\|J_2 - J_1\|_F]\to \R$ by
$$
g(t) = F\lp(J_1 + tS\rp), \quad \text{where }
S = \frac{J_2 - J_1}{\|J_2 - J_1\|_F}.
$$
First, we notice that $g(0) = F(J_1), g(\|J_2 - J_1\|_F) = F(J_2)$.
Also, some simple calculations show that
\begin{align*}
g'(t) = \E_{y \sim J_1 + tS}\lp[\frac{1}{2} y^\top Sy\rp]\\
g''(t) = \Var_{y \sim J_1 + tS}\lp[\frac{1}{2} y^\top Sy\rp]
\end{align*}
where  the notation $y \sim A$ 
means that $y$ is sampled from an Ising model with interaction matrix $A$.
Now, we do some simple calculations with the KL-divergence
\begin{align*}
D_{KL}(P_{J_1}\|P_{J_2}) &= 
\E_{y\sim P_{J_1}} 
\ln \frac{P_{J_1}(y)}{P_{J_2}(y)}\\
&= \E_{y\sim P_{J_1}}
\lp[\frac{1}{2}y^\top (J_1 - J_2) y - F(J_1) + F(J_2)\rp]\\
&=F(J_2) - F(J_1) -  \|J_2 - J_1\|_F\E_{y \sim J_1}\lp[\frac{y^\top S y}{2}\rp] \\
&= g(\|J_2 - J_1\|_F) - g(0) - \|J_2 - J_1\|_Fg'(0)\\
&= \frac{\|J_2 - J_1\|_F^2}{2} g''(\xi)
\end{align*}
for some $\xi \in [0, \|J_2 - J_1\|_F]$. In the last step, we used Taylor's Theorem on $g$. We have
$$
\frac{\|J_2 - J_1\|_F^2}{2} g''(\xi) = \frac{\|J_2 - J_1\|_F^2}{2}  \Var_{y \sim J_1 + \xi S}\lp[\frac{1}{2} y^\top Sy\rp]=
\frac{1}{2} \Var_{y \sim J_1 + \xi S}\lp[\frac{1}{2} y^\top (J_2 - J_1)y\rp] \enspace.
$$
Hence, it suffice to bound this variance in order to obtain a bound on the KL-divergence. This is the variance of a second degree polynomial of an Ising model with interaction matrix $J_1 + \xi S$. It is easy to see
that as $t$ varies in $[0, \|J_2 - J_1\|_F]$, $J_1 + tS$ moves along the line segment connecting $J_1,J_2$. Hence, $J_1 + \xi S$ also belongs in this line segment. Since $\|J_1\|_\infty, \|J_2\|_\infty < 1 - \alpha$, it follows that any point in the line segment connecting them has infinity norm upper bounded by $1 - \alpha$. 
We conclude that $\|J_1 + \xi S\|_\infty$ defines an Ising model in high temperature. The task is thus to bound a second degree polynomial of an Ising model in high temperature. Using the bound of Theorem 2.1 in \cite{gheissari2018concentration} we obtain that there exists a constant $C = C(\alpha)$ such that
$$
\Var_{y \sim J_1 + \xi S}\lp[\frac{1}{2} y^\top (J_2 - J_1)y\rp] \leq C(\alpha) \|J_2 - J_1\|_F^2\enspace.
$$
This implies the result of this Lemma.
\end{proof}

Now that we have a bound for KL-divergence between two Ising models, we can prove our general lower bound.
\begin{comment}
\begin{theorem}\label{t:lb}
Let $\mathcal{P} = \{P_J \colon J \in \mathcal{J}\}$ be a family of Ising model distributions indexed by interaction matrices $J \in \mathcal{J}$. Suppose $\|J\|_\infty < 1 - \alpha$ for all $J \in \mathcal{J}$ for some $\alpha \in (0,1)$. Let $C(\alpha)$ be the coefficient specified in Lemma~\ref{l:klbound} and define the set
$$
U = \lp\{r \in \R^+ : \exists R>0 \exists \text{ set $\mathcal{K}$ of diameter at most $R$ s.t. :}  C(\alpha)R^2 + \log 2 \leq \frac{\log\mathcal{N}(\mathcal{J}\cap \mathcal{K},\|\cdot\|_F,2r)}{2}\rp\}
$$
Then, if 
$$
r^* = \sup U
$$
the minimax risk of estimating the interaction matrix $J$ in Frobenius norm is greater than $r^*/2$.
\end{theorem}
\end{comment}
\lb*
\begin{proof}
The proof is essentially an application of the Fano method when we use the upper bound on the KL proven in Lemma~\ref{l:klbound}. Consider any $r > 0$ that satisfies the requirements of the Theorem.
Then, by definition, there exists an $R>0$, such that the distance of any two matrices in $\mathcal{J}$ is bounded by $R$ in Frobenius norm. We need to define the following concept.

\begin{definition}[\cite{vershynin2018high}]
A subset $L$ of a metric space $(T,d,\epsilon)$ is $\epsilon$-\emph{separated} if $d(x,y) > \epsilon$ for any distinct $x,y \in L$. The largest possible cardinality of an $\epsilon$ separated subset of a given set $L \subseteq T$ is called the \emph{packing number} of $L$ and denoted by $\mathcal{P}(L,d,\epsilon)$ 
\end{definition}

The packing number of a set is the maximum amount of points that we can choose so that all of them are far from each other. 
Notice that this concept is very related to the requirements of Fano method. Indeed, to apply it we need a set of distributions that are far from each other. 

To use these concepts, let $\mathcal{V}$ be a maximum cardinality $2r$ separated subset of $\mathcal{J}$ with respect to the Frobenius norm. This means that
$$
|\mathcal{V}| = \mathcal{P}(\mathcal{J} ,\|\cdot\|_F,2r)
$$
Then, if $X,V$ are defined as in Lemma~\ref{l:fano}, we have that 
$$
I(V;X) \leq  \frac{1}{|\mathcal{V}|^2}\sum_{v,v' \in \mathcal{V}}D_{KL}(P_v\|P_{v'}) \leq C(\alpha) R^2
$$
which follows from Eq.~\eqref{eq:mutual-kl} and Lemma~\ref{l:klbound} since $\mathcal{V} \subseteq \mathcal{J}$ and $\mathcal{J}$ has diameter at most $R$ in Frobenius norm.

Now, by the assumption of this Theorem, we conclude that
$$
I(V;X) + \log 2 \leq C(\alpha)R^2 + \log 2 \leq \frac{\log\mathcal{N}(\mathcal{J},\|\cdot\|_F,2r)}{2}.
$$
We would like to have $|\mathcal{V}|$ appear on the right hand side. 
The covering and packing numbers of a set are connected by the following simple inequality, whose proof can be found in Chapter 4 of \cite{vershynin2018high}.
$$
\mathcal{N}(\mathcal{J},\|\cdot\|_F,2r) \leq \mathcal{P}(\mathcal{J},\|\cdot\|_F,2r)
$$
Hence, we conclude that
$$
I(V;X) + \log 2 \leq \frac{\log\mathcal{P}(\mathcal{J},\|\cdot\|_F,2r)}{2} = \frac{\log |\mathcal{V}|}{2}
$$
Hence, we have that
$$
\frac{I(V;X) + \log 2}{\log |\mathcal{V}|} \leq \frac{1}{2}
$$
Thus, by Lemma~\ref{l:fano} we conclude that the minimax risk is at least $r/2$, which concludes the proof.

\end{proof}

Theorem~\ref{t:lb} offers a lower bound in terms of the covering numbers. However, the exact dependence of the minimax risk on the covering numbers is not clear from this general formulation. We now offer a simple Corollary of this theorem that gives (nearly)-optimal lower bounds when $\mathcal{J}$ is a linear subspace of matrices with bounded infinity norm.

\lowerbound*

\begin{proof}
 Let $r = c_1\sqrt{k}$, where $c_1$ will be determined later.
We will prove that $r$ satisfies all the requirements of Theorem~\ref{t:lb} for some suitable constant $c_1$. 
Suppose that we renormalize every matrix $J_i$ as follows
$$
J_i ' = \frac{1}{2\|J_i\|_F} J_i
$$

We know that $\|J_i\|_F \geq k$ for all $i$.
Thus, for any set of $\alpha_i \in [-1,1]$, we have that
$$
\lp\|\sum_{i=1}^k \alpha_i J_i'\rp\|_\infty \leq k\frac{1}{2k} = \frac{1}{2}
$$
hence $\sum_{i=1}^k \alpha_i J_i' \in \mathcal{J}$.
Also
$$
\|J_i'\|_F = \frac{1}{2}.
$$
Now consider the family 
$$
\mathcal{V} = \lp\{\sum_{i=1}^k \sigma_i J_i' : \sigma_i \in \{-1,1\}\rp\}
$$
We will prove that the minimax risk for estimation in this family is lower bounded. Since $\mathcal{V} \subseteq \mathcal{J}$, this would immediately imply a lower bound for the estimation rate on $\mathcal{J}$.
Since the supports of $J_i$ are disjoint, these matrices are orthogonal with respect to the trace inner product.
This means that the maximum distance $R$ between any pair of matrices in $\mathcal{V}$ is at most
$
2\sqrt{k}
$.
We will show that we can find a large number of matrices in $\mathcal{V}$ which are all at distance at least $r$ from each other. This is equivalent to finding a lower bound for the packing number of $\mathcal{V}$.

To do so, we take a packing of the hypercube $\{-1,1\}^k$ with respect to the Hamming distance. Namely, we take a set $U \subseteq \{-1,1\}^k$ such that for any $\sigma,\sigma' \in U$, the Hamming distance between $\sigma$ and $\sigma'$, $\|\sigma-\sigma'\|_1/2$, is at least $k/3$. It is known that there exists such a set $\mathcal U$ with cardinality $e^{ck}$ for some universal constant $c>0$.

Given the set $\mathcal U$, we take the following set $\mathcal{V}'$ to be our $r$-packing of $\mathcal{V}$:
\[
\mathcal{V}' = \lp\{ \sum_{i=1}^k \sigma_i J'_i\colon \sigma \in \mathcal U\rp\}.
\]
Since the matrices $J'_i$ have disjoint support, they are orthonormal with respect to the Forbenious norm, hence for any $\sigma \ne \sigma' \in \mathcal U$,
\[
\lp\| \sum_{i=1}^k \sigma_i J'_i - \sum_{i=1}^k \sigma_i' J'_i\rp\|_F^2
=\sum_{i=1}^k \lp\| (\sigma_i-\sigma_i')^2 J'_i\rp\|_F^2
= \frac{1}{4}\sum_{i=1}^k (\sigma_i-\sigma_i')^2
\ge \frac{k}{3},
\]
using the fact that each $J'_i$ has Forbenius norm of $1/2$, and that $\sigma,\sigma' \in \mathcal U$, hence they differ in at least $k/3$ coordinates.
%From simple concentration properties on the hypercube we know that for a fixed $\pmb{\sigma}\in \{-1,1\}^k$, the number 
%of $\mathbf{y}\in \{-1,1\}^k$ that disagrees with $\pmb{\sigma}$ in at most $k/3$ coordinates is $\leq 2^k e^{-ck}$ for some constant $c$. 
%
%Hence, we can apply a greedy strategy to find these points. First we choose any $\pmb{\sigma}_1 \in \{-1,1\}^k$. Then, there are at least $2^k(1 - e^{-ck})$ elements that have at least $k/3$ different coordinates, so we choose one of them as $\pmb{\sigma}_2$. Then, we choose a $\pmb{\sigma}_3$ that has at least $k/3$ different coordinates with both $\pmb{\sigma}_1,\pmb{\sigma}_2$ and so on. In the $i$-th step, there are at least $2^k(1 - i e^{-ck})$ elements that have at least $k/3$ different coordinates with all the first $i$ elements that we have chosen. Hence, we can keep finding elements for at least $\Omega(e^{ck})$ steps, which gives us $\Omega(e^{ck})$ points. 
Hence, by setting $r = \sqrt{k/3}/ 2$, we just found a $2r$ packing on the set $\mathcal{V}$ that contains at least $e^{ck}$ elements. Hence,
$$
\log \mathcal{P}(\mathcal{V},\|\cdot\|_F, 2r) \geq ck
$$
for some $c$. Now notice that in order to apply Theorem~\ref{t:lb}, the inequality that we have to satisfy is
\begin{equation}\label{eq:lb-P}
C(\alpha)R^2 + \log 2 \leq \frac{\log\mathcal{P}(\mathcal{V},\|\cdot\|_F,2r)}{2},
\end{equation}
where $R = \max_{J\ne J' \in \mathcal{V}} \|J-J'\|_F$
(in fact, the statement of Theorem~\ref{t:lb} requires a similar inequality with $\mathcal{P}$ replaced by $\mathcal{N}$, but it follows from the proof that Eq.~\eqref{eq:lb-P} suffices).
We essentially want to reduce $R$ to a suitable multiple of $\sqrt{k}$ so that this inequality will hold. To do that, we simply scale down the $J_i'$ by a suitable constant. Of course, we scale down $r$ by the same constant. Notice that the packing numbers will not be affected by this, since we just scaled the whole space. Hence, the right hand side remains the same, while the left hand side is properly reduced. Hence, this inequality will hold for $r = c_1\sqrt{k}$ for some suitable $c_1$.
This means by Theorem~\ref{t:lb} that the minimax risk of estimation in Frobenius norm is $\Omega(\sqrt{k})$.
\end{proof}
\subsection*{Acknowledgements.}
We would like to thank Dheeraj Nagaraj for interesting discussions and help in proving the lower bound, and Frederic Koehler for the helpful discussion on the log partition function.

This work was supported by NSF Awards IIS-1741137, CCF-1617730
and CCF-1901292, by a Simons Investigator Award, by the DOE PhILMs project (No. DE-AC05-76RL01830), and by the DARPA award HR00111990021.
\appendix
\section{Algorithm for Maximizing Pseudo-Likelihood}
\label{sec:optimization}
In this Section, we show that there exists a polynomial time algorithm for the MPLE, which is the content of  Lemma~\ref{lem:optimization}.
The central problem is maximizing the pseudo-likelihood
with the constraint that the matrix we output has
low infinity norm.
To achieve this, we use a regularizer of the form
$\lambda\max(0,\|A_{\pmb{\beta}}\|_\infty - M)$. This
ensures that the optimal solution that the algorithm outputs 
will have small infinity norm. 
We should also argue that the output
of the regularized procedure is close to the
minimizer of $\varphi$.
To do this,
we need to calculate bounds for $\varphi$ and its 
derivative. Since the regularized part
is not differentiable, we have to use subgradient
optimization methods. 
\begin{lemma}\label{lem:optimization}
	There exists an algorithm that outputs $\hat{J}$ satisfying $\hat{J} \in V$ and $\|\hat{J}\|_\infty \le 2M$, such that 
	$$
	\varphi(\hat{J}) \le \varphi(J^*) + \epsilon
	$$
	
	and runs in time $O(k^2n^6M^2/\epsilon^2)$.
\end{lemma}

\begin{algorithm}\label{algorithm}
\caption{Optimization Procedure}
    \hspace*{\algorithmicindent} \textbf{Input:}\\ 
    \hspace*{\algorithmicindent}\hspace*{\algorithmicindent} Basis of matrices $J_1,\ldots,J_k$\\
    \hspace*{\algorithmicindent}\hspace*{\algorithmicindent} Accuracy $\epsilon$\\
    \hspace*{\algorithmicindent}\hspace*{\algorithmicindent} Sample $(x_1,\ldots,x_n)$\\
    \hspace*{\algorithmicindent}\hspace*{\algorithmicindent} Infinity norm bound $M$\\
    \hspace*{\algorithmicindent} \textbf{Output:}\\ 
     \hspace*{\algorithmicindent}\hspace*{\algorithmicindent} Vector $\hat{\beta}$
\begin{algorithmic}[1]
\STATE $(A_1,\ldots,A_k) := $ \textsc{GRAM-SCHMIDT}$(J_1,\ldots,J_k)$ \COMMENT{Find orthonormal basis}
\FOR {i := $1, \ldots,k$} 
\STATE $\beta_i^0 := 0$ \COMMENT{Initializations}
\ENDFOR
\STATE $T = \frac{M^2n^4k}{\epsilon^2}$ \COMMENT{Number of steps of gradient descent}
\STATE $\eta := \frac{M}{n\sqrt{k}\sqrt{T}}$ \COMMENT{stepsize}
\FOR {t := $1,\ldots,T$}
\STATE $U := \beta_1^t A_1 + \ldots \beta_k^t A_k$ \COMMENT{The interaction matrix at time $t$}
\FOR {i := $1,\ldots,k$}
\STATE $S_i := 0$
\FOR{j := $1,\ldots,n$}
\STATE $S_i:= S_i + ((A_i)_jx)\lp(\tanh(U_jx) - x_j\rp)$\COMMENT{Compute the derivative}
\ENDFOR
\ENDFOR
\IF{$\|U\|_\infty \ge M$}
\STATE $MAX := -\infty$ \COMMENT{If regularized part is nonzero, also need subgradient}
\STATE $ARGMAX := 1$ \COMMENT{First find max row sum}
\FOR{l := $1,\ldots,n$}
\STATE $TEMP := 0$
\FOR{v := $1,\ldots,n$}
\STATE $TEMP := TEMP + |U_{lv}|$
\ENDFOR
\IF{$TEMP > MAX$}
\STATE $MAX := TEMP$
\STATE $ARGMAX := l$
\ENDIF
\ENDFOR
\FOR{i := $1,\ldots, k$}
\FOR{v := $1,\ldots,n$}
\STATE $S_i := S_i + 5n \cdot \mathrm{sgn}(\beta_i^t)\cdot\|(A_i)_{MAX,v}|$\COMMENT{compute subgradient for max row}
\ENDFOR
\ENDFOR
\ENDIF
\FOR {i := $1,\ldots,k$}
\STATE $\beta_i^{t+1} = \beta_i^t - \eta S_i$\COMMENT{Iteration of subgradient descent}
\ENDFOR
\ENDFOR
\STATE \textbf{Output} $\hat{\beta} := \frac{1}{T}\sum_{t=1}^T \beta^t$
\end{algorithmic}
\end{algorithm}

\begin{proof}
The algorithm essentially solves the constrained optimization
problem described in ~\ref{eq:pmle}.
To design the optimization algorithm, it will be more 
convenient to work with a specific base of matrices
that span $\mathcal{V}$. For this reason,
let $\{A_1,\ldots,A_k\}$ be a fixed orthonormal basis of $\mathcal{V}$ with
respect to the inner product induced by the frobenius norm.
Then, each $J \in \mathcal{V}$ can be uniquely written in the form 
$$
J = \sum_{i=1}^k \beta_i A_i
$$
We use the notation $\pmb{\beta} = (\beta_1,\ldots,\beta_k)$
and $A_{\pmb{\beta}} = \sum_{i=1}^k \beta_i A_i$.
Thus, minimizing $\varphi(J)$ subject to $J \in \mathcal{J}$
is the same as minimizing 
$$
\psi(\pmb{\beta}) = \varphi(A_{\pmb{\beta}})
$$
subject to the constraint $\|A_{\pmb{\beta}}\|_\infty \le M$.
Thus, in order to solve ~\ref{eq:pmle}, we can equivalently
solve 
\begin{equation}\label{eq:alt_pmle}
    \arg\min_{\pmb{\beta}:\|A_{\pmb{\beta}}\|_\infty \le M}
    \psi(\pmb{\beta})
\end{equation}
One way to solve this is to use constrained
optimization.
 However, we would like to avoid the complicated
projection procedure associated with this strategy.
We note that projecting to the set of matrices with 
low infinity norm does not directly translate to a similar
procedure in the space of $\mathbf{\beta}$,
since $A_i$ might have large infinity norm.
Instead, we will solve the following regularized
optimization problem:
$$
\arg\min_{\pmb{\beta}} \lp(\psi(\pmb{\beta})+ \lambda\max(0,\|A_{\pmb{\beta}}\|_\infty - M)\rp)
$$
Set $h(\pmb{\beta}) = \psi(\pmb{\beta})+ \lambda\max(0,\|A_{\pmb{\beta}}\|_\infty - M)$, which is
clearly a convex function.
We will describe a way to pick $\lambda$ shortly after.
Denote by $\pmb{\beta}_1$ the solution of this optimization problem and by $\pmb{\beta}^*$ the one corresponding to $J^*$.
Also, suppose the algorithm outputs a point $\hat{\pmb{\beta}}$
for accuracy $\epsilon$.
The details are given in Algorithm~\ref{algorithm}.
We will prove that
$\|A_{\pmb{\beta}_1}\|_\infty, \|A_{\hat{\pmb{\beta}}}\|_\infty \le 3M$.
First, for $\pmb{\beta} = \pmb{0}$ we have
$\|A_{\pmb{0}}\|_\infty \le M$
and by equation~\ref{eq:pmleform} we get $\psi(\pmb{0}) =n\log2$.

We set $\lambda = 5n$. We now examine
what the value of the function is $\pmb{\beta}$ such that 
$\|A_{\pmb{\beta}}\|_\infty \ge 3M$.
If we manage to show that it is greater than
the value at $0$, then it is clear that the minimizer will
not lie in this set. 
Using the inequality $\cosh(x) \ge \exp{-|x|}$ and equation~\ref{eq:pmleform} we have:
$$
\psi(\pmb{\beta}) =
\sum_{i=1}^n \lp(\log\cosh\lp((A_{\pmb{\beta}})_ix\rp) - x_i (A_{\pmb{\beta}})_ix + \log 2\rp)
\ge \sum_{i=1}^n \lp(-|(A_{\pmb{\beta}})_ix| - x_i(A_{\pmb{\beta}})_ix + \log 2\rp)
$$
Using the triangle inequality and the fact that
$x$ is a $\{+1, -1\}^n$ vector, we have:
$|(A_{\pmb{\beta}})_ix| \le \|A_{\pmb{\beta}}\|_\infty$.
Also, since $\|A_{\pmb{\beta}}\|_2 \le \|A_{\pmb{\beta}}\|_\infty$ we have:
$$
\sum_{i=1}^n x_i(A_{\pmb{\beta}})_ix = x^\top A_{\pmb{\beta}}x
\le \|A_{\pmb{\beta}}\|_2 \|x\|_2^2 = n\|A_{\pmb{\beta}}\|_2 \le n \|A_{\pmb{\beta}}\|_\infty
$$
Overall, we get:
$$
\psi(\pmb{\beta}) \ge -n\|A_{\pmb{\beta}}\|_\infty -n\|A_{\pmb{\beta}}\|_\infty + n\log 2 = -2n \|A_{\pmb{\beta}}\|_\infty + n\log 2
$$
So, the value of the optimized function at $\pmb{\beta}$ is
$$
h(\beta) = \psi(\beta) + 5n (\|A_{\pmb{\beta}}\|_\infty - M)
\ge  -2n \|A_{\pmb{\beta}}\|_\infty + n\log 2 
+3n \|A_{\pmb{\beta}}\|_\infty = n\log 2 = h(\pmb{0})+ n\|A_{\pmb{\beta}}\|_\infty
$$
Hence, we conclude that $\|A_{\pmb{\beta}_1}\|_\infty \le 3M$.
Moreover, we have
$$
h(\hat{\pmb{\beta}}) - h(\pmb{\beta}_1)\le \epsilon \le Mn
\implies h(\hat{\pmb{\beta}}) \le Mn + n\log 2
\implies \|A_{\hat{\pmb{\beta}}}\|_\infty \le 3M
$$
As for the guarantee of the algorithm, we notice that:
$$
\psi(\hat{\pmb{\beta}}) \leq h(\hat{\pmb{\beta}})
\leq h(\pmb{\beta}_1) + \epsilon
\le h(\pmb{\beta}^*) + \epsilon = \psi(\pmb{\beta}^*) + \epsilon
$$
since $\|A_{\pmb{\beta}}^*\|_\infty \leq M$.
Thus, if we denote by $\hat{J} = A_{\hat{\pmb{\beta}}}$,
we conclude that:
$$
\varphi(\hat{J}) \le \varphi(J^*) + \epsilon
$$
which is what we wanted to prove.
It remains now to argue about the computational complexity
of the procedure. The algorithm we will use is subgradient descent, which performs reasonably well when the subgradient is
upper bounded.
The following Theorem can be found in \cite{bubeck2015convex}.
\begin{lemma}[\cite{bubeck2015convex}]\label{t:gd}
Suppose $f$ is a convex function with minimum $x^*$ and $g \in \partial f$ is
a subgradient such that $\|g\| \le L$. Suppose we are running the
following iterative procedure:
$$
x_{t+1} = x_t - \eta g(x_t)
$$
with $\|x_1 - x^*\| \le R$.
If we choose $\eta = R/(L\sqrt{t})$
then 
$$
f\lp(\frac{1}{t} \sum_{i=1}^t f(x_i)\rp) - f(x^*) \leq \frac{RL}{\sqrt{t}}
$$
\end{lemma}
Hence, if we want accuracy $\epsilon$, we need to run the algorithm
for $t = O(R^2L^2/\epsilon^2)$ iterations.
We now argue that both $R$ and $L$ are polynomially
bounded in our case. 
To bound a subgradient of $h$, we begin by bounding each partial
derivative of $\psi$, since for this function it is also a subgradient. 
First of all, it is easy to see that $\partial \psi(\pmb{\beta})/\partial \beta_i \le 2\sqrt{n}$.
Indeed, by equation~\ref{eq:pmle_der} we get:
\begin{align*}
\lp|\frac{\psi(\pmb{\beta})}{\partial \beta_i}\rp| &= \lp|\sum_{k=1}^n ((A_i)_kx)\lp(\tanh ((A_{\pmb{\beta}})_kx) - x_k\rp)\rp|
\le
\sum_{k=1}^n |(A_i)_kx|\lp|\tanh ((A_{\pmb{\beta}})_kx) - x_k\rp|\\
&\le \sum_{k=1}^n |(A_i)_kx| \lp(|\tanh ((A_{\pmb{\beta}})_kx)| +1\rp)
\le 2\sum_{k=1}^n |(A_i)_kx| \\
&\le 2\sum_{k=1}^n \sum_{j=1}^n |(A_i)_{kj}|
\le 2 \sqrt{n}\|A_i\|_F = 2\sqrt{n}
\end{align*}
In the preceding calculation, we used the Cauchy Schwarz
inequality along with the fact that $A_i$ belongs to the orthonormal basis.
Now, we turn our attention to the non-smooth part of $h$.
We use the following general fact about subgradients.
\begin{lemma}[folklore]
    Suppose $f_1,f_2:\R^n \mapsto \R$ are differentiable convex
    functions. Then, $h = \max(f_1,f_2)$ has a subgradient
    of the form
    $$
    g(x) = \lp\{ 
    \begin{array}{ll}
         \nabla f_1(x) \text{ , if } f_1(x) \ge f_2(x)\\
         \nabla f_2(x) \text{ , otherwise}
    \end{array}
    \rp\}
    $$
\end{lemma}
This means that to bound the subgradient of $\lambda\max(0,\|A_{\pmb{\beta}}\|_\infty - M)$
we just need to bound the gradient of each function in the max. The function $0$ obviously has gradient $0$. We have
$$
\|A_{\pmb{\beta}}\|_\infty = \max_{u \in [n]}
\sum_{v=1}^n |(A_{\pmb{\beta}})_{uv}|
$$
Hence, to bound the subgradient of this function,
we focus on a fixed row $u$. 
Set 
$$
L(\pmb{\beta}) = \sum_{v=1}^n |(A_{\pmb{\beta}})_{uv}|
= \sum_{v=1}^n \lp|\sum_{i=1}^k\beta_i(A_i)_{uv}\rp|
$$.
Using the fact that a subgradient for $|x|$ is $\mathrm{sgn}(x)$, we obtain that a subgradient of $L$ w.r.t.
variable $\beta_i$ is 
$\mathrm{sgn}(\beta_i)\sum_{v=1}^n|(A_i)_{uv}|$,
which is clearly bounded in norm by $\sum_{v=1}^n|(A_i)_{uv}|\le \sqrt{n}$. Hence, the subgradient
w.r.t. $\beta_i$ of the function $\lambda\max(0,\|A_{\pmb{\beta}}\|_\infty - M)$
is bounded in absolute value by $\lambda \sqrt{n} = 5n\sqrt{n}$.
This means that the subgradient $g_i$ of $h$ is $O(n\sqrt{n})$
in absolute value.
It follows that $\|g\|_2 \le n\sqrt{k}\sqrt{n}$.
Finally, by choosing $x_1 = 0$ we have 
$$
\|x_1 - \pmb{\beta}^*\| = \|A_{x_1} - A_{\pmb{\beta}^*}\|_F \leq \sqrt{n}\|A_{x_1} - A_{\pmb{\beta}^*}\|_\infty \le M\sqrt{n}
$$
Hence, after $t = O(n^4M^2k/\epsilon^2)$ rounds we achieve error
of at most $\epsilon$.
The only part of the algorithm we haven't analysed is the 
Gram-Schmidt computation. It is well known that
for a subspace of $\R^k$ of dimension $l$, 
the Gram-Schmidt procedure takes $O(ml^2)$.
For $m = k, l= n^2$ we see that this 
is less than the complexity for the optimization part.

The last issue we should address to get the time complexity
of the algorithm is the calculation of the subgradient.
The gradient of $\psi$ amounts to the calculation of $A_{\pmb{\beta}}x$, which takes $O(n^2k)$ time($A_ix$ can be
precomputed for all $i$). 
The calculation of the subgradient for the nonsmooth part is
easy once we determine which row has the maximum absolute sum
for the particular value of $\pmb{\beta}$. This also takes $O(kn^2)$ time, since we have to calculate the sum of the
matrices in each row. 
Once we do that, a precomputation of the row sums of each matrix
can give us the subgradient in $O(1)$ time.
Overall, computation of the gradient takes $O(kn^2)$ time,
which means that our algorithm runs in $O(k^2n^6M^2/\epsilon^2)$ time.
\end{proof}

%Now, by using ~\ref{eq:main2} and the hypothesis of the lemma
%\vardis{this is where we use it}we get
%\begin{equation}\label{eq:final1}
%\lp|h'(0) - h'(\|\Delta\|_F)\rp| = 
%\lp|\frac{\partial \varphi(J^*)}{\partial A}
%- \frac{\partial \varphi(\hat{J})}{\partial A}\rp|
%\leq 2C\sqrt{k \log n + \log(1/\delta)}\sqrt{\min_{J \in \mathcal{J}} \frac{\partial^2 \varphi(J)}{\partial A^2}}
%\end{equation}
%At the same time, utilising ~\ref{eq:main1} we get
%\begin{equation}\label{eq:final2}
%    \lp|h'(0) - h'(\|\Delta\|_F)\rp| = 
%    \int_0^{\|\Delta\|_F}h''(t)dt 
%    \geq \|\Delta\|_F \min_{t \in \lp[0,\|\Delta\|_F\rp]}
%    h''(t) \geq \|\Delta\|_F \min_{J \in \mathcal{J}} \frac{\partial^2 \varphi(J)}{\partial A^2}
%\end{equation} 
%Putting ~\ref{eq:final1} and  ~\ref{eq:final2} together and using ~\ref{eq:main1},
%we obtain
%$$
%\|\Delta\|_F \leq 2C\frac{\sqrt{k \log n + \log(1/\delta)}}{\sqrt{\min_{J \in \mathcal{J}} \frac{\partial^2 \varphi(J)}{\partial A^2}}}
%\leq \frac{2C\sqrt{k \log n + \log(1/\delta)}}{c'},
%$$
%which is what we wanted to prove.

%\input{temp-additional-lemmas}
%\input{chaining-bounds}
\section{Dobrushin's Uniqueness Condition}
\label{sec:dobrushin}

Here we present Dobrushin's uniqueness condition in its full generality.
First we define the influence of a node $j$ on a node $i$.
\begin{definition}[Influence in Graphical Models]
	\label{def:influence}
	Let $\pi$ be a probability distribution over some set of variables $V$. Let $B_j$ denote the set of state pairs $(X,Y)$ which differ only in their value at variable $j$. Then the influence of node $j$ on node $i$ is defined as
	\begin{align*}
	I(j,i) = \max_{(X,Y) \in B_j} \norm{\pi_i(. | X^{-i}) - \pi_i(. | Y^{-i})}_{TV}
	\end{align*}
\end{definition}

Now, we are ready to state Dobrushin's condition.
\begin{definition}[Dobrushin's Uniqueness Condition]
	\label{def:dobrushin-influence}
	Consider a distribution $\pi$ defined on a set of variables $V$. Let
	\begin{align*}
	\alpha = \max_{i \in V} \sum_{j \in V} I(j,i)
	\end{align*}
	$\pi$ is said to satisfy Dobrushin's uniqueness condition if $\alpha < 1$.
\end{definition}
Notice that $\| J^*\|_{\infty} < 1$ implies Dobrushin's condition and a proof can found in \cite{chatterjee2005concentration}.
A generalization of this condition was given by \cite{hayes2006simple} who defined the generalized Dobrushin's condition as $\|I\|_2 < 1$, where $I=I(i,j)$ is the influence matrix. This condition, while being weaker, retains most desirable properties of the original Dobrushin's condition.

\section{Technical Proofs} \label{app:technical}
We begin with the proof of Lemma~\ref{lem:derivative-concentration}
and then move to the proof of Lemma~\ref{lem:conc-mainpart}.

\subsection{Proof of Lemma~\ref{lem:derivative-concentration}}
\label{s:conc_aux}
In this section, we prove the following lemma:
\lemmaderivativeconcentration*

We will decompose $\psi_j(x;A)$ to parts that depend on $x_{I_j}$ and parts the depend on the parameters that we condition on, $x_{-I_j}$:
\begin{equation} \label{eq:7}
\psi_j(x;A)=\sum_{i\in I_j} (A_{i,I_j} x_{I_j} + A_{i,-I_j}x_{-I_j})(x_i - \tanh(J^*_{i,I_j} x_{I_j} + J^*_{i,-I_j} x_{-I_j}))
= \sum_{i\in I_j} (A_{i}' x' + b_i')(x_i' - \tanh(J'_{i} x' + h_i')),
\end{equation}
where 
\[
A' = A_{I_j I_j} ;\quad
b' = A_{I_j, -I_j} x_{-I_j}; \quad
J' = J^*_{I_j,I_j}; \quad
h' = J^*_{I_j,-I_j} x_{-I_j}; \quad \text{and }
x' = x_{I_j}.
\]
Notice that $A',b',J',h'$ are fixed conditioned on $x_{-I_j}$ and $x'$ is distributed as the conditional distribution of $x_{I_j}$ conditioned on $x_{-I_j}$. Furthermore, $x'$ is a $(1/2,\gamma)$-Ising model, with interaction matrix $J'$ and external field $h'$, conditioned on $x_{-I_j}$. Hence, the following lemma will imply that the right hand size of \eqref{eq:7} concentrates conditioned on $x_{-I_j}$.
\begin{lemma} \label{lem:derivative-calc}
	Let $x$ be a $(1/2,\gamma)$ Ising model over $\{-1,1\}^m$ with interaction matrix $J$ and external field $h$. Let $A$ be a symmetric real matrix of dimension $m\times m$ with zeros on the diagonal, let $b \in \mathbb{R}^m$ be a vector and let
	\[
	f(x) = \sum_{i \in [m]}(A_ix + b_i) (x - \tanh(J_ix + h)).
	\]
	Then, for any $t > 0$,
	\[
	\Pr[|f(x)| \ge t]
	\le \exp\left(-c\min\left(\frac{t^2}{\|\E Ax+b\|_2^2}, \frac{t^2}{\|A\|_F^2}, \frac{t}{\|A\|_2}\right)\right),
	\]
	where $c>0$ is lower bounded by a constant constant whenever $\|A\|_\infty$, $\|b\|_\infty$, $\|A'\|_\infty$ and $\|b'\|_\infty$ are bounded from above by a constant.
\end{lemma}

First we derive Lemma~\ref{lem:derivative-concentration} based on Lemma~\ref{lem:derivative-calc}. Substituting $A = A'$, $b = b'$, $J = J'$ and $h = h'$ and $x= x'$, we derive that,
\begin{align*}
\Pr\left[\psi_j(x;A) \ge t ~\middle|~ x_{-I_j} \right]
&\le \exp\left(-c\min\left(\frac{t^2}{\|\E A'x+b'\|_2^2}, \frac{t^2}{\|A'\|_F^2}, \frac{t}{\|A'\|_2}\right)\right)\\
&= \exp\left(-c\min\left(\frac{t^2}{\|\E A_{I_j} x\|_2^2}, \frac{t^2}{\|A_{I_j,I_j}\|_F^2}, \frac{t}{\|A_{I_j,I_j}\|_2}\right)\right)\\
&\le \exp\left(-c\min\left(\frac{t^2}{\|\E A x\|_2^2}, t^2, \frac{t}{\|A\|_2}\right)\right),
\end{align*}
using the fact that $\|A_{I_j,I_j}\|_2 \le \|A\|_2$ and that $\|A_{I_j,I_j}\|_F \le \|A\|_F = 1$ for all $A \in \mathcal{A}$. This concludes the proof of Lemma~\ref{lem:derivative-concentration}.
Lastly, we prove Lemma~\ref{lem:derivative-calc}.
\begin{proof}[Proof of Lemma~\ref{lem:derivative-calc}]
	By abuse of notation, define 
	$\tanh \colon \mathbb{R}^m \to \mathbb{R}^m$ 
	by 
	\[\tanh(y_1,\dots,y_m) = (\tanh(y_1),\dots,\tanh(y_m)).
	\]
	Decompose $f(x)$ as
	\[
	f(x) = g(x)^\top h(x),
	\]
	where $g(x) = Ax+b$ and $h(x) = x - \tanh(Jx+h)$.
	
	We start by bounding $\sum_i (D_i f(x))^2$.
	Decompose
	\begin{align}
	2D_i f(x) 
	&=f(x_{i+}) - f(x_{i-})
	= g(x_{i+})^\top(h(x_{i+}) - h(x_{i-}))
	+ (g(x_{i+}) - g(x_{i-}))^\top h(x_{i-})\notag\\
	&= (g(x_{i+})-g(x))^\top(h(x_{i+}) - h(x_{i-})) + g(x)^\top(h(x_{i+}) - h(x_{i-})) \notag\\
	&+ (g(x_{i+}) - g(x_{i-}))^\top(h(x_{i-})-h(x)) 
	+ (g(x_{i+})-g(x_{i-}))^\top h(x).\notag
	\end{align} 
	By Cauchy Schwartz, for any $a,b,c,d \in \mathbb{R}$, we have that $(a+b+c+d)^2 \le 4 (a^2+b^2+c^2+d^2)$. Hence,
	\begin{align}
	\sum_{i=1}^m(D_i f(x))^2
	&\le \sum_{i=1}^m ((g(x_{i+})-g(x))^\top(h(x_{i+}) - h(x_{i-})))^2
	+  \sum_{i=1}^m (g(x)^\top(h(x_{i+}) - h(x_{i-})))^2 \notag\\
	&+ \sum_{i=1}^m ((g(x_{i+}) - g(x_{i-}))^\top(h(x_{i-})-h(x)))^2
	+ \sum_{i=1}^m ((g(x_{i+})-g(x_{i-}))^\top h(x))^2. \label{eq:55}
	\end{align}
	
	We will bound the terms in the right hand side of \eqref{eq:55} one after the other. 
	
	Term 4:
	\[
	(g(x_{i+})-g(x_{i-}))^\top h(x)
	= (x_{i+}-x_{i-})^\top A (x - \tanh(Jx+h))
	= 2e_i^\top A(x-\tanh(Jx+h)).
	\]
	Summing over all $i$, we get
	\[
	\sum_i ((g(x_{i+})-g(x_{i-}))^\top h(x))^2
	= 4 \|A(x-\tanh(Jx+h))\|_2^2.
	\]
	
	Term 3: using the fact that $\tanh$ is $1$-Lipschitz and Cauchy Schwartz,
	\begin{align*}
	|(g(x_{i+}) - g(x_{i-}))^\top(h(x_{i-})-h(x))|
	&= |2e_i^\top A(x_{i-} - \tanh(Jx_{i-}+h) - x - \tanh(Jx+h))|\\
	&\le 2 \|e_i^\top A\|_2 \|x_{i-} - \tanh(Jx_{i-}+h) - x - \tanh(Jx+h)\|_2\\
	&\le 2 \|e_i^\top A\|_2 (\|x_{i-} -x\|_2 
	+ \|\tanh(Jx_{i-}+h) - \tanh(Jx+h)\|_2)\\
	&\le 2 \|e_i^\top A\|_2 (2 
	+ \|Jx_{i-}+h-Jx-h\|_2)\\
	&\le 2 \|e_i^\top A\|_2 (2 
	+ \|J\|_2 \|x_{i-}-x\|_2) \\
	&\le C \|e_i^\top A\|_2.
	\end{align*}
	Summing over all $i$, we get
	\[
	\sum_i |(g(x_{i+}) - g(x_{i-}))^\top(h(x_{i-})-h(x))|^2
	\le C \sum_i \|e_i^\top A\|_2^2
	= C \|A\|_F^2.
	\]
	
	Term 1:
	\begin{align*}
	|(g(x_{i+})-g(x))^\top(h(x_{i+}) - h(x_{i-}))|
	= |(x_{i_+}-x)^\top A (x_{i+} - \tanh(Jx_{i+}+h) - x_{i-} - \tanh(Jx_{i-}+b))|\\
	\le \|(x_{i_+}-x)^\top A\|_2\|x_{i+} - \tanh(Jx_{i+}+h) - x_{i-} - \tanh(Jx_{i-}+b)\|_2
	\le C \|e_i^\top A\|_2,
	\end{align*}
	using a bound similar to the above.
	Summing over all $i$ we get
	\[
	\sum_i |(g(x_{i+})-g(x))^\top(h(x_{i+}) - h(x_{i-}))|^2 \le C \|A\|_F^2.
	\]
	
	Term 2:
	\[
	\sum_i (g(x)^\top(h(x_{i+}) - h(x_{i-})))^2
	= \|W g(x)\|_2^2,
	\]
	where $W$ is a matrix of size $m\times m$ such that 
	\begin{equation} \label{eq:28}
	W_{ij} = h(x_{i+})_j - h(x_{i-})_j
	= (x_{i+})_j - (x_{i-})_j
	-\tanh(J_j^\top x_{i+} + h) 
	+ \tanh(J_j^\top x_{i-} + h).
	\end{equation}
	Using the Lipschitzness of $\tanh$ and the triangle inequality,
	\[
	|W_{ij}| \le |(x_{i+})_j - (x_{i-})_j| + |J_j^\top (x_{i+}-x_{i-})|
	= 2 (\mathbf{1}(i=j) + |J_{ij}|).
	\]
	We obtain that
	\begin{equation}\label{eq:29}
	\|W\|_2 \le \|W\|_\infty \le 2\|J\|_\infty + 2 \le C.
	\end{equation}
	Hence, $\|Wg(x)\|_2^2 \le \|W\|_2^2 \|g(x)\|_2^2 \le C^2 \|Ax + b\|_2^2$.
	
	To summarize, by \eqref{eq:55} and the calculations below, we obtain that
	\[
	\sum_i (D_i f(x))^2
	\le C(\|A\|_F^2 + \|Ax+b\|_2^2 + \|A(x-\tanh(Jx+h)))\|_2^2.
	\]
	Define the pseudo discrete derivative to be a function of $2m+1$ coordinates, such that for coordinate $i$, $i \in [m]$, we have $\tilde{D}_i(x) = C(A_i^\top x + b_i)$, in coordinate $m+i$ we have $\tilde{D}_{n+i}(x) = C A_i^\top(x - \tanh(Jx+h))$, and in coordinate $2m+1$ we have $\tilde{D}_{2m+1}(x) = C\|A\|_F$.
	
	Next, we like to define a pseudo discrete Hessian. For this purpose, we bound $\sum_{j=1}^m D_j(\xi^\top \tilde D(x))^2$, for any fixed $\xi \in \mathbb{R}^{2m+1}$. Note that using the fact that $\tilde{D}_{2m+1}$ is constant in $x$ and using Cauchy Schwartz, 
	\begin{align}
	\sum_{j=1}^m D_j(\xi^\top \tilde D(x))^2
	&= \sum_{j=1}^m \left(\sum_{i=1}^{2m+1} \xi_i(\tilde{D}_i(x_{j+}) - \tilde{D}_i(x_{j-}))\right)^2\notag\\
	&\le 2 \sum_{j=1}^m \left(\sum_{i=1}^{m} \xi_i(\tilde{D}_i(x_{j+}) - \tilde{D}_i(x_{j-}))\right)^2
	+ 2\sum_{j=1}^m \left(\sum_{i=n+1}^{2m} \xi_i(\tilde{D}_i(x_{j+}) - \tilde{D}_i(x_{j-}))\right)^2.\label{eq:88}
	\end{align}
	We will bound both terms from the right hand size of \eqref{eq:88}. Starting with the first term, for any $i \in [m]$ we have 
	\[
	\tilde{D}_i(x_{j+}) - \tilde{D}_i(x_{j-})
	= 2 A_{ij}.
	\]
	Hence, 
	\[
	\sum_{j=1}^m \left(\sum_{i=1}^{m} \xi_i(\tilde{D}_i(x_{j+}) - \tilde{D}_i(x_{j-}))\right)^2
	= 4 \sum_j \left(\sum_i \xi_i A_{ij}\right)^2
	= 4\|\xi_{1\cdots m}^\top A\|_2^2,
	\]
	where $\xi_{1\cdots m} = (\xi_1,\dots,\xi_m)$.
	For the second term of \eqref{eq:88}, we have:
	\begin{align*}
	&\sum_{j=1}^m \left(\sum_{i=n+1}^{2m} \xi_i(\tilde{D}_i(x_{j+}) - \tilde{D}_i(x_{j-}))\right)^2\\
	&= \left\| \xi_{m+1\cdots2m}^\top A \left(\sum_{i \in [n]} x_{i+} - x_{i-}
	-\tanh(A' x_{i+} + h) 
	+ \tanh(A' x_{i-} + h)\right)\right\|_2^2
	= \|\xi_{m+1\cdots2m}^\top A W\|_2^2,
	\end{align*}
	for the matrix $W$ that we defined in the calculations of the discrete derivative, in \eqref{eq:28}. By \eqref{eq:29} we get that
	\[
	\|\xi_{m+1\cdots2m}^\top AW\|_2^2 \le \|\xi_{m+1\cdots2m}^\top A\|_2^2 \|W\|_2^2 \le C \|\xi_{m+1\cdots2m}^\top A\|_2^2.
	\]
	By \eqref{eq:88} and the bounds on both terms in its right hand side, we derive that
	\[
	\sum_{j\in [m]} D_j(\xi^\top \tilde{D}(x))^2
	\le \|\xi^\top \tilde{H}\|_2^2,
	\]
	where $\tilde{H} = C(A|A|0)^\top$ is the matrix of dimension $(2m+1) \times m$ obtained from stacking two copies of $A$ one on top of each other on top of one row of zeros at the bottom, all multiplied by a sufficiently large constant $C$. Hence, we can define the pseudo Hessian as the constant function $\tilde{H}(x) = \tilde{H}$.
	
	Lastly, we would like to apply Theorem~\ref{t:newnote}, applying it with the pseudo discrete derivative and Hessian defined above. We would just have to calculate:
	\[
	\|\E_x[\tilde{D}(x)]\|_2^2
	= C \sum_{i=1}^m (\E_x[A_i^\top x + b_i])^2
	+ C \sum_{i=1}^m (\E_x[A_i^\top(x-\tanh(Jx+h))])^2
	+ C \|A\|_F^2.
	\]
	The first summand equals $\|\E[Ax+b]\|_2^2$, while the second equals zero, from the same argument as in Claim~\ref{cla:expectation-zero}.
	This implies that $\|\E[\tilde D(x)]\|_2^2 \le C\|\E[Ax+b]\|_2^2 + C\|A\|_F^2$.
	Next, we bound the terms corresponding to the pseudo Hessian: we have that $\|\tilde{H}\|_2 \le C \|A\|_2$, and $\|\tilde{H}\|_F^2 \le C\|A\|_F^2$. Plugging these in Theorem~\ref{t:newnote} concludes the proof.
\end{proof}

\subsection{Proof of Lemma~\ref{lem:conc-mainpart}}\label{s:conc-mainpart}
Now, we move on to the proof of Lemma~\ref{lem:conc-mainpart}.
The function that we wish to show concentration about
is a second degree polynomial, hence Theorem~\ref{t:note} applies.
However, this Theorem requires the matrix to have $0$ in
the diagonal, which is not necessarily the case for
$A^\top A$. Hence, we need to modify the matrix so that
it is zero-diagonal, obtain the concentration bound
for the modified matrix and then translate the result
in terms of the original matrix. This is done in the
following proof.
\begin{proof}[Proof of Lemma~\ref{lem:conc-mainpart}]
  Denote $p(x) = \|Ax\|_2^2 = x^\top A^\top Ax$. 
  Let $E$ be obtained from $A^\top A$ by zeroing all elements of the diagonal and denote $\tilde p(x) = x^\top Ex$. Note that $\tilde p(x) - p(x)$ is a constant as a function of $x$,
  since $x_i^2 = 1$ for all $i$. This means that it suffices to bound the deviation of $\tilde{p}(x)$. 
  By Lemma~\ref{lem:ising-subsample}, conditioning on $x_{-I_j}$ yields an Ising model with Dobrushin constant $1/2$.
  Hence, we can apply Theorem~\ref{t:note} and get that for any $t \ge 0$,
	\begin{equation}\label{eq:concE}
	\Pr\lp[\lp|\tilde p(x) - \E\lp[\tilde p(x) ~\middle|~ x_{-I_j}\rp]\rp| > t~\middle|~ x_{-I_j}\rp]
	\le \exp\lp(-c \min\lp(\frac{t^2}{\|E\|_F^2+\|\E [Ex | x_{-I_j}]\|_2^2}, \frac{t}{\|E\|_2}\rp)\rp).
	\end{equation}
	Now, we show how $E$ can be replaced with $A^\top A$ in \eqref{eq:concE}.
	First:
	\[
	\lp\|\E\lp[ Ex ~\middle|~ x_{-I_j}\rp]\rp\|_2 \le  \lp\|\E\lp[ A^\top Ax ~\middle|~ x_{-I_j}\rp]\rp\|_2
	+ \lp\|\E\lp[ (E-A^\top A)x ~\middle|~ x_{-I_j}\rp]\rp\|_2,
	\]
	which, using the inequality $(a+b)^2 \le 2a^2 + 2b^2$, implies that
	\begin{align*}
	\|E\|_F^2+\lp\|\E\lp[ Ex ~\middle|~ x_{-I_j}\rp]\rp\|_2^2
	&\le 2 \lp\|E\rp\|_F^2
	+ 2\lp\|\E\lp[ A^\top Ax ~\middle|~ x_{-I_j}\rp]\rp\|_2
	+ 2\lp\|\E\lp[ (E-A^\top A)x ~\middle|~ x_{-I_j}\rp]\rp\|_2^2\\
	&= 2 \lp\|A^\top A \rp\|_F^2 + 2\lp\|\E\lp[ A^\top Ax ~\middle|~ x_{-I_j}\rp]\rp\|_2^2.
	\end{align*}
	In the last equality, we used the fact that $\lp\|\E\lp[ (E-A^\top A)x ~\middle|~ x_{-I_j}\rp]\rp\|_2^2$ is just the sum of the squares
	of the diagonal entries of $A^\top A$, which means
	that together with $\|E\|_F^2$ they add up to 
	$\|A\|_F^2$. 
	Next, notice that $\|E\|_2 \le \|A^\top A\|_2$. To prove this, we note that for all $x \in \mathbb{R}^n$, 
	$$
	x^\top E x = x^\top A^\top A x - \sum_{i \in [n]} (A^\top A)_{ii} \le x^\top A^\top A x.
	$$
	Putting all of this together, we obtain that the right hand side of ~\ref{eq:concE} is bounded by
	\[
	\exp\lp(-c' \min\lp(\frac{t^2}{\lp\|A^\top A \rp\|_F^2 + \lp\|\E\lp[ A^\top Ax ~\middle|~ x_{-I_j}\rp]\rp\|_2^2}, \frac{t}{\|A^\top A\|_2}\rp)\rp).
	\]
	Finally, we want to make this bound depend on $A$ 
	rather than $A^\top A$.
	First,
	\[
	\|A^\top A\|_F^2
	\le \|A\|_F^2 \|A\|_2^2
	\]
	using the well known inequality $\|AB\|_F \le \|A\|_2 \|B\|_F$.
	
	Next, we have:
	\begin{equation}
	\lp\|\E\lp[ A^\top Ax ~\middle|~ x_{-I_j}\rp]\rp\|_2^2
	= \lp\|A^\top \E\lp[ Ax ~\middle|~ x_{-I_j}\rp]\rp\|_2^2
	\le \|A\|_2^2 \lp\|\E\lp[ Ax ~\middle|~ x_{-I_j}\rp]\rp\|_2^2
	\end{equation}
	Lastly,
	\[ 
	\|A^\top A\|_2
	= \|A\|_2^2.
	\]
	This concludes the proof.
\end{proof}

\printbibliography

\end{document}